\documentclass[12pt]{article}
\usepackage{e-jc}

\usepackage[english]{babel}
\usepackage{amsthm,amsmath,amssymb}
\usepackage{graphicx}
\usepackage[small,bf,hang]{caption} 

\usepackage{algorithm}                   
\usepackage{algorithmic}

\usepackage{epsf} 

\usepackage{hyperref}
\usepackage{enumerate} 

\usepackage{color}

\newtheorem{theorem}{Theorem}[section]
\theoremstyle{plain}

\newtheorem{lemma}[theorem]{Lemma}
\newtheorem{proposition}[theorem]{Proposition}
\newtheorem{corollary}[theorem]{Corollary}
 
\theoremstyle{definition}
\newtheorem{definition}{Definition}

\graphicspath{{figures/}}

\title{\vspace{-3.0ex}Ramsey numbers $R(K_3, G)$ for graphs of order 10}

\author{Gunnar Brinkmann\\
\small Department of Applied Mathematics \& Computer Science\\[-0.8ex]
\small Ghent University\\[-0.8ex] 
\small Krijgslaan 281-S9,\\[-0.8ex] 
\small 9000 Ghent, Belgium\\
\small\tt gunnar.brinkmann@ugent.be\\
\and
Jan Goedgebeur\\
\small Department of Applied Mathematics \& Computer Science\\[-0.8ex]
\small Ghent University\\[-0.8ex] 
\small Krijgslaan 281-S9,\\[-0.8ex] 
\small 9000 Ghent, Belgium\\
\small\tt jan.goedgebeur@ugent.be\\
\and
Jan-Christoph Schlage-Puchta\\
\small Department of Mathematics\\[-0.8ex]
\small Ghent University\\[-0.8ex] 
\small Krijgslaan 281-S22,\\[-0.8ex] 
\small 9000 Ghent, Belgium\\
\small\tt jcsp@cage.ugent.be
}

\date{\dateline{XX}{XX}}

\begin{document}

\maketitle

\begin{abstract}
In this article we give the generalized triangle Ramsey numbers
$R(K_3, G)$ of 12~005~158 of the 12~005~168 graphs of order
10. There are 10 graphs remaining for which we could not determine
the Ramsey number. Most likely these graphs need approaches focusing
on each individual graph in order to determine their triangle Ramsey number.
The results were obtained
by combining new computational and theoretical results.  We also
describe an optimized algorithm for the generation of all 
maximal triangle-free graphs and triangle Ramsey
graphs.  All Ramsey numbers up to 30 were computed by our
implementation of this algorithm. We also prove some theoretical
results that are applied to determine several triangle Ramsey
numbers larger than 30.
As not only the number of graphs is increasing very fast, but also
the difficulty to determine Ramsey numbers, we consider it very likely
that the table of all triangle Ramsey numbers for graphs of order
10 is the last complete table that can possibly be determined for a 
very long time.

\end{abstract}


\section{Introduction}
The Ramsey number $R(G,H)$ of two graphs $G$ and $H$ is the smallest
integer $r$ such that every assignment of two colours (e.g.\ red and blue)
to the edges of $K_r$ gives $G$ as a red subgraph or $H$ as a blue
subgraph. Or equivalently $R(G,H)$ is the smallest integer $r$ such
that every graph $F$ with at least $r$ vertices contains $G$ as a
subgraph, or its complement $F^c$ contains $H$ as a subgraph. A graph
$F$ is a \textit{Ramsey graph} for a pair of graphs $(G,H)$ if $F$
does not contain $G$ as a subgraph and its complement $F^c$ does not
contain $H$ as a subgraph.

The existence of $R(G,H)$ follows from Ramsey's
theorem~\cite{ramsey_30} from 1930. The classical Ramsey numbers
(where both $G$ and $H$ are complete graphs) are known to be extremely
difficult to determine. It is even difficult to obtain narrow bounds when $H$ or $G$
have a large order. Therefore only few exact results are known. The
last exact result was obtained by McKay and
Radziszowski~\cite{mckay_radz_95} in 1995 when they proved that
$R(K_4, K_5) = 25$.

For a good overview of the results and bounds of Ramsey numbers which
are currently known, we refer the reader to Radziszowski's 
dynamic survey~\cite{staszek_ds}.

In this article, we focus on triangle Ramsey numbers, that is Ramsey
numbers $R(G,H)$ where $G=K_3$. When we speak about Ramsey numbers 
or Ramsey graphs in
the remainder of this article, we always mean triangle Ramsey numbers
resp. triangle Ramsey graphs.

Already in 1980 all triangle Ramsey numbers for graphs of order 6 were
determined by Faudree, Rousseau and Schelp~\cite{faudree_80}. In 1993
the Ramsey numbers for connected graphs of order 7 were computed by
Jin Xia~\cite{xia_93}. Unfortunately some of his results turned out to
be incorrect. These were later corrected by
Brinkmann~\cite{brinkmann_98} who determined all triangle Ramsey
numbers for connected graphs of order 7 and 8 by using computer
programs. Independently, Schelten and Schiermeyer also determined
Ramsey numbers of graphs of order 7 by hand~\cite{schelten_97,
  schelten_98}.

In~\cite{brandt_98} all triangle Ramsey
numbers for connected graphs of order 9 and all Ramsey numbers $R(K_3,
G) \le 24$ for connected graphs of order 10 are given. For 2001 graphs
of order $10$ the Ramsey number remained open.

We used the same basic approach for the generation of maximal triangle-free
graphs (in short, \textit{mtf graphs}) that was already used in~\cite{brandt_98}, but some observations
about the structure of mtf graphs in~\cite{brandt_98} made it possible to 
improve the basic algorithm. It was observed that an astonishingly large ratio
of small mtf graphs had an automorphism group of size 2 caused by
two vertices with identical neighborhoods. 

We implemented the optimized algorithm for the
generation of mtf graphs and also added improved routines for the
restriction to Ramsey graphs. Using this program we independently
verified the results from~\cite{brandt_98} and determined all Ramsey
numbers $R(K_3, G)$ up to 30 for connected graphs $G$ of order 10. The
improved algorithm is described in
section~\ref{section_algorithm}. 
Next to these computational results,
we also proved some theoretical results that allowed to 
determine the Ramsey number of several graphs with Ramsey
number larger than 30. Combining these computational and theoretical results, only
10 graphs with 10 vertices are left for which the triangle Ramsey number is unknown.
We hope that other researchers will help to complete this list of triangle
Ramsey numbers which will then most likely be the last complete list of
triangle Ramsey numbers for a very long time.

As human intuition and insight is often based on examples, data about small
graphs -- like complete lists of 
Ramsey numbers -- can help to discover mathematical theorems,
suggest conjectures and give
insight into the structure of mathematical problems. An example is given in
\cite{alphadef}, where a large amount of computational data about $alpha$-labelings
gave insight into the structure of $alpha$-labelings of trees 
so that new theorems could be proven and some
unexpexted conjectures were suggested.
In order not to be mislead by too small examples, it is important to have as
much data as possible to develop a good intuition, as e.g.\ the following example shows: 
If $K_n-(m\cdot e)$ denotes the graph obtained by removing
$m$ disjoint edges from $K_n$, then 
the previously existing
lists of triangle Ramsey numbers for graphs of order at most $9$ have the property that  
for fixed $n$ the value $R(K_3,K_n-(m\cdot e))$ is the same for all $2\le m \le n/2$.
This may be considered as a hint that it could be true in general, but
the list in this article shows that for $n=10$ this equation does not hold.


\section{General results}
\label{section_general_results}

In this section we prove some general results on Ramsey numbers of the
form $R(K_3, G)$, where $G$ is close to a complete graph. Let $T_{s+}$ denote the tree obtained from $K_{1,s}$ by adding an extra vertex and connecting it
to a vertex with degree 1 in $K_{1,s}$. We write 
$\Delta_s$ for the graph obtained from $K_{1,s}$ by adding
one edge connecting two vertices with degree~1 in $K_{1,s}$, 
and $D_{s,t}$ for the
double star obtained from the disjoint union of $K_{1,s}$ and
$K_{1,t}$ by joining the vertices with degrees $s$ and $t$.
We denote the set of vertices of a graph $G$ by $V(G)$ and the set of edges by $E(G)$. We denote the neighbourhood of a vertex $v\in V(G)$ by $N(v)$.

The first result is a slight modification of Theorem~1 from \cite{brandt_98}. We give the
proof here, as we shall use the same argument repeatedly.

\begin{lemma}\label{lem:oldlemma}
Let $M$ be a triangle-free graph on $r$ vertices, such that $M^c$ contains
$K_{n-1}$, and let $s$ be an integer satisfying $1\leq s< n$ and
$(r-n)(s+1)>(n-1)(n-2)$. Then $M^c$ contains $K_n-K_{1,s}$.
\end{lemma}

\begin{proof}
Suppose otherwise. If there exists a vertex with degree at least $n$, then
$M^c$ contains $K_n$, since $M$ is triangle-free. Now assume that
there exists a vertex $v$ with degree $n-1$. Then the neighbourhood of
$v$ consists of an anti-clique
of size $n-1$, that is, deleting this vertex we obtain a graph $M'$
with $r-1$ vertices, which contains an anti-clique of size $n-1$,
such that each vertex in this anti-clique has degree $\leq n-2$. If
there is no vertex with degree $n-1$, we delete an arbitrary vertex not
contained in some specified anti-clique of size $n-1$. In each case we
obtain an induced subgraph $M'$ of $M$ with $r-1$
vertices, which contains an anti-clique $A$ of size $n-1$, such that
every vertex in this anti-clique has degree at most $n-2$. From each
vertex in $V\setminus A$ there are at least 
$s+1$ edges connecting this vertex with an element of $A$, for
otherwise adding this vertex to $A$ we would obtain a supergraph of
$K_n-K_{1,s}$ in $M^c$. Hence, there exists a vertex $v\in A$, which has degree
at least $\frac{(r-n)(s+1)}{n-1}$. By assumption this quantity is larger
than $n-2$, contradicting the choice of $M'$.
Hence our claim follows.
\end{proof}

\begin{proposition}
Let $r, n, s$ be integers such that $1\leq s< n$ and $(r-n)(s+1)>(n-1)(n-2)$. Then for every
triangle-free graph $M$ on $r$ vertices, such that $M^c$ contains
$K_n-\Delta_{s+1}$, we have that $M^c$ contains $K_n-T_{s+}$.
\end{proposition}

\begin{proof}
Let $M$ be a counterexample. Since $M$ is triangle-free, at least one of the edges
in the triangle missing in $K_n-\Delta_{s+1}$ must be present in the subgraph of $M^c$
containing $K_n-\Delta_{s+1}$. So 
$M^c$ contains $K_n-T_{s+}$ or $K_n-K_{1,s+1}$. In the former case our
claim follows immediatelly, while in the latter we have that $M^c$
contains $K_{n-1}$, and by Lemma~\ref{lem:oldlemma} we obtain that $M^c$ contains
$K_n-K_{1,s}$ and therefore also
$K_n-T_{s+}$.
\end{proof}

\begin{proposition}
Suppose that $r,n,s$ satisfy $r\geq R(K_3, K_{n-1}-e)$, $(r-n+1)s>
(n-2)(n-3)$ and $(r-n)(s+1)>(n-1)(n-2)$. Then $r\geq R(K_3, K_n-T_{s+})$. 
\end{proposition}

\begin{proof}
Let $M$ be a triangle-free graph on $r$ vertices. 
If $M^c$ contains $K_{n-1}$ our
claim follows from Lemma~\ref{lem:oldlemma}.
So we may assume that every vertex in
$M$ has degree at most $n-2$ and that $M^c$ does not contain
$K_{n-1}$. From $r\geq R(K_3, K_{n-1}-e)$ it now follows that $M^c$
contains $K_{n-1}-e$ as an induced subgraph. Let $A$ be a set of
$n-1$ vertices of $M$, such that the edge $(v,w)$ is the unique edge 
between vertices in $A$. As at most $(n-2)(n-3)$ edges go from
$A-\{v,w\}$ to $V\setminus A$ and
$(r-(n-1))s>(n-2)(n-3)$, there exists a vertex $x\in V\setminus A$
which is connected to at most $s-1$ elements of 
$A-\{v,w\}$. If $x$ was connected to both $v$ and $w$, then $M$ would
contain a triangle, hence the induced subgraph on $A\cup\{x\}$ is
contained in $T_{s+}$. We conclude that $M^c$ contains $K_n-T_{s+}$, and
our claim follows.
\end{proof}

\begin{proposition}
\label{Prop:doublestar}
Let $n$ be an integer, $M$ be a triangle-free
graph, such that $M^c$ contains $K_n$. Assume further that $|V|\geq
3n+4$. Then $M^c$ contains $K_{n+2}-D_{m,m}$, where
$m=\lfloor\frac{n-1}{2}\rfloor$.
\end{proposition}
\begin{proof}
Assume the statement was false, and that $M$ was a counterexample.

Fix an anti-clique of size $n$ and call it $A$. We now partition $V-A$
into four sets: $L$, the set of large vertices, which are connected
with more than $n/2$ elements of $A$; $H$, the set of medium vertices,
which are connected to exactly $n/2$ vertices of $A$; $S$, the set
of small vertices, which are connected to at least 1, but at most $m$
vertices of $A$; and $X$, the set of exceptional vertices, which are
not connected to $A$. Note that medium vertices can only
exist for even $n$.

If there are two different vertices $v, w\in X$, 
then $A\cup\{v,w\}$ contains at most one edge, and our
claim follows. Hence we may assume that $|X|\leq 1$ and therefore
$|L\cup H \cup S|\ge 2n+3$. We will prove that the graph induced
by $L\cup H \cup S$ is bipartite and therefore contains an
anti clique of size $n+2$ contradicting the assumption.

If 2 vertices $v\not= w \in S$ were adjacent, $N(v)\cap A$
and $N(w)\cap A$ would be disjoint as $M$ is triangle-free.
But then $A\cup\{v,w\}$ would induce a supergraph of $K_{n+2}-D_{m,m}$
in $M^c$.

Vertices in $L$ can not be adjacent with vertices in $L\cup H$
as the two endpoints of the edge would have to have a common
neighbour in $A$.

So cycles in the graph induced by $L\cup H \cup S$ contain
either only vertices from $H$ or can be split into parts
by vertices from $S$. 

If two vertices $v\not= w \in H$ are adjacent,
due to $M$ being triangle-free and each vertex having
$n/2$ neighbours in $A$, we have $N(w)\cap A= A \setminus N(v)$.

If $v_1,\dots,v_k=v_1$ is a cycle containing only vertices from
$H$, then we have $N(v_{i+1})\cap A= A \setminus N(A)$.
So $v_k=v_1$ implies that $k$ must be even.

If a cycle contains elements from $S$, then each part between two
subsequent vertices $v,w$ (which can be the same) 
from $S$ contains an even number of edges:
If $v$ is followed by a vertex from $L$, then the next vertex is $w$ --
so the segment contains two edges. If there was a path
$v, x_1, x_2, \ldots, x_k, w$ with $k>0$ even, then $N(x_{1})\cap A= A \setminus N(x_{k})$.
As $M$ is triangle-free and $v, x_1$ are adjacent we have $N(v)\subset N(x_{k})\cap A$
and analogously $N(w)\subset N(x_{1})\cap A$ -- so the neighbourhoods of $v$
and $w$ are disjoint in $A$, so that $A\cup\{v,w\}$ would again
induce a supergraph of $K_{n+2}-D_{m,m}$ in $M^c$. So $k$ must be odd
and the segment contains an even number of edges. This implies that each cycle
consists of a certain number of segments of even length and is therefore even --
proving that the graph induced by $L\cup H \cup S$ is bipartite.
\end{proof}

\begin{proposition}
Let $n, r, s, t$ be integers, such that $s+t+2\le n$, $s\geq t>0$,
$(r-n)(s+1)>(n-1)(n-2)$, and $(r-(n-1))(s+1)>(n+2(s-t)-2)(n-3)$. Then
every graph on $r$ vertices which contains $K_{n-1}-e$ contains
$K_n-K_{1,s}-K_{1,t}$.
\end{proposition}

\begin{proof}
Assume our statement is false, and let $M$ be a counterexample. If
$M^c$ contains $K_{n-1}$, then Lemma~\ref{lem:oldlemma} shows that $M^c$
contains $K_n-K_{1,s}$, and we are done. This implies also that
all vertices in $M$ have degree at most $n-2$. Let
$A$ be a set of $n-1$ vertices, such that among the vertices of $A$ there
is a single edge $(v,w)$. Put $X=N(v)-\{w\}$, $Y=N(w)-\{v\}$. As $M$
is triangle-free, the sets
$X, Y$ are disjoint anti-cliques and as the degrees of $v,W$ are at most 
$n-2$, we have $|X|, |Y|\leq n-3$. Each element $z$ of $|X|, |Y|$ has at least
two neighbours in $A$, as otherwise $A\cup \{z\}$ would induce
a $K_n-K_{1,2}$ in $M^c$.

Suppose that $X$ contains elements
$x, x'$, such that $x$ is connected with at most $t$ elements of
$A-\{v,w\}$, and $x'$ is connected with at most $s$ elements of
$A-\{v,w\}$. Then all edges in $A\cup\{x,x'\}-\{v\}$ are between
$\{x,x'\}$ and elements of $A-\{v,w\}$, and we obtain
$K_n-K_{1,s}-K_{1,t}$ in $M^c$. 

Hence either each element in $X$ is
connected with at least $t+1$ elements in $A-\{v,w\}$, or all but at most
one element of $X$ is connected with at least $s+1$ elements of
$A-\{v,w\}$ and the remaining element is connected with at least one
element of $A-\{v,w\}$. The same argument applies for $Y$. An element $x$ of
$V-(A\cup X\cup Y)$ is not connected to $v$ or $w$, hence if this element is
connected with $p\le s$ elements of $A-\{v,w\}$, then $A\cup\{v\}$
forms a $K_n-K_{1,p}-e$ in the complement that has $K_n-K_{1,s}-K_{1,t}$
as a subgraph. We conclude that each vertex in $V-(A\cup X\cup Y)$
has at least $s+1$ neighbours in $A-\{v,w\}$. 

Counting the edges between $A-\{v,w\}$ and $V-A$ we get a lower bound of

\begin{multline*}
\min\big((t+1)|X|, (s+1)(|X|-1)+1\big) + 
\min\big((t+1)|Y|,(s+1)(|Y|-1)+1\big)\\
 + (s+1)(r-(n-1)-|X|-|Y|) 
\end{multline*}

As a function of
$|X|$ and $|Y|$ this expression is non-increasing, hence this quantity
has its minimum for $|X|=|Y|=n-3$, which gives

\[
2\min\big((t+1)(n-3), (s+1)(n-4)+1\big)
 + (s+1)(r-3n+7).
\]

On the other hand
each vertex in $A-\{v,w\}$ has degree at most $n-2$, giving an upper
bound of $(n-2)(n-3)$. 
Now for $s=t$ we obtain $(s+1)(r-n-1)+2\leq
(n-2)(n-3)$, which contradicts our assumption $(r-n)(s+1)>(n-1)(n-2)$. 
If $t<s$, we obtain
$(s+1)(r-n+1) - 2(s-t)(n-3)\leq (n-2)(n-3)$, which contradicts
our assumption $(r-(n-1))(s+1)>(n+2(s-t)-2)(n-3)$.
Hence, in both cases our claim follows.
\end{proof}

Applying these results to the case of graphs on 10 vertices, we obtain
the following:

\begin{corollary}
\begin{enumerate}
\item For $9\ge s\geq 2$ we have $R(K_3, K_{10}-K_{1,s})=36$; 
\item For $8\ge s\geq 3$ we have $R(K_3, K_{10}-T_{s+}) = R(K_3,
  K_{10}-\Delta_{s+1})=R(K_3, K_{10}-K_{1,s}-e) = 31$;
\item We have $R(K_3, K_{10}-D_{3,3})=28$.
\end{enumerate}
\end{corollary}
\begin{proof}
The upper bounds follow from the propositions, while the lower bounds
are implied by $R(K_3, K_9)=36$, $R(K_3, K_9-e)=31$, and $R(K_3,
K_8)=28$, respectively.
\end{proof}

\section{The algorithm}
\label{section_algorithm}

A {\em maximal} triangle-free graph (in short, an \textit{mtf graph})
is a triangle-free graph so that the insertion of each
new edge introduces a triangle. For $|V|>2$ this is equivalent to being triangle-free and
having diameter $2$.

As adding edges to a triangle-free graph removes edges from its complement, it is easy to
see that there is a triangle Ramsey graph of order $r$ for some graph $G$ if and only if
there is an mtf graph of order $r$ that is a Ramsey graph for $G$ (in short, an \textit{mtf Ramsey graph}).

In order to prove that $R(K_3, G) = r$, we have to show that:

\begin{itemize}

\item There are no mtf Ramsey graphs for $G$ with $r$ vertices \\(which implies $R(K_3, G) \le  r$). 

\item There is an mtf Ramsey graph for $G$ with $r-1$ vertices \\(which implies $R(K_3, G) > r-1$).

\end{itemize}

Even though only a very small portion of the triangle-free graphs are also
maximal (e.g.\ $0.002\%$ for 13 vertices and  $0.000044\%$ for 16 vertices), 
the number of mtf graphs 
still grows very fast (see Table~\ref{table:mtf_times}). Thus it is not
possible for large $r$ to generate all mtf graphs with $r$ vertices and test if they
are Ramsey graphs for a given $G$. Therefore it is
necessary to include the restriction to Ramsey graphs already in the 
generation process.

In section~\ref{section_generate_mtf} we describe an algorithm for
the generation of all non-isomorphic mtf graphs. This algorithm follows the same lines
as the algorithm in \cite{brandt_00} but uses some structural information
obtained from \cite{brandt_00} to speed up the generation. In
section~\ref{section_generate_ramsey} we describe how we extended this
algorithm to generate only mtf Ramsey graphs for a given graph $G$.
In section~\ref{section_compute_ramsey} we
describe how we used the generator for Ramsey graphs to determine the
Ramsey numbers $R(K_3, G)$. The main difference to the approach described in
\cite{brandt_98} is that the approach used here is optimized for small lists
of graphs with larger Ramsey numbers instead of large lists with comparatively
small Ramsey numbers like the approach in \cite{brandt_98}.

\subsection{Generation of maximal triangle-free graphs}
\label{section_generate_mtf}

Mtf graphs with $n+1$ vertices are generated from mtf graphs with $n$ vertices
using the same construction method as in \cite{brandt_00} but different
isomorphism rejection routines. 
To describe the construction, we
first introduce the concept of \textit{good dominating sets}.

\begin{definition}
$S \subseteq V(G)$ is a dominating set of $G$ if $S \cup \{N(s) \ | \ s \in S\} = V(G)$.

A dominating set $S$ of an mtf graph $G$ is \textit{good} if after removing all edges with both
endpoints in $S$ for every $s \in S$ and $v \in V(G) \setminus S$,
the distance from $s$ to $v$ is at most two.
\end{definition}

The basic construction operation removes all edges between 
vertices of a good dominating set $S$ and connects all vertices of $S$
to a new vertex $v$. It is easy to see that this is a recursive structure
for the class of all mtf graphs~\cite{brandt_00}.

In~\cite{brandt_00} it was observed that a surprisingly large number of mtf graphs
had automorphism groups of size 2. This was caused by two vertices with identical
neighbourhoods. We exploit this observation to improve the efficiency of the
isomorphism rejection routines. To this end we distinguish between 3 types of
good dominating sets.

\begin{enumerate}
\item[type 0:] A set $S=N(v)$ for some $v\in V$. Note that in an mtf graph for each vertex $v$ the set $N(v)$
is a good dominating set without internal edges.

\item[type 1:] A good dominating set $S$ without internal edges, but $S\not= N(v) \,\forall v \in V$.

\item[type 2:] A good dominating set $S$ with internal edges.
\end{enumerate}

We call construction operations also \textit{expansions} and the
inverse operations \textit{reductions} and will also talk about
reductions or expansions of type 0, 1 and 2 if the good dominating sets involved are
of this type.
If $G'$ is
obtained from $G$ by an expansion, we call $G'$ the child of
$G$ and $G$ the parent of $G'$.


We use the canonical construction path method~\cite{mckay_98} to make
sure that only pairwise non-isomorphic mtf graphs are generated. 
Two reductions of mtf graphs $G$ and $G'$ (which may be identical)
are called equivalent if there is an isomorphism from $G$ to $G'$
mapping the vertices that are removed onto each other and inducing
an isomorphism of the reduced graphs.
In order
to use this method, we first have to define which of the various possible
reductions of an mtf graph $G$ to a smaller mtf graph is the \textit{canonical
  reduction} of $G$. This canonical reduction must be uniquely determined up
to equivalence. We call the graph obtained by applying the
canonical reduction to $G$ the \textit{canonical parent} of
$G$ and an expansion that is the inverse of a canonical reduction
a \textit{canonical expansion}. 

Furthermore, we also define an equivalence relation on the set of
possible expansions of a graph $G$. Note that the expansions are
uniquely determined by the good dominating set $S$ to which they are
applied. Therefore we define two expansions of $G$ to be equivalent if
and only if there is an automorphism of $G$ mapping the two good
dominating sets onto each other. 

The two rules of the canonical construction path method are:

\begin{enumerate}

\item[(a)] Only accept a graph if it was constructed by a canonical expansion.

\item[(b)] For every graph $G$ to which construction operations are
  applied, perform exactly one expansion from each equivalence class of
  expansions of $G$.

\end{enumerate}

If we start with $K_1$ and recursively apply these rules to each graph
until the output size is reached, exactly one graph of each isomorphism class of mtf graphs is generated. We refer the reader to \cite{brandt_00} for a proof. The coarse structure of the algorithm is given as pseudocode in Algorithm~\ref{algo:pseudocode_mtf}.

\begin{algorithm}[h]
\caption{Construct(mtf graph $G$)}
  \begin{algorithmic}
  \label{algo:pseudocode_mtf}
	\IF{$G$ has the desired number of vertices}  
  		\STATE output $G$
  	\ELSE
  		\STATE find expansions
  		\STATE compute classes of equivalent expansions
  		\FOR{each equivalence class}
  			\STATE choose one expansion $X$
			\STATE perform expansion $X$
			\IF{expansion is canonical}  
				\STATE Construct(expanded mtf graph)
			\ENDIF
                 \STATE perform reduction $X^{-1}$
		\ENDFOR
  	\ENDIF
  \end{algorithmic}
\end{algorithm}

For deciding whether or not a reduction is canonical, we use a two step strategy.
First we decide which vertex should be removed by the canonical reduction. In case the
graph is not an mtf graph after the removal of this vertex, we determine the canonical way to insert edges.
A 5-tuple 
$t(v)=(x_0(v),\dots ,x_4(v))$ represents the vertex $v$ involved
in the reduction in such a way that
two vertices have the same 5-tuple if and only if they are in the same
orbit of the automorphism group. 
The canonical reduction will be a reduction using the vertex
with lexicographically smallest 5-tuple.

The first entry $x_0(v)$ is the type of the neighbourhood of $v$ in the reduced graph.
The most expensive part in computing the canonical reduction is the
computation of how edges have to be inserted between the former 
neighbours of the removed vertex. If the graph has vertices with identical
neighbourhood, a reduction with $x_0=0$ is always possible and 
other reductions do not have to be considered in order to find the one
with minimal 5-tuple. In case there are exactly two vertices with identical
neighbourhood, the canonical reduction is even found after this step:
no matter how the remaining entries of the 5-tuple are defined, removing one of these two vertices is the canonical reduction as they
are the only ones with minimal value for $x_0$. Furthermore there is an automorphism
exchanging the vertices and fixing the rest, so the two reductions are equivalent
and both are canonical.

The way the remaining values are chosen is the result of a lot of performance tests
comparing different choices.
The value of $x_1(v)$ is the degree $-deg(v)$ of the vertex $v$ that is to be removed
in case $x_0(v)\in \{0,1\}$ and $deg(v)$ in case $x_0(v)=2$. Furthermore
$x_2(v)=-\sum_{w\in N(v)}deg(w)$ and $x_3(v)$ can be described as
$-\sum_{w\in N(v)}|V|^{deg(w)}$. In the program $x_3$ is in fact implemented as a 
sorted string
of degrees, but it results in the same ordering.

We call a vertex $v$ \textit{eligible} for position $j$ if it is among the vertices
for which \\ $(x_0(v),...,x_{j-1}(v))$ is minimal among all possible reductions.
For for step (a) of the canonical construction path method we do not have to find the canonical reduction, but only have to
determine whether the last expansion producing vertex $w$ is canonical. Therefore each $x_i$ is only computed if the vertex $w$ is
still eligible for position $i$ and only for vertices which are eligible for position $i$. 
If $x_0(w)\in \{0,1\}$ and 
$w$ is the only vertex eligible for position $i$, 
we know that the expansion was canonical. If $x_0(w)=2$, we still have to determine
whether the edges that have been removed are equivalent to 
the edges that would be inserted for a canonical reduction.

If there are also other vertices eligible for position $4$, 
we canonically label the graph
$G$ using the program \textit{nauty}~\cite{mckay_81} and define $x_4(v)$
to be the negative of the largest label in the canonical labelling of $G$ of a vertex
which is in the same orbit of the automorphism group of $G$ as
$v$. The discriminating power of $x_0,...,x_3$ is usually enough to
decide whether or not a reduction is canonical. For example for generating all mtf graphs with $n=20$ vertices, the more
expensive computation of $x_4$ is only required in $7.8$\% of the
cases. This fraction is decreasing with the number of vertices
to e.g.\ $6.2$\% for $n=22$. 
After computing $x_4$ the vertex in the canonical reduction is uniquely defined up to
isomorphism.

In case $x_0=2$ the canonical reduction is not completely determined
by the vertex $v$ which is removed by the canonical reduction as there can be multiple 
ways to insert the edges in the former neighbourhood of $v$.
In this case we use the same method as in ~\cite{brandt_00},
which is essentially a canonical choice of a set of edges that can be inserted
which gives priority to sets of small size.
This part hardly has any impact on the time consumption of the program.
For generating mtf graphs with $n=18$ vertices, only about $2.5$\% of the time is spent on the routines dealing with this part and already for $n=20$ this decreases to $1.5$\%. 
Therefore we decided not to develop any improvements for this part and refer the reader to~\cite{brandt_00} 
for details.


The priority of the operations expressed in the 5-tuple allows look-aheads for deciding whether or not
an expansion can be canonical before actually performing it. This is also an advantage when
constructing good dominating sets for expansion as it often allows to reduce
the number of sets that have to be constructed.

A vertex which has the same neighbourhood as another vertex is called a \textit{double vertex}. An mtf graph with double vertices can be reduced by a reduction of type 0.
If two vertices have the same neighbourhood in an mtf graph, each good dominating set without internal
edges either contains both vertices or none, so after 
an operation of type 0 or 1 the vertices still have identical neighbourhoods allowing
a reduction of type 0.
So if a graph $G$ contains a reduction of type 0, we do not have to apply expansions of type 1. Furthermore we only have to apply expansions of type 0 to neighbourhoods of vertices 
$v$ of $G$ for which $deg(v)$
is at least as large as the degree of the canonical double vertex in $G$,
otherwise the new vertex will not have the maximal value of $x_1$. If $G$
did not contain any double vertices, we have to apply operations
of type 0 to the neighbourhoods of all vertices.

After a canonical operation of type 2 no reductions of type 0 are
possible, so we only have to apply operations of type 2 that make sure
that afterwards no vertices with identical neighbourhoods exist.  Therefore the good dominating sets to which an operation of type 2 is
applied must contain at least one vertex from the neighbourhood of
each double vertex. Since if no vertex of the common neighbourhood of a pair
of double vertices is included, both vertices must be contained in the
dominating set themselves.  But then they would still have identical
neighbourhoods after the operation. Each good
dominating set must also contain a vertex from each set of vertices with identical
neighbourhood. In the program we use
this in its strongest form only if there is just one common
neighbourhood, else we use a weaker form. This is not a problem for
the efficiency as there is usually only one common
neighbourhood. 

If a graph has at least 3 vertices which have the same neighbourhood,
every graph obtained by applying an expansion of type 2 to $G$ has a
reduction of type 0. Therefore we do not have to apply expansions of
type 2 to this kind of graphs.

Due to the choice of $x_1$ the degree of the vertex to be removed
is minimal for canonical reductions of type 2. 
If we apply an operation of type 2 to a good dominating set
$S$, the new vertex $v$ will have degree $|S|$. If the
minimum degree of a graph is $m$, we only have to apply
operations of type 2 to good dominating sets of size at most $m$ (or size
$m+1$ if the good dominating set contains all vertices of minimum
degree). 



Recall that we also have to compute the equivalence classes of
expansions of a graph in order to comply with rule (b) of the
canonical construction path method. We use \textit{nauty} to compute
the automorphism group of the graph and then compute the orbits of
good dominating sets using the generators of the group. In case
we know that only an operation of type 0 can be canonical 
we actually compute the orbits of vertices representing the
good dominating sets formed by their neighbourhoods.

In some cases we do not have to call \textit{nauty} to compute the automorphism
group. For example if $G$ has a trivial automorphism group and we apply
an operation of type 0 by inserting a vertex $v'$ with the same neighbourhood as
 $v$, the expanded graph $G'$ will have an
automorphism group of size 2 generated by the automorphism exchanging $v$ and $v'$
and fixing all other vertices.

\subsection*{Testing and results}

We used our program to generate all mtf graphs up to 23 vertices. The
number of graphs generated were in complete agreement with the numbers
obtained by running the program from Brandt et al.~\cite{brandt_00}
(which is called \textit{MTF}). The graph counts, running times and a
comparison with \textit{MTF} are given in Table~\ref{table:mtf_times}. Our
program is called \textit{triangleramsey}. Both programs were compiled
by gcc and the timings were performed on an Intel Xeon L5520 CPU at
2.27 GHz. The timings for $|V(G)| \ge 20$ include a small overhead due
to parallelisation.

Table~\ref{table:mtf_operations} gives an overview how many graphs
are constructed by canonical operations of the different types. This
table shows that operations of type 2 are by far the least common
canonical operations.


\begin{table}[h]
\begin{center}
\begin{tabular}{| c || r | r | r | c |}
\hline 
$|V(G)|$ & number of graphs & MTF (s) & triangleramsey (s) & speedup\\
\hline 
17  &  164 796  &  4.0  &  0.8  &  5.00\\
18  &  1 337 848  &  30.5  &  6.2  &  4.92\\
19  &  13 734 745  &  315  &  67  &  4.70\\
20  &  178 587 364  &  4 390  &  972  &  4.52\\
21  &  2 911 304 940  &  75 331  &  17 109  &  4.40\\
22  &  58 919 069 858  &  1 590 073  &  373 417  &  4.26\\
23  &  1 474 647 067 521  &  40 895 299  &  10 431 362  &  3.92\\
\hline
\end{tabular}
\end{center}

\caption{Counts and generation times for mtf graphs.}

\label{table:mtf_times}
\end{table}

\begin{table}[h]
\begin{center}
\begin{tabular}{| c || r | r | r | r |}
\hline
number & number  & num. generated & num. generated  & num. generated\\
of  & of        & by an operation      & by an operation  & by an operation  \\
 vertices      &  mtf graphs           & of type 0            & of type 1     & of type 2   \\
\hline
4 & 2 & 2 & 0 & 0\\
5 & 3 & 2 & 0 & 1\\
6 & 4 & 4 & 0 & 0\\
7 & 6 & 6 & 0 & 0\\
8 & 10 & 9 & 0 & 1\\
9 & 16 & 15 & 0 & 1\\
10 & 31 & 29 & 1 & 1\\
11 & 61 & 57 & 3 & 1\\
12 & 147 & 139 & 4 & 4\\
13 & 392 & 368 & 15 & 9\\
14 & 1 274 & 1 183 & 75 & 16\\
15 & 5 036 & 4 595 & 391 & 50\\
16 & 25 617 & 22 889 & 2 420 & 308\\
17 & 164 796 & 142 718 & 19 577 & 2 501\\
18 & 1 337 848 & 1 105 394 & 213 743 & 18 711\\
19 & 13 734 745 & 10 674 672 & 2 855 176 & 204 897\\
20 & 178 587 364 & 129 333 325 & 46 244 514 & 3 009 525\\
\hline
\end{tabular}
\end{center}

\caption{The number of mtf graphs which were generated by operations of each type.}

\label{table:mtf_operations}
\end{table}


\subsection{Generation of Ramsey graphs}
\label{section_generate_ramsey}

The construction operations for mtf graphs never add edges between
vertices of the parent. So if $G$ is contained in the
complement of an mtf graph $M$, $G$ will also be contained in the
complement of all descendants of $M$. Thus if
$M$ is not a Ramsey graph for $G$, its descendants also won't be
Ramsey graphs. So we can prune the generation process.

The same pruning was already used in~\cite{brandt_98}, but as
the graphs with 10 vertices whose Ramsey number could not be determined in~\cite{brandt_98} 
are all very dense, we mainly optimized our algorithm for this kind of graphs
and will describe these optimisations here.

\newpage
For a graph $G$ and an mtf graph $M$ the following criteria are equivalent:

\begin{description}
\item[(i)] $G$ is subgraph of $M^c$ 
\item[(ii)] $M$ contains a spanning subgraph of $G^c$ as an induced subgraph
\end{description}

If $G$ is dense, $G^c$ has relatively few edges and therefore
it is easier to test (ii) instead of (i) in this case.

By just applying this simple algorithm, even with the faster generator we were not able to go much
further than the results in~\cite{brandt_98}. Therefore we
designed and applied several optimisations specifically for dense
graphs. These optimisations are crucial for the
efficiency of the algorithm.

The bottleneck of the algorithm is the procedure which tests if the
generated mtf graphs contain a spanning subgraph of $G^c$ as induced
subgraph. This procedure basically constructs all possible sets with
$|V(G)|$ vertices and an upper bound of $|E(G^c)|$ on the number of edges
and tests for each set if the graph induced by this
set is a subgraph of $G^c$. Various bounding criteria are used to
avoid the construction of sets which cannot be a subgraph of $G^c$.

If the algorithm as described so far is applied and the order of the
mtf graphs is sufficiently large, by far most of the mtf graphs that
are generated are rejected as they turn out to be no Ramsey graphs for
the testgraph $G$. For example for $G=K_{10}-P_5$ and $|V(M)|=28$
(without other optimisations) approximately 99\% of the mtf graphs
which were generated are no Ramsey graphs (and are thus rejected).  So
most of the tests for making spanning induced subgraphs give a positive
result -- that is: there is an induced subgraph of $M$ that is
a subgraph of $G^c$. We take this into account by first using some
heuristics to try to find a set of vertices which is a spanning
subgraph of $G^c$ quickly. If such a set is found, we can abort the
search.

More specifically: when an mtf graph is rejected because it is not a
Ramsey graph for $G$, we store the set of vertices which induces a
spanning subgraph of $G^c$. For each order~$n$, we
store up to 100 sets of vertices which caused an mtf graph with $n$
vertices to be rejected.  When a graph with $n$ vertices is generated,
we first investigate if one of those 100 sets of vertices induces a
spanning subgraph of $G^c$. Only if this is not the case, we continue
the search.  Experimental results showed that storing 100 sets seemed
to be a good compromise between cost to test if a set induces a spanning subgraph of $G^c$ and the chance to have
success. Without other optimisations this makes the program e.g.\ $5$ times faster
for $G=K_{10}-P_5$ and $|V(M)|=26$. 

The second step in trying to prove that $M$ is not a Ramsey graph
is a greedy heuristic.
We construct various
sets of $|V(G)|$ vertices which have as few neighbours with each other
as possible. These sets are good candidates to induce a subgraph of
$G^c$. Only if none of these sets induces a subgraph of $G^c$, we have
to continue to investigate the graph. This gives an additional speedup of approximately 10\%.

These heuristics allow to find a set of vertices which induces a
spanning subgraph of $G^c$ quickly in about 98\% of the cases. If
these heuristics did not yield such a set of vertices, we start a complete search. 
In about 70\% of
the cases the graphs passing the heuristical search are actually
Ramsey graphs for $G$. The coarse pseudocode of the procedure which tests if an mtf graph $M$ is a Ramsey graph for $G$ is given in Algorithm~\ref{algo:pseudocode_ramsey}.

\begin{algorithm}[H]
\caption{Is\_Ramsey\_graph(mtf graph $M$, testgraph $G$)}
  \begin{algorithmic}
  \label{algo:pseudocode_ramsey}
  	\FOR{each stored set $S$ with $n=|V(M)|$}
		\IF{$S$ induces a spanning subgraph of $G^c$ in $M$}  
			\RETURN $M$ is not a Ramsey graph for $G$
		\ENDIF
	\ENDFOR  
	\STATE construct sets of $|V(G)|$ vertices in a greedy way
	\IF{set found which induces a spanning subgraph of $G^c$ in $M$}  
		\STATE store set
		\RETURN $M$ is not a Ramsey graph for $G$
	\ENDIF
	\STATE construct all possible sets of $|V(G)|$ vertices
	\IF{set found which induces a spanning subgraph of $G^c$ in $M$}  
		\STATE store set
		\RETURN $M$ is not a Ramsey graph for $G$
	\ELSE
		\RETURN $M$ is a Ramsey graph for $G$
	\ENDIF		
  \end{algorithmic}
\end{algorithm}

The construction of all possible sets of $|V(G)|$ vertices can also be
improved. Recall that our algorithm constructs Ramsey graphs from
Ramsey graphs. Therefore if an mtf graph $M$ was constructed by
operations of type 0 or 1 (i.e.\ no edges were removed), we only have
to investigate sets of vertices which contain the new vertex which was
added by the construction. The subgraphs induced by the other sets did
not change and are already proven not to induce a spanning subgraph of
$G^c$.  Moreover if $M$ was constructed by an operation of type 0, we
only have to investigate sets of vertices which contain the new vertex
and all other vertices which have the same neighbourhood as the new
vertex. Since if a set does not contain a vertex $v$ which has the
same neighbourhood as the new vertex, we can swap $v$ and the new
vertex. 

Similarly, if $M$ was constructed by an operation of type 2 and one edge $e$
was removed (say $e=\{v_1,v_2\}$), we only have to investigate sets of
vertices which contain the new vertex or which contain both $v_1$ and
$v_2$. Similar optimisations can also be used when more edges are removed, but this does not speed up the program as in most cases such operations 
turn out to be not canonical. So then the graph is already rejected before it is tested whether or not this graph is a Ramsey graph.

We also avoid constructing mtf graphs that are no Ramsey graphs for $G$. This is of course even better than efficiently rejecting graphs after they are constructed. More specifically, each time a new mtf Ramsey graph $M$ for a graph $G$ was constructed, 
we search and store
\textit{approximating} sets of vertices. We call a set of vertices
\textit{approximating} if it induces a spanning subgraph of
$G^c_\delta$, where $G^c_\delta$ is a graph obtained by removing a
vertex of minimum degree from $G^c$. For all graphs $G$ with 10
vertices whose Ramsey
number could not be determined in~\cite{brandt_98}, 
$G^c$ has minimum degree~0.

If for a graph $M'$ which is constructed from $M$ there is an
approximating set $S$ of $F$ for which no vertex $s \in S$ is a
neighbour of the new vertex $v$, the graph induced by $S \cup
\{v\}$ in $F'$ is a spanning subgraph of $G^c$. So graphs constructed
from $M$ can only be Ramsey graphs if the good dominating set of $M$
contains at least one vertex from each approximating set in $F$. 
On average this optimisation avoids the construction of more
than 90\% of the children. 

Since searching for all approximating sets is expensive, we search for
them during the search for sets of vertices which induce a spanning
subgraph of $G^c$: when a set of $|V(G^c_\delta)|$ vertices was
formed, we test if it is an approximating set and store it if it is
the case.

\subsection{Computing the Ramsey numbers}
\label{section_compute_ramsey}

To determine the Ramsey numbers with our algorithm, we again use the same
basic strategy as Brandt et al.\ used in~\cite{brandt_98}:

Assume we have a list of all graphs $G$ with Ramsey number 
$R(K_3,G) \ge r$. We want to split this list into those with
$R(K_3,G) = r$ and those with $R(K_3,G) > r$.
We have a (possibly empty) list of
MAXGRAPHs. These are graphs which have Ramsey number $r$. We also have
a (possibly empty) list of RAMSEYGRAPHs, which are triangle-free
graphs with $r$ vertices which are (or might be) Ramsey graphs for
some of the remaining graphs.

The procedure to test whether the remaining graphs have Ramsey number 
at most $r$ or at least $r+1$ works as follows:

  \begin{algorithmic}
	\FOR{$k={n\choose 2}$ {\bf downto} $n-1$}
		\FOR{every connected graph $G$ with $k$ edges in the list}
			\IF{$G$ is not contained in any MAXGRAPH}
            	\IF{$G$ is contained in the complement of every RAMSEYGRAPH}
            		\IF{ \textit{triangleramsey} applied to $G$ finds a Ramsey graph of order $r$}
            			\STATE add this Ramsey graph to the list of RAMSEYGRAPHs
            		    \STATE $R(K_3,G) > r$            			
            		\ELSE 
            			\STATE add $G$ to the list of MAXGRAPHs
            			\STATE $R(K_3,G) \le r$
            		\ENDIF
            	\ELSE
            		\STATE $R(K_3,G) > r$
            	\ENDIF
			\ELSE
			   \STATE $R(K_3,G) \le r$
			\ENDIF
		\ENDFOR
	\ENDFOR
  \end{algorithmic}

For large orders of $r$ (i.e.\ $r \ge 26$), the bottleneck of the
procedure is computing individual Ramsey graphs by 
\textit{triangleramsey}. Here we used some additional
optimisations: if $r$ is close to $R(K_3,G)$, there are usually only
very few Ramsey graphs of order $r$ for $G$. Therefore for certain
graphs $G$ where we expected $r$ to be close to $R(K_3,G)$,
we used \textit{triangleramsey} to compute all Ramsey
graphs of order $r$ for $G$, instead of aborting the program as soon as one
Ramsey graph was found. \textit{Triangleramsey} constructs Ramsey graphs from
smaller Ramsey graphs, so in order to construct all Ramsey graphs for
$G$ of order $r+1$, we can start the program from the Ramsey graphs of
order $r$. This avoids redoing the largest part of the work. 
Of course this approach 
only works if there are not too many Ramsey graphs of order $r$ to
be stored. We used this strategy amongst others to
generate all Ramsey graphs of order $28$ for 
$K_{10}-P_5$ and $K_{10}-2P_3$ (where $P_x$ stands for the path with
$x$ vertices). Computing all Ramsey graphs with 28
vertices for $K_{10}-2P_3$ for example, required almost 4 CPU years. This yielded 7
Ramsey graphs and constructing the Ramsey graphs with 29 vertices from
these 7 graphs took less than 2 seconds.

Let $H$ be a subgraph of $G$ for which we know that $R(K_3,H) \ge r$. If
we have the list of all Ramsey graphs of order $r$ for $G$ and none of
these Ramsey graphs is a Ramsey graph for $H$, we know that $R(K_3,H)
= r$. This allowed us amongst others to determine that several
subgraphs of $K_{10}-P_5$ and $K_{10}-2P_3$ have Ramsey number 28.

\subsection{Testing and results}
By using the algorithm described in
section~\ref{section_compute_ramsey} we were able to compute all
Ramsey numbers $R(K_3,G)=r$ for connected graphs of order 10 for which
$r \le 30$ and to determine which Ramsey graphs have Ramsey number
larger than 30. Since \textit{triangleramsey} is more than 20 times
faster than \textit{MTF} for generating triangle Ramsey graphs of large order
$r$, we could only compute Ramsey numbers up to $r=26$ with the
original version of \textit{MTF}. In order to be able to compare our results
for larger $r$, we added several of the optimisations which were
described in section~\ref{section_generate_ramsey} to \textit{MTF}. With this
improved version of \textit{MTF} we were also able to determine all Ramsey
numbers up to $r=30$ and to determine which graphs have Ramsey number
larger than 30. All results were in complete agreement.

In the cases where we generated all Ramsey graphs of order $r$ for a
given testgraph, the results obtained by \textit{triangleramsey} and \textit{MTF} were
also in complete agreement.

For each Ramsey graph which was generated, we also used an independent
program to confirm that the Ramsey graph does not contain $G$ in its complement.

There are 34 graphs $G$ for which $R(K_3,G) > 30$. In
section~\ref{section_general_results}, we proved that $R(K_3,K_{10} -
T_{3+}) = R(K_3,K_{10} - K_{1,3} - e) = 31$. So the graphs with
$R(K_3,G) > 30$ which are a subgraph of $K_{10} - T_{3+}$ or $K_{10} -
K_{1,3} - e$ also have Ramsey number 31. In that section, we also
proved that $R(K_3, K_{10} - K_{1,s}) = 36$ (for $2\le s\le 9$). This
leaves us with 10 graphs with $R(K_3,G) > 30$ for which we were unable
to determine their exact Ramsey number. Also note that if $G^c$
contains a triangle and $H^c$ is the only graph which can be obtained
by removing an edge from that triangle of $G^c$, then $R(K_3,G) =
R(K_3,H)$. Thus among the 10 remaining graphs, $R(K_3,K_{10}-K_3-e) =
R(K_3,K_{10}-P_3-e)$ and $R(K_3,K_{10}-K_4) = R(K_3,K_{10}-K_4^-)$
(where $K_4^-$ stands for $K_4$ with 1 edge removed).  

Table~\ref{table:connected_graphs} contains the number of connected
graphs $G$ of order 10 which have $R(K_3,G)=r$. The triangle Ramsey
numbers of connected graphs of order 10 are given in
section~\ref{section_graph_figures}.

Previously only the Ramsey numbers for disconnected graphs of order at
most 8 were known (see~\cite{brinkmann_98}). We independently verified
these results for order 8 and also determined all Ramsey numbers for
disconnected graphs of order 9 and 10. These results are listed in
Table~\ref{table:disconnected_graphs}. The Ramsey numbers smaller than
28 were obtained computationally. We also independently confirmed the
computational results by using \textit{MTF}. The other Ramsey numbers were
obtained by some simple reasoning. More specifically, if a
disconnected graph is the union of 2 connected graphs $G_1$ and $G_2$
and $R(K_3, G_1) - |V(G_1)| \ge R(K_3, G_2)$, then $R(K_3, G_1 \cup
G_2) = R(K_3, G_1)$.

Unfortunately we did not succeed to compute new values of the functions $f(),g()$ and $h()$
given in Table~2 of \cite{brandt_98}. Nevertheless we could confirm all
the values given in Table~2 of \cite{brandt_98} with the new program.

\subsubsection*{Classical Ramsey numbers}
In 1992 McKay and Zhang~\cite{mckay_92} proved that $R(K_3,K_8) = 28$,
but the complete set of Ramsey graphs with 27 vertices for $K_8$ was
not yet known. Until now 430~215 such graphs were known (most of
these were generated by McKay).

We used \textit{triangleramsey} to compute all maximal triangle-free
Ramsey graphs with 27 vertices for $K_8$. This yielded 21~798 mtf
graphs. We also independently generated these Ramsey graphs with
\textit{MTF} and obtained the same results. We then recursively
removed edges in all possible ways from these mtf Ramsey graphs to
obtain the complete set of Ramsey graphs for $R(K_3,K_8)$ with 27
vertices. This yielded 477~142 Ramsey graphs. As a test we verified
that all of the 430~215 previously known Ramsey graphs are indeed
included in our list. Goedgebeur and Radziszowski~\cite{staszek_12}
generated all Ramsey graphs for $R(K_3,K_8)$ with 27 vertices and at
most 88 edges using an independent program and obtained the same
results. The list can be downloaded from~\cite{ramsey-data-site}.
Table~\ref{table:counts_r38} contains the counts of these graphs
according to their number of edges.

We did not construct the lists of all Ramsey graphs with less than 27
vertices for $R(K_3,K_8)$ as there are too much of these graphs to
store.


\begin{table}
\begin{center}
\begin{tabular}{| c || c |}
\hline 
Number of & Number of \\
edges     & Ramsey graphs\\
\hline 
85  &  4\\
86  &  92\\
87  &  1 374\\
88  &  11 915\\
89  &  52 807\\
90  &  122 419\\
91  &  151 308\\
92  &  99 332\\
93  &  33 145\\
94  &  4 746\\
\hline
\end{tabular}
\end{center}

\caption{Counts of all 477 142 Ramsey graphs with 27 vertices for $R(K_3, K_8)$ according to their number of edges.}

\label{table:counts_r38}
\end{table}


\begin{table}
\begin{center}
\small\addtolength{\tabcolsep}{-2pt}
\begin{tabular}{|c|c|c|c|c|c|c|c|c|}
\hline
 & $|H|=3$ & $|H|=4$ & $|H|=5$ & $|H|=6$ & $|H|=7$ & $|H|=8$ & $|H|=9$ & $|H|=10$ \\
\hline
$r=5$ & 1 & & & & &  &  & \\
\hline
$r=6$ &1 & & & & &  &  & \\
\hline
$r=7$ &  & 5 & & & &  &  & \\
\hline
$r=8$ &  &   & & & &  &  & \\
\hline
$r=9$ &  & 1 & 18 & & &  &  & \\
\hline
$r=10$ & &   &    & & &  &  & \\
\hline
$r=11$ & &   & 2  & 98 & &  &  & \\
\hline
$r=12$ & &   &    & 6  & &  &  & \\
\hline
$r=13$ & &   &    & 2  & 772 &  &  & \\
\hline
$r=14$ & &   & 1  & 4  &  40 &  &  & \\
\hline
$r=15$ & &   &    &    &     & 9 024  &  & \\
\hline
$r=16$ & &   &    &    &  13 & 1 440  &  & \\
\hline
$r=17$ & &   &    & 1  &  19 &  498  & 242 773  & \\
\hline
$r=18$ & &   &    & 1  &   7 &  119  & 16 024  & \\
\hline
$r=19$ & &   &    &    &     &       & 311  & 10 101 711\\
\hline
$r=20$ & &   &    &    &     &      &  & 504\\
\hline
$r=21$ & &   &    &    &   1 &  28    & 1 809  & 1 602 240\\
\hline
$r=22$ & &   &    &    &     &       & 22  & 3 155\\
\hline
$r=23$ & &   &    &    &   1 &  6    & 98  & 6 960 \\
\hline
$r=24$ & &   &    &    &     &       &  &   \\
\hline
$r=25$ & &   &    &    &     &  1    & 26  &  1 384 \\
\hline
$r=26$ & &   &    &    &     &       & 5  &  316 \\
\hline
$r=27$ & &   &    &    &     &       & 3  &  92 \\
\hline
$r=28$ & &   &    &    &     &  1    & 7  &  142 \\
\hline
$r=29$ & &   &    &    &     &       &  &   30\\
\hline
$r=30$ & &   &    &    &     &       &  &   3\\
\hline
$r=31$ & &   &    &    &     &       & 1  &  $\ge 16$ \\
\hline
$r=36$ & &   &    &    &     &       & 1  &  $\ge 8$ \\
\hline
\end{tabular}
\end{center} 

\caption{Numbers of connected graphs $H$ with Ramsey number $R(K_3,H)= r$. Note that the 10 graphs with 
$R(K_3,H) \ge 31$, but whose Ramsey number we were unable to determine are not included in the table.}

\label{table:connected_graphs}

\end{table}


\begin{table}
\begin{center}
\small\addtolength{\tabcolsep}{-2pt}
\begin{tabular}{|c|c|c|c|c|c|c|c|c|}
\hline
  & $|H|=3$ & $|H|=4$ & $|H|=5$ & $|H|=6$ & $|H|=7$ & $|H|=8$ & $|H|=9$ & $|H|=10$\\
\hline
$r=3$ & 2 & & & & & & & \\
\hline
$r=4$ & & 2 & & & & & & \\
\hline
$r=5$ & & 2 & 4 & & & & & \\
\hline
$r=6$ & & 1 & 3 & 7  & & & & \\
\hline
$r=7$ & &   & 5 & 11 & 18  & & & \\
\hline
$r=8$ & &   &   & 3  & 5   & 23  & & \\
\hline
$r=9$ &  &  & 1 & 20 & 50  & 60  & 83    & \\
\hline
$r=10$ & &  &   &    &     & 36  & 68    & 151 \\
\hline
$r=11$ & &  &   & 2  & 102 & 225 & 427   & 596 \\
\hline
$r=12$ & &  &   &    & 6   & 12  & 144   & 168 \\
\hline
$r=13$ & &  &   &    & 2   & 776 & 1 552 & 3 734\\
\hline
$r=14$ & &  &   & 1  & 6   & 52  & 107   & 447\\
\hline
$r=15$ & &  &   &    &     &     & 9 024 & 18 048\\
\hline
$r=16$ & &  &   &    &     & 13  & 1 466 & 2 933\\
\hline
$r=17$ & &  &   &    & 1   & 21  & 540   & 243 856\\
\hline
$r=18$ & &  &   &    & 1   & 9   & 137   & 16 301 \\
\hline
$r=19$ & &  &   &    &     &     &       & 311\\
\hline
$r=20$ & &  &   &    &     &     &       & \\
\hline
$r=21$ & &  &   &    &     &  1  & 30    & 1 869\\
\hline
$r=22$ & &  &   &    &     &     &       & 22\\
\hline
$r=23$ & &  &   &    &     &  1  & 8      & 114\\
\hline
$r=24$ & &  &   &    &     &     &        & \\
\hline
$r=25$ & &  &   &    &     &     &    1   & 28\\
\hline
$r=26$ & &  &   &    &     &     &        & 5\\
\hline
$r=27$ & &  &   &    &     &     &        & 3\\
\hline
$r=28$ & &  &   &    &     &     &    1    & 9\\
\hline
$r=31$ & &  &   &    &     &     &         & 1\\
\hline
$r=36$ & &  &   &    &     &     &         & 1\\
\hline
\end{tabular}
\end{center}

\caption{Numbers of disconnected graphs $H$ with Ramsey number $R(K_3,H)= r$.}

\label{table:disconnected_graphs}

\end{table}


\section{Closing remarks}

Since all computational results were independently obtained by both
\textit{MTF} and \textit{triangleramsey}, the chance of wrong results caused by errors in
the implementation is extremely small.

We believe that specialized algorithms and/or new theoretical results
will be required to determine the triangle Ramsey number of the
remaining 10 graphs and hope that this challenge to complete the
possibly last complete list of triangle Ramsey numbers for a very long time
will be taken up by the mathematical community..

Besides the already mentioned property that $n=10$ is the first case where
the Ramsey numbers of  $R(K_3,K_n-(m\cdot e))$ are not 
the same for all $2\le m \le n/2$, the most striking observation in the
list is possibly that while for $7 \le n \le 9$ the graph 
$K_n - P_5$ has a smaller Ramsey number than
$K_n - 2P_3$, for $n = 10$ they have the same Ramsey number
(i.e.\ 29).

The latest version of \textit{triangleramsey} can be downloaded from~\cite{triangleramsey-site}. The list of
the Ramsey graphs used in this research can be obtained
from \textit{House of Graphs}~\cite{HOG} by searching for the keywords ``ramsey * order 10'' and the Ramsey numbers can be obtained from~\cite{ramseynumber-site}.

\subsection*{Acknowledgements}

This work was carried out using the Stevin Supercomputer Infrastructure at Ghent University.
Jan Goedgebeur is supported by a PhD grant from the Research Foundation of Flanders (FWO).


\newpage
\section{The triangle Ramsey number for connected graphs of order 10}
\label{section_graph_figures}

\setlength{\parindent}{0cm} 

The following 10 graphs $H^c$ have  $R(K_3,H) > 30$, but we were unable to determine their Ramsey number. Graphs which must have the same Ramsey number are grouped by $\lfloor$ and $\rfloor$.

\setlength{\unitlength}{1cm}
\begin{minipage}[t]{2.2cm}
\begin{picture}(1.4,1.8)
\leavevmode
\epsfxsize=1.4cm
\epsffile{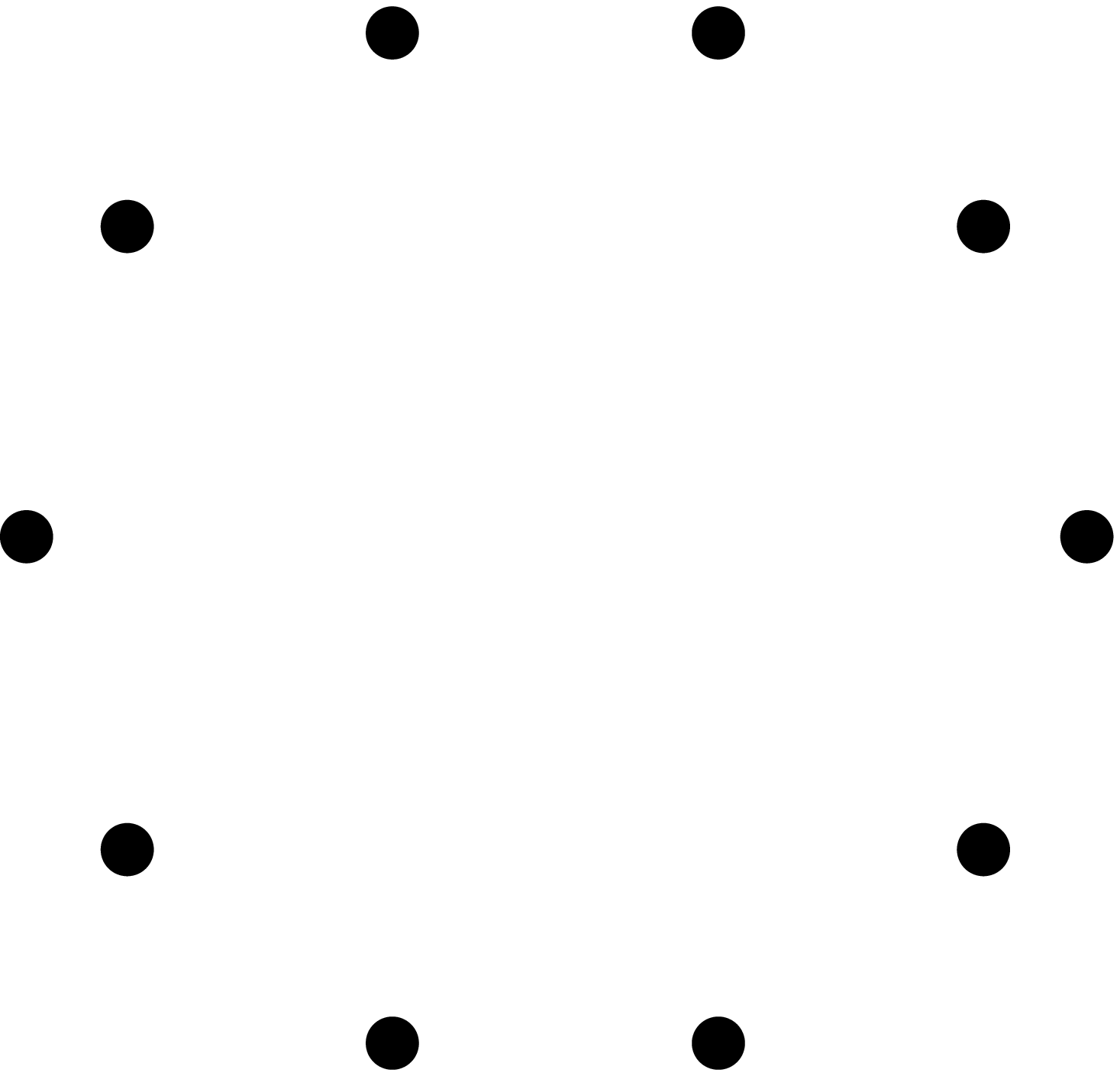}
\end{picture}\par
\end{minipage}
\begin{minipage}[t]{2.2cm}
\begin{picture}(1.4,1.8)
\leavevmode
\epsfxsize=1.4cm
\epsffile{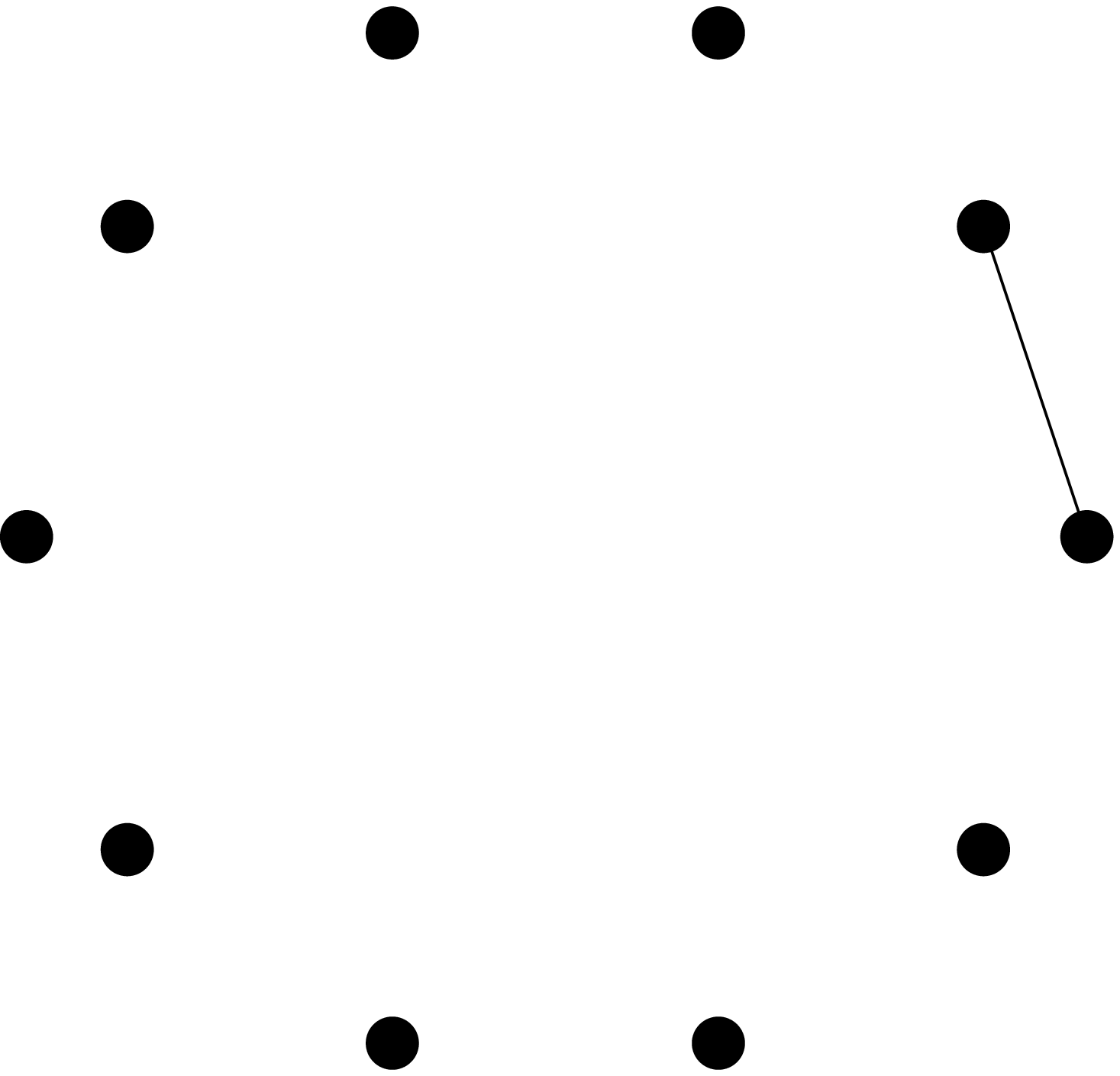}
\end{picture}\par
\end{minipage}
\begin{minipage}[t]{2.2cm}
\begin{picture}(1.4,1.8)
\leavevmode
\epsfxsize=1.4cm
\epsffile{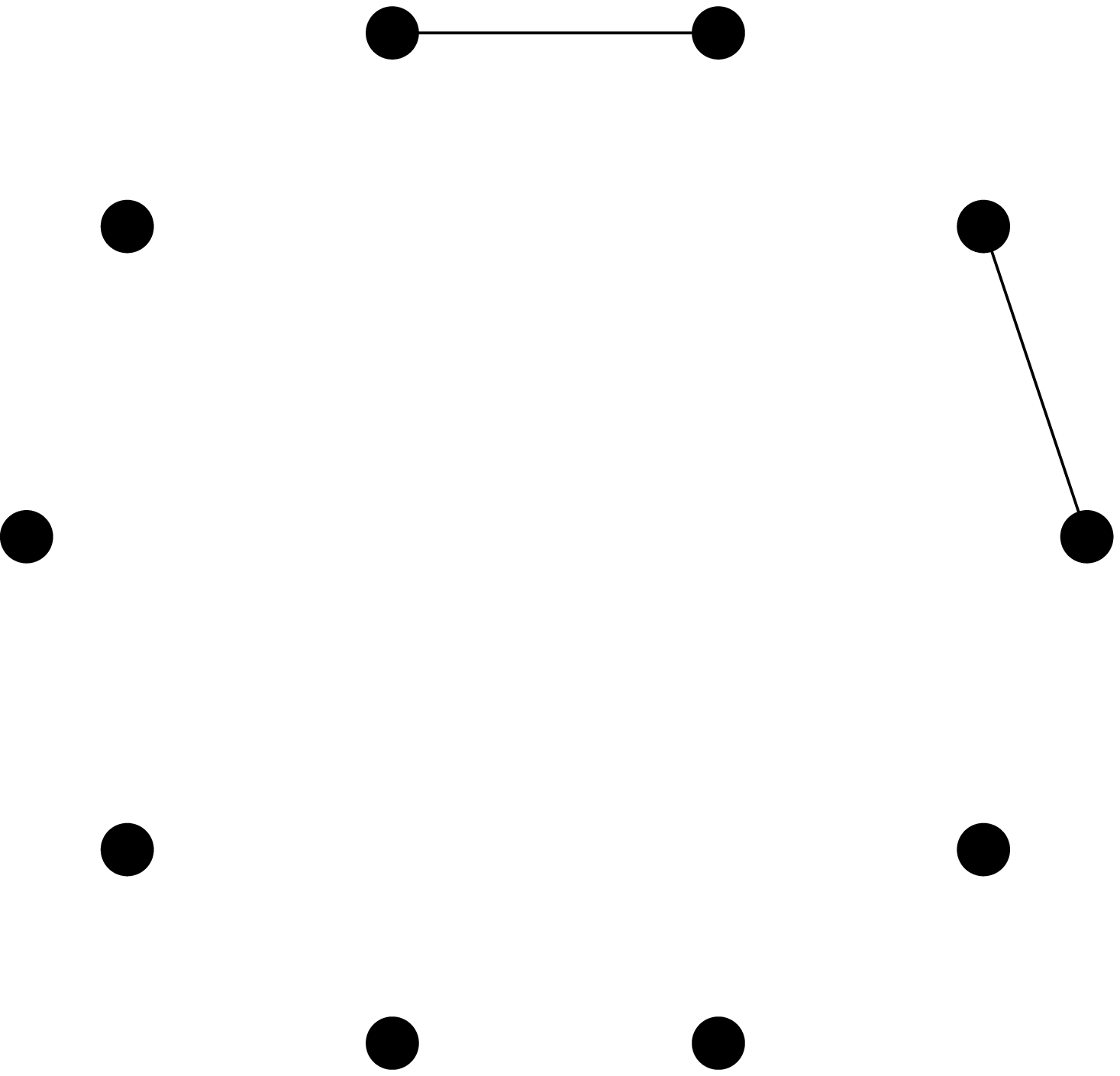}
\end{picture}\par
\end{minipage}
$\lfloor$
\begin{minipage}[t]{2.2cm}
\begin{picture}(1.4,1.8)
\leavevmode
\epsfxsize=1.4cm
\epsffile{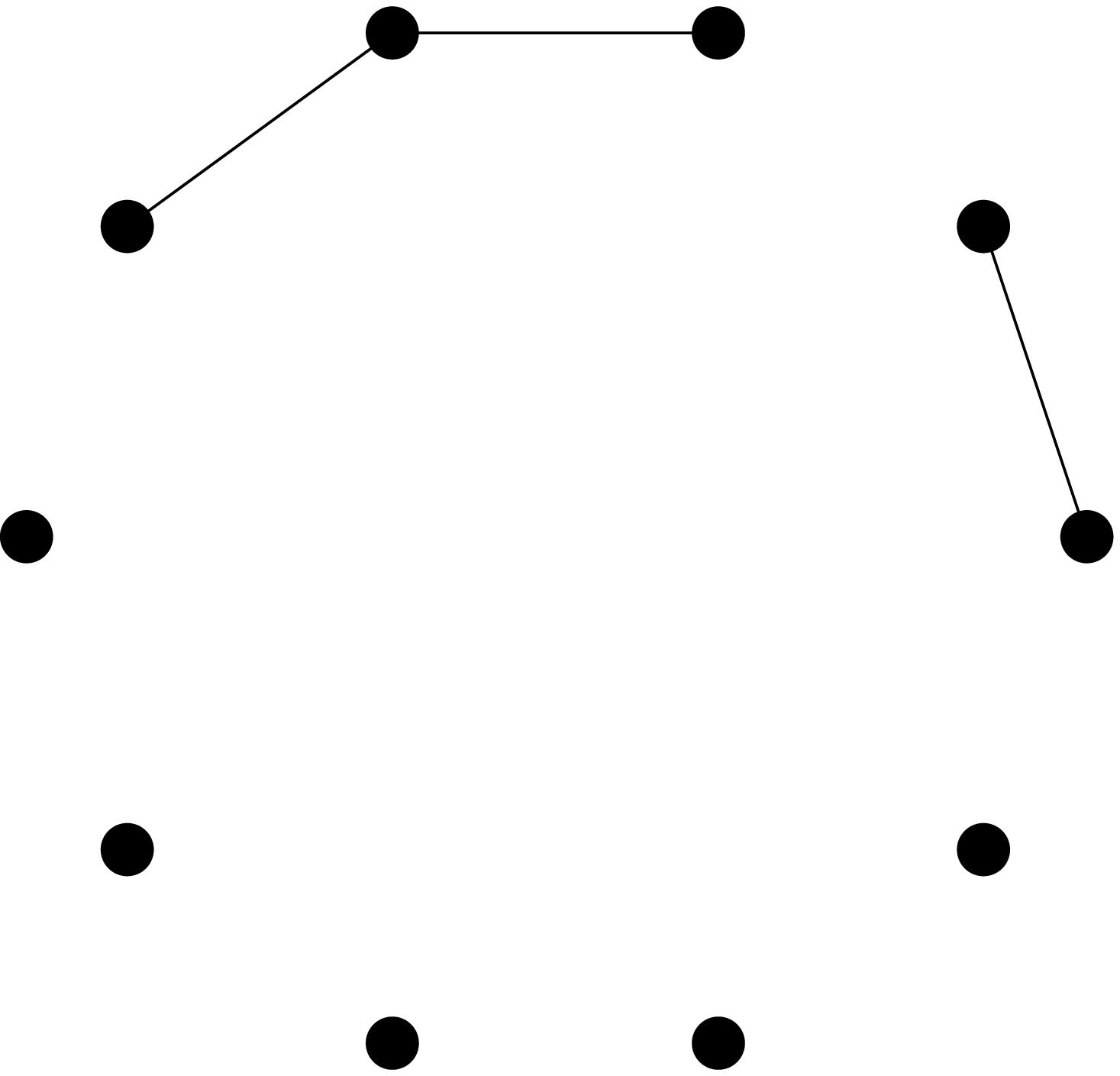}
\end{picture}\par
\end{minipage}
\begin{minipage}[t]{2.2cm}
\begin{picture}(1.4,1.8)
\leavevmode
\epsfxsize=1.4cm
\epsffile{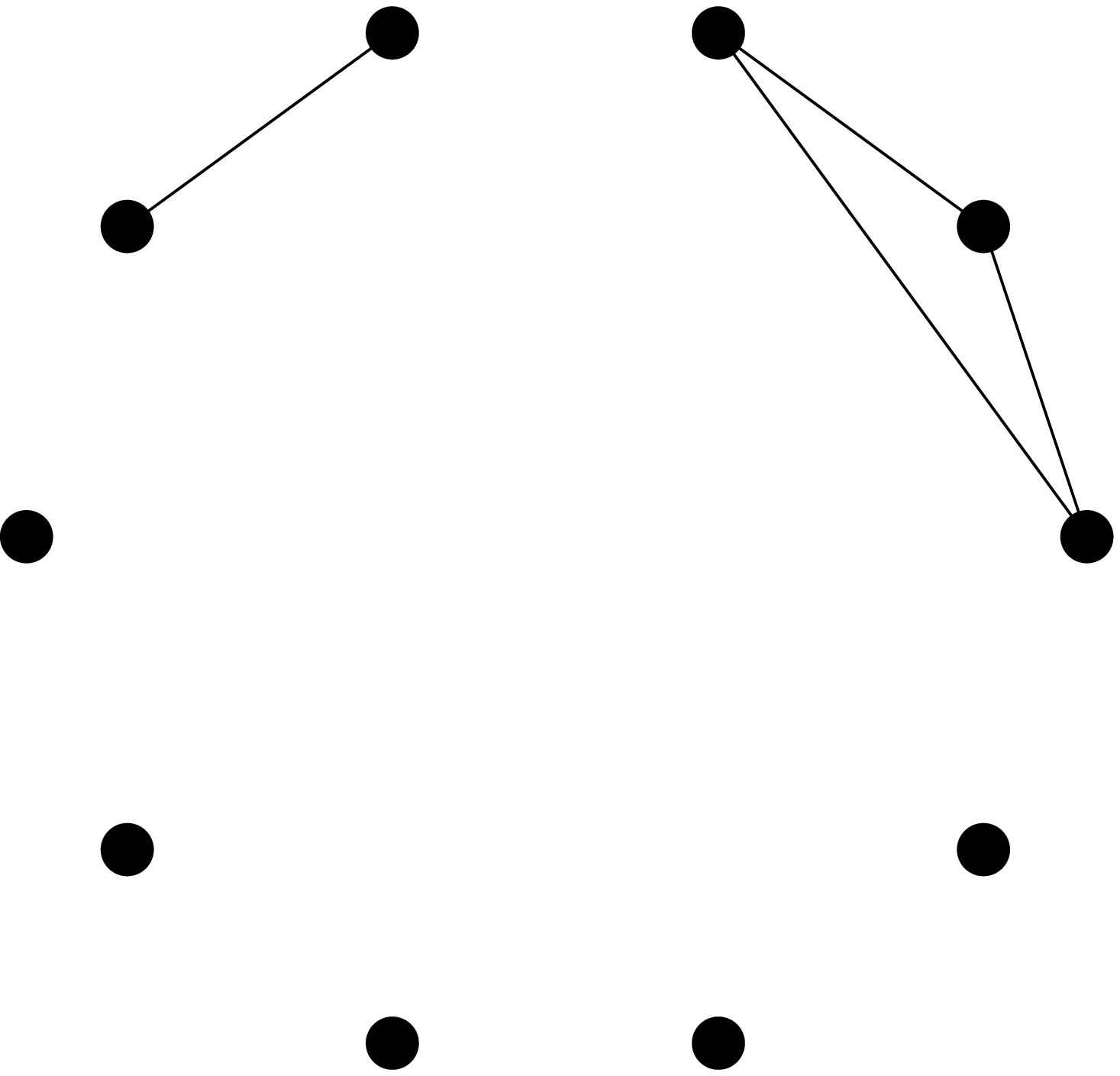}
\end{picture}$\rfloor$\par
\end{minipage}
\begin{minipage}[t]{2.2cm}
\begin{picture}(1.4,1.8)
\leavevmode
\epsfxsize=1.4cm
\epsffile{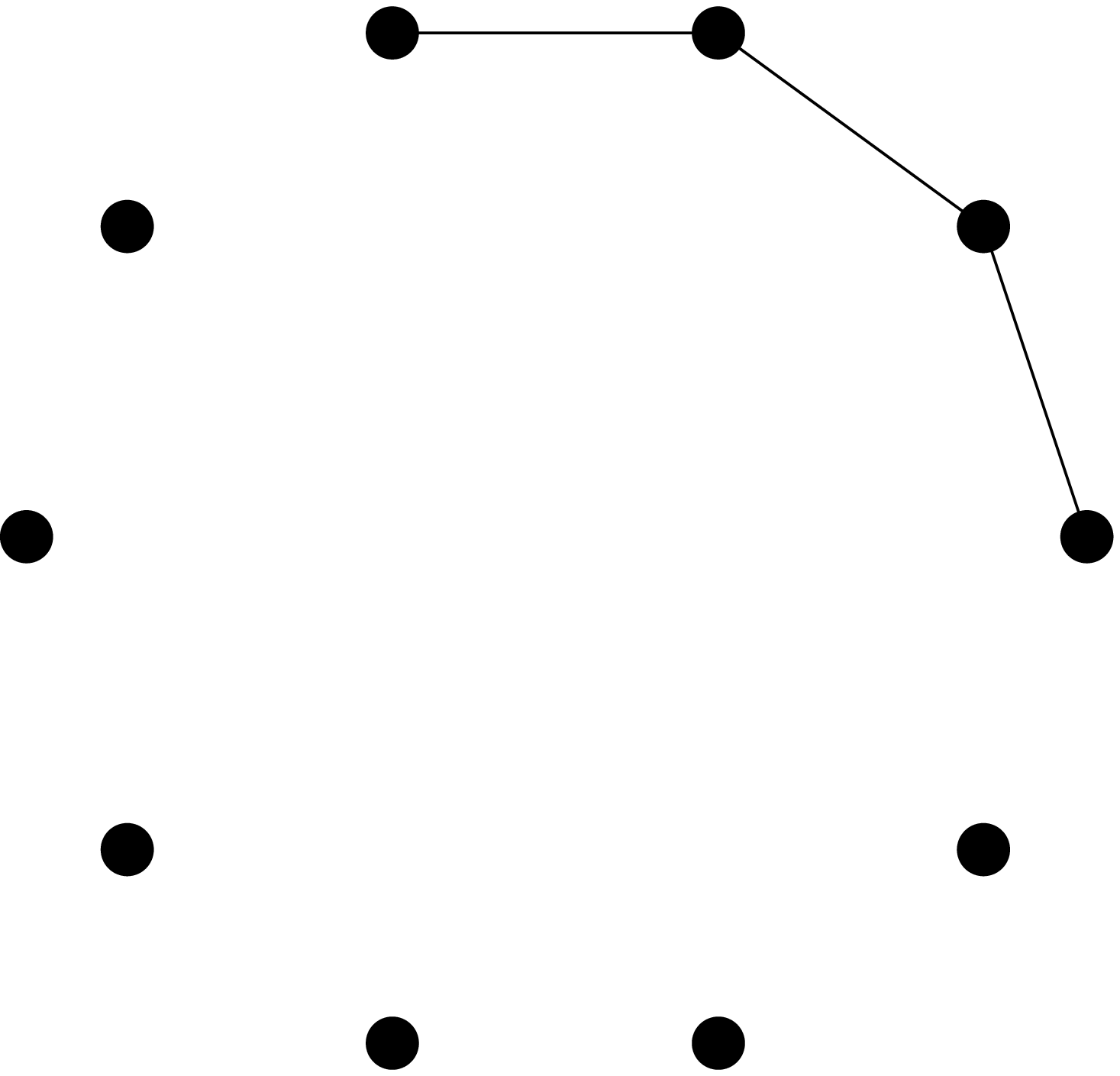}
\end{picture}\par
\end{minipage}
\begin{minipage}[t]{2.2cm}
\begin{picture}(1.4,1.8)
\leavevmode
\epsfxsize=1.4cm
\epsffile{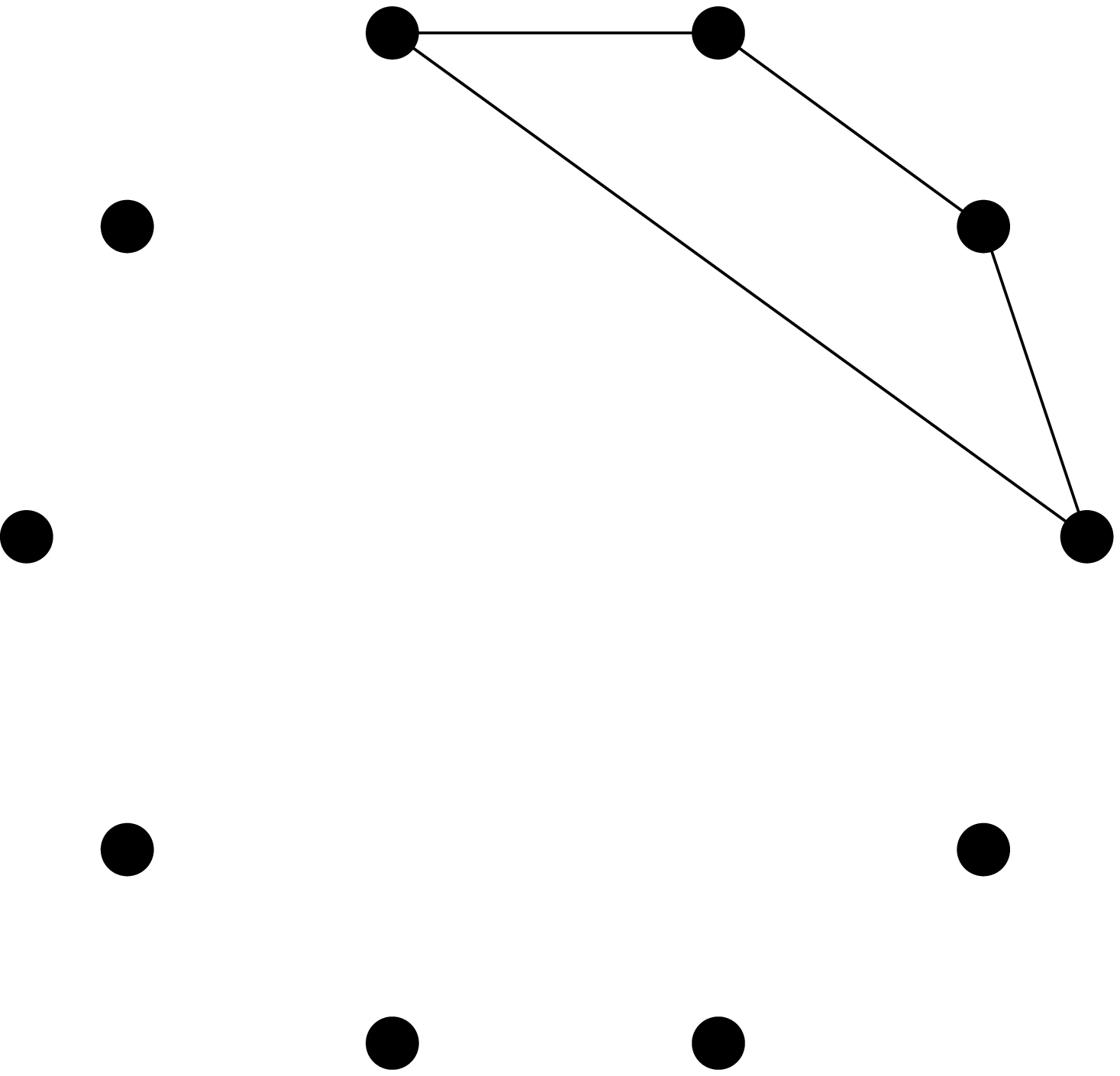}
\end{picture}\par
\end{minipage}
\begin{minipage}[t]{2.2cm}
\begin{picture}(1.4,1.8)
\leavevmode
\epsfxsize=1.4cm
\epsffile{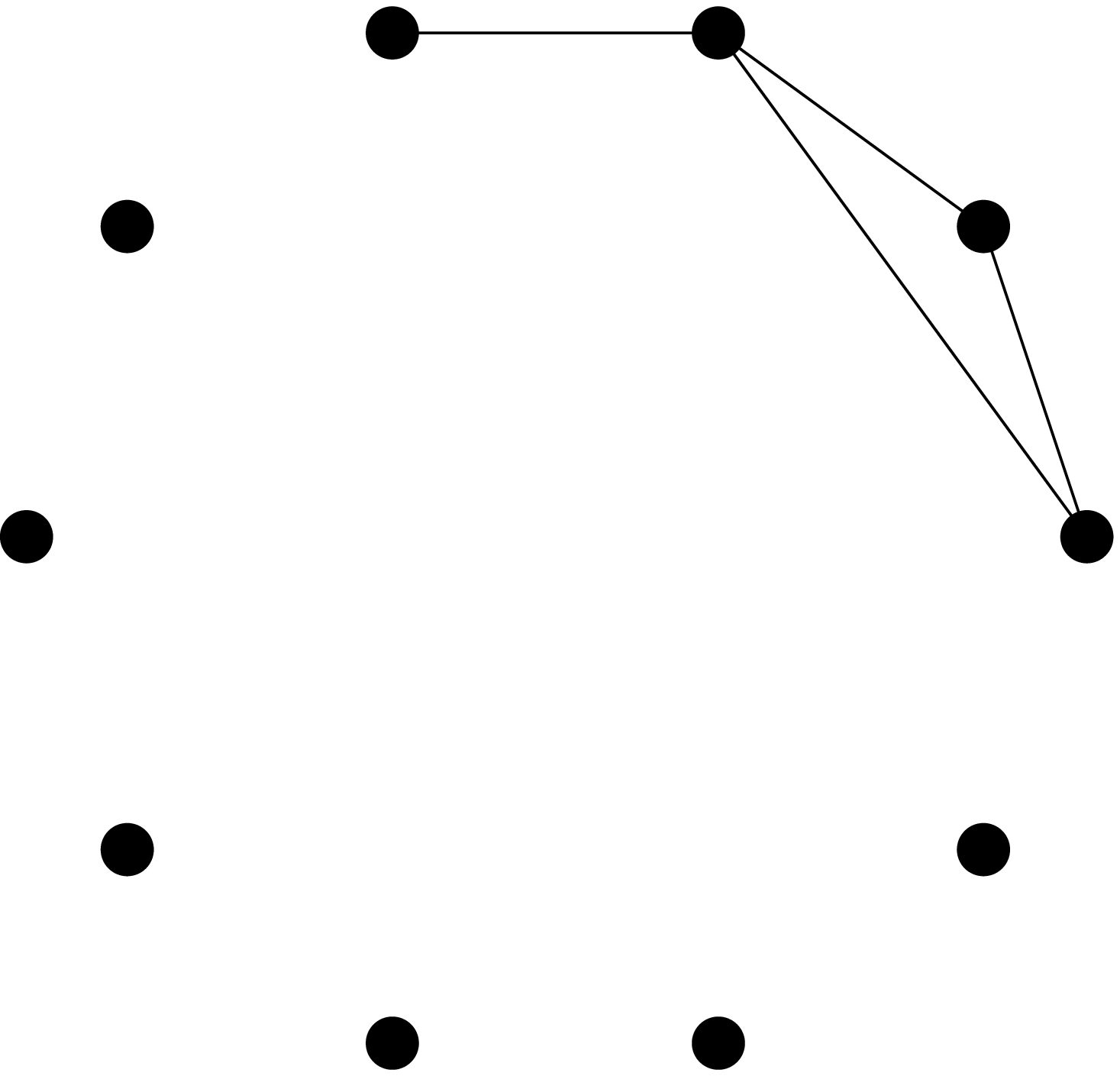}
\end{picture}\par
\end{minipage}
$\lfloor$
\begin{minipage}[t]{2.2cm}
\begin{picture}(1.4,1.8)
\leavevmode
\epsfxsize=1.4cm
\epsffile{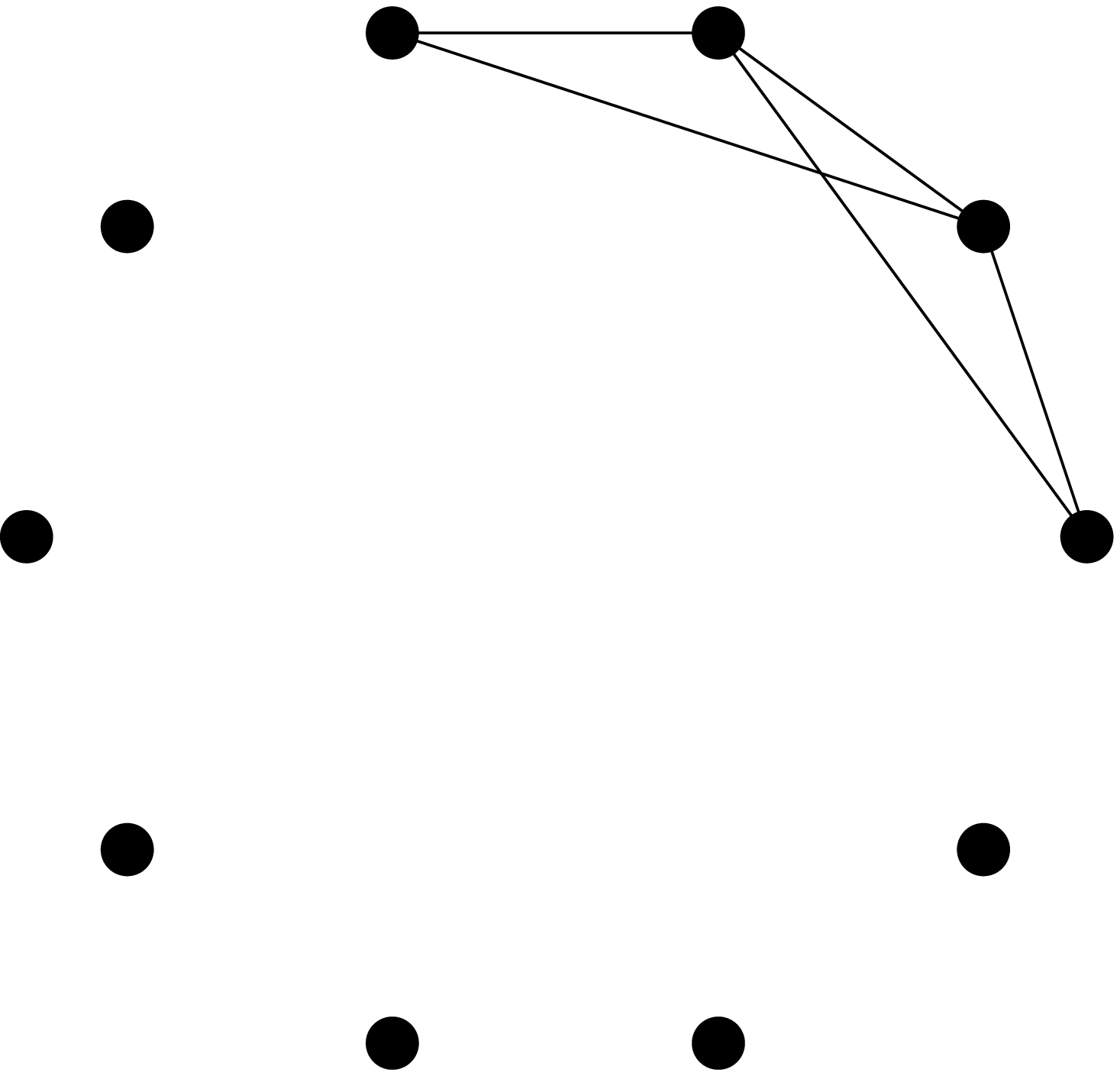}
\end{picture}\par
\end{minipage}
\begin{minipage}[t]{2.2cm}
\begin{picture}(1.4,1.8)
\leavevmode
\epsfxsize=1.4cm
\epsffile{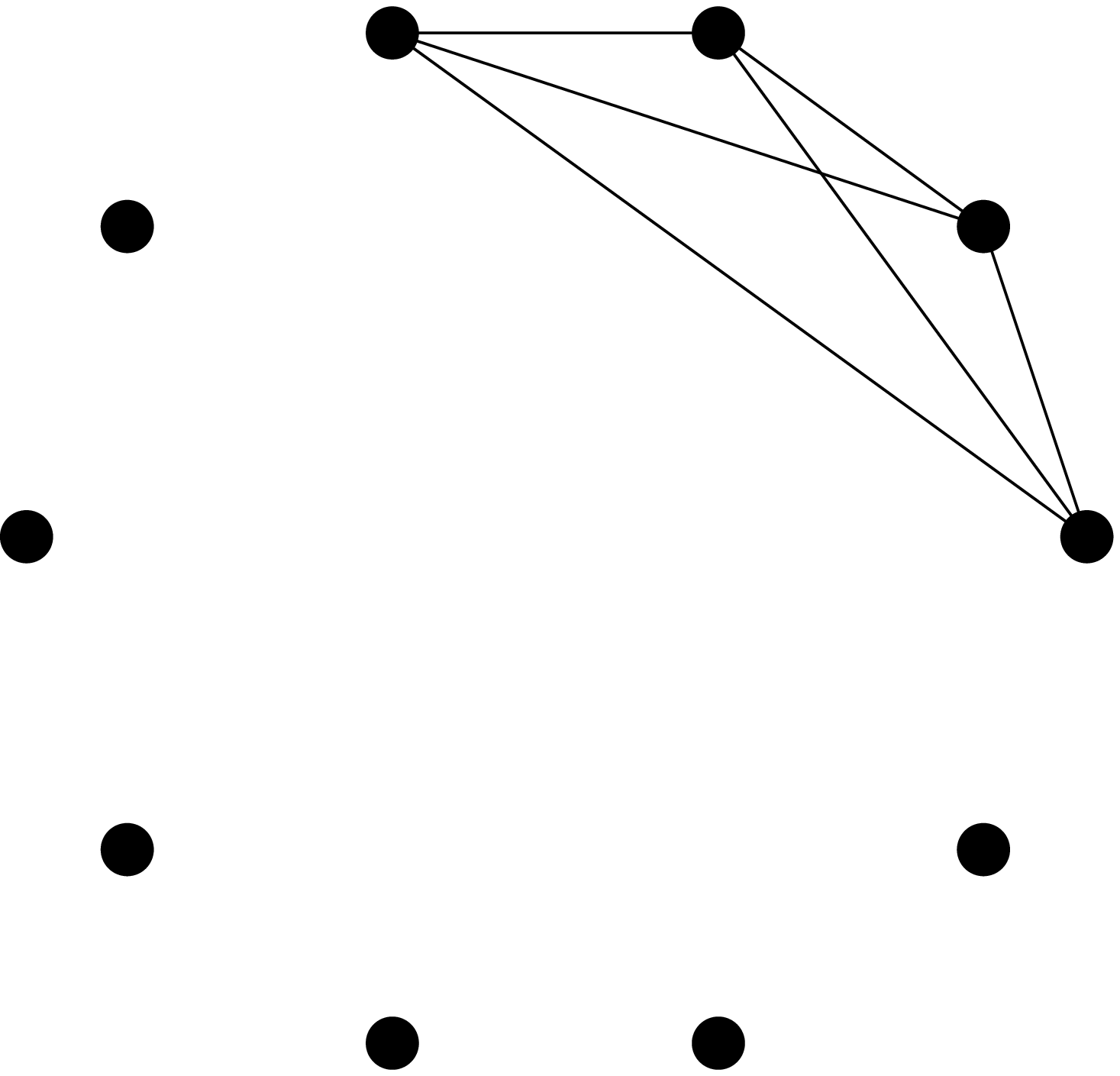}
\end{picture}$\rfloor$\par
\end{minipage}

\bigskip
\hrule
\smallskip
\hrule
\bigskip

If $H^c$ is one of the graphs:

\setlength{\unitlength}{1cm}
\begin{minipage}[t]{2.2cm}
\begin{picture}(1.4,1.8)
\leavevmode
\epsfxsize=1.4cm
\epsffile{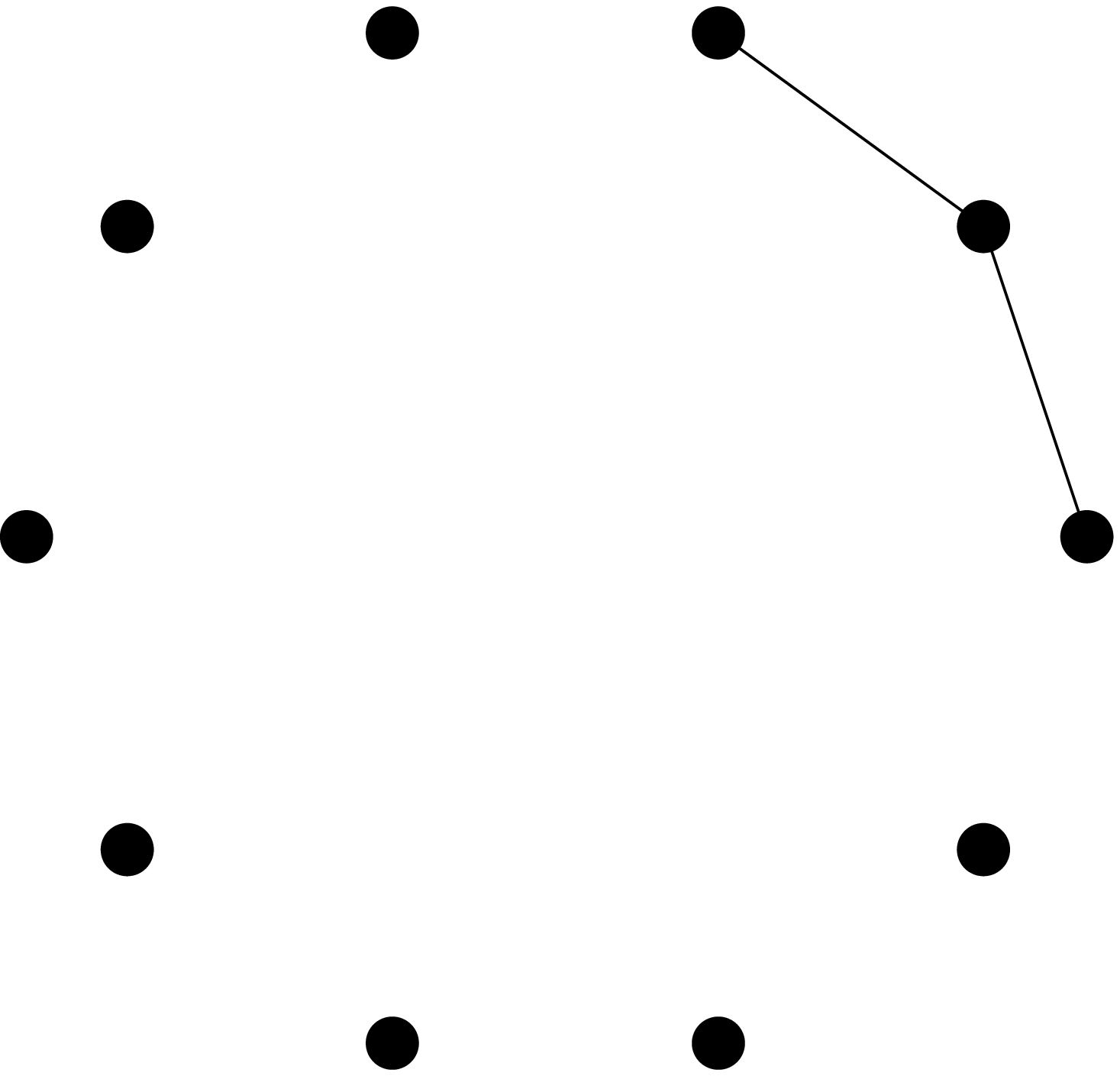}
\end{picture}\par
\end{minipage}
\begin{minipage}[t]{2.2cm}
\begin{picture}(1.4,1.8)
\leavevmode
\epsfxsize=1.4cm
\epsffile{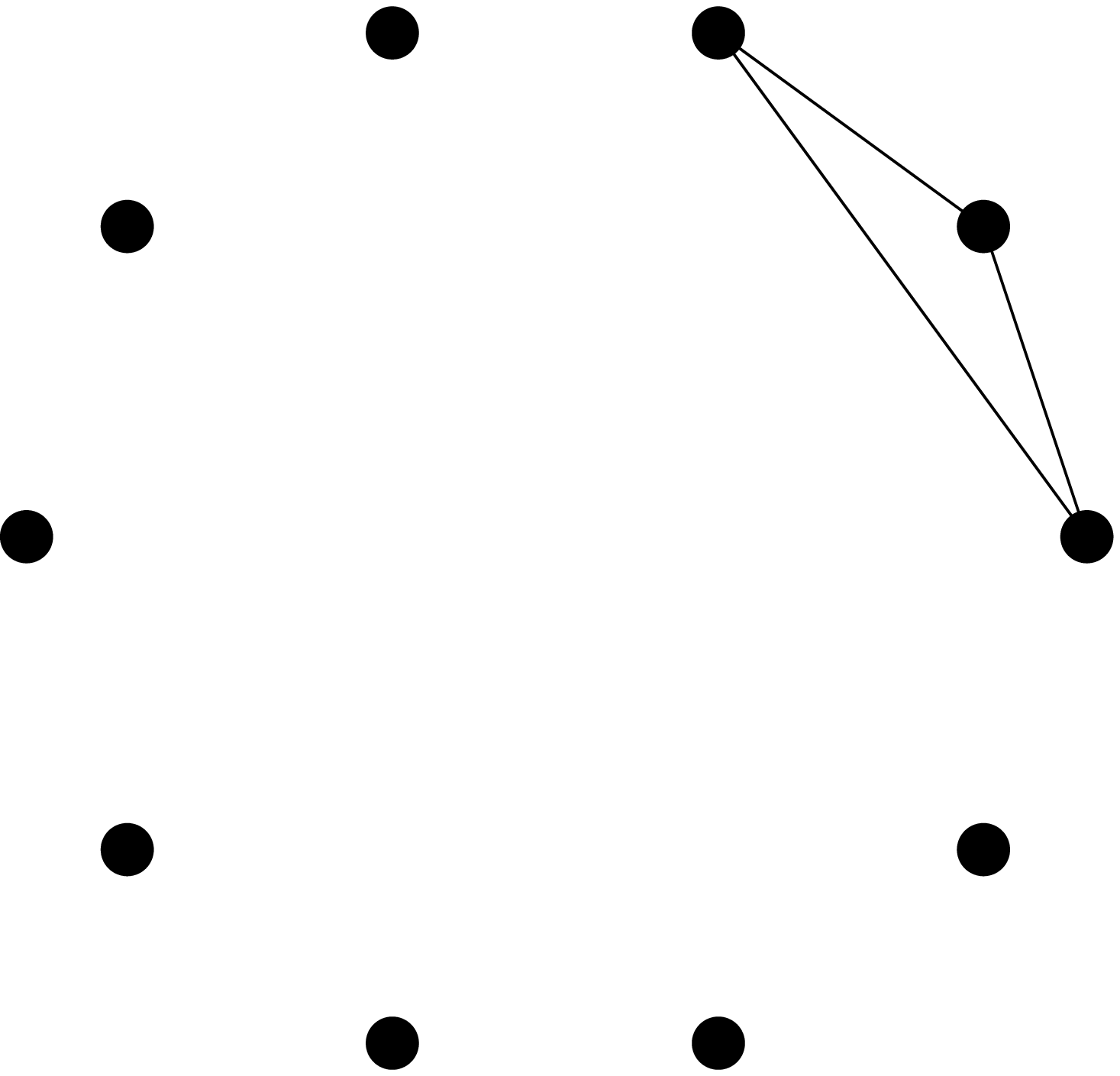}
\end{picture}\par
\end{minipage}
\begin{minipage}[t]{2.2cm}
\begin{picture}(1.4,1.8)
\leavevmode
\epsfxsize=1.4cm
\epsffile{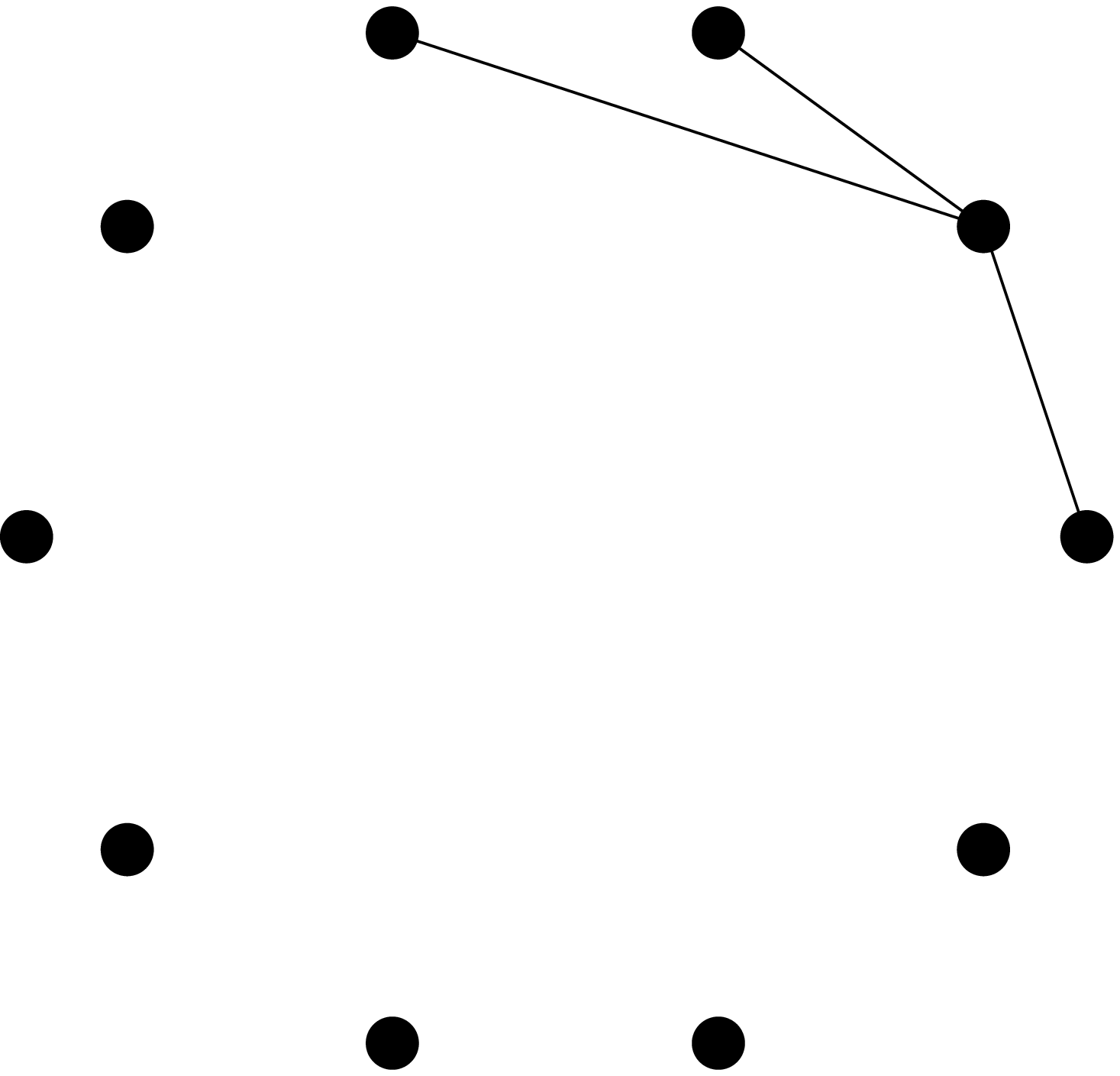}
\end{picture}\par
\end{minipage}
\begin{minipage}[t]{2.2cm}
\begin{picture}(1.4,1.8)
\leavevmode
\epsfxsize=1.4cm
\epsffile{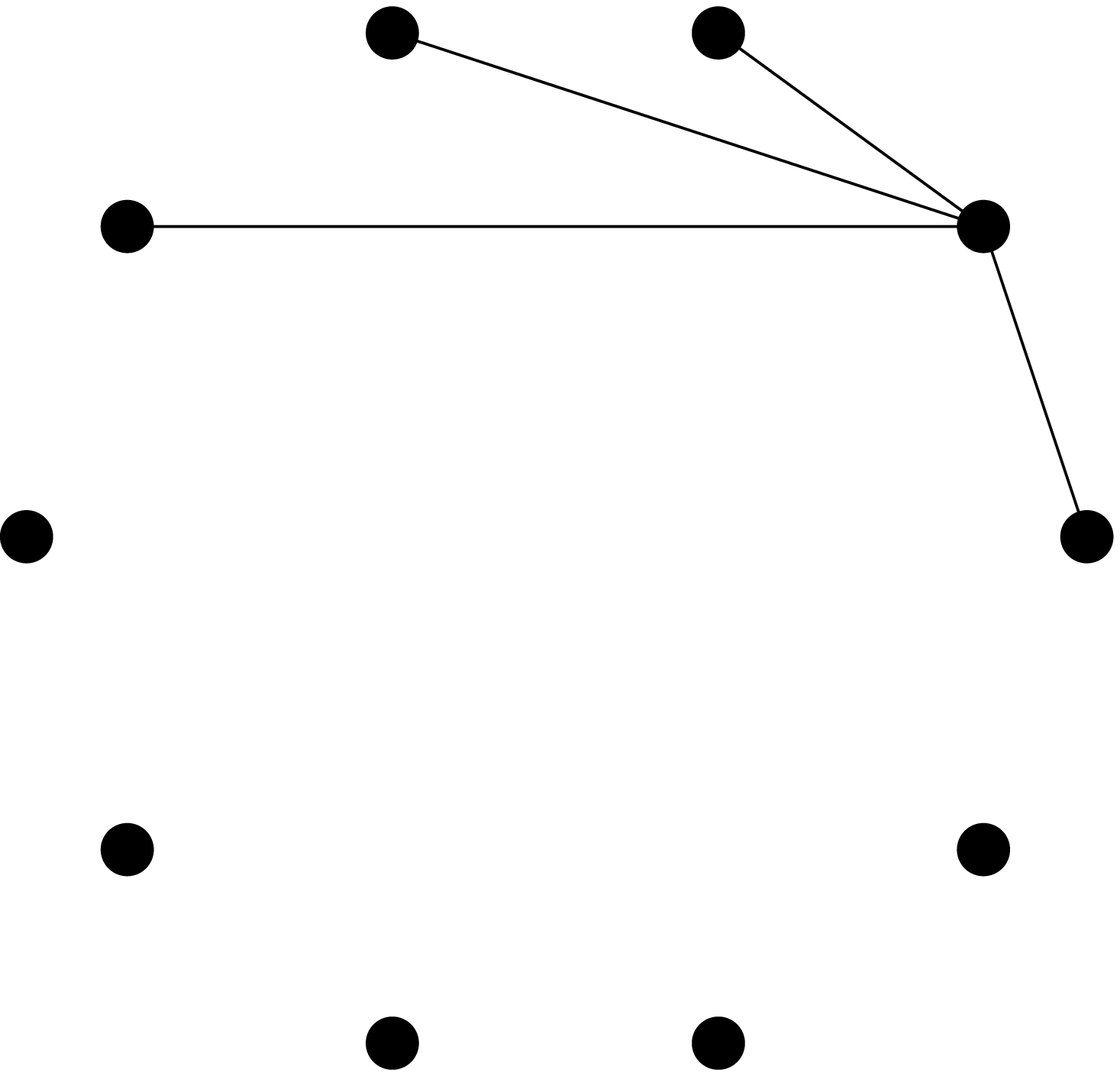}
\end{picture}\par
\end{minipage}
\begin{minipage}[t]{2.2cm}
\begin{picture}(1.4,1.8)
\leavevmode
\epsfxsize=1.4cm
\epsffile{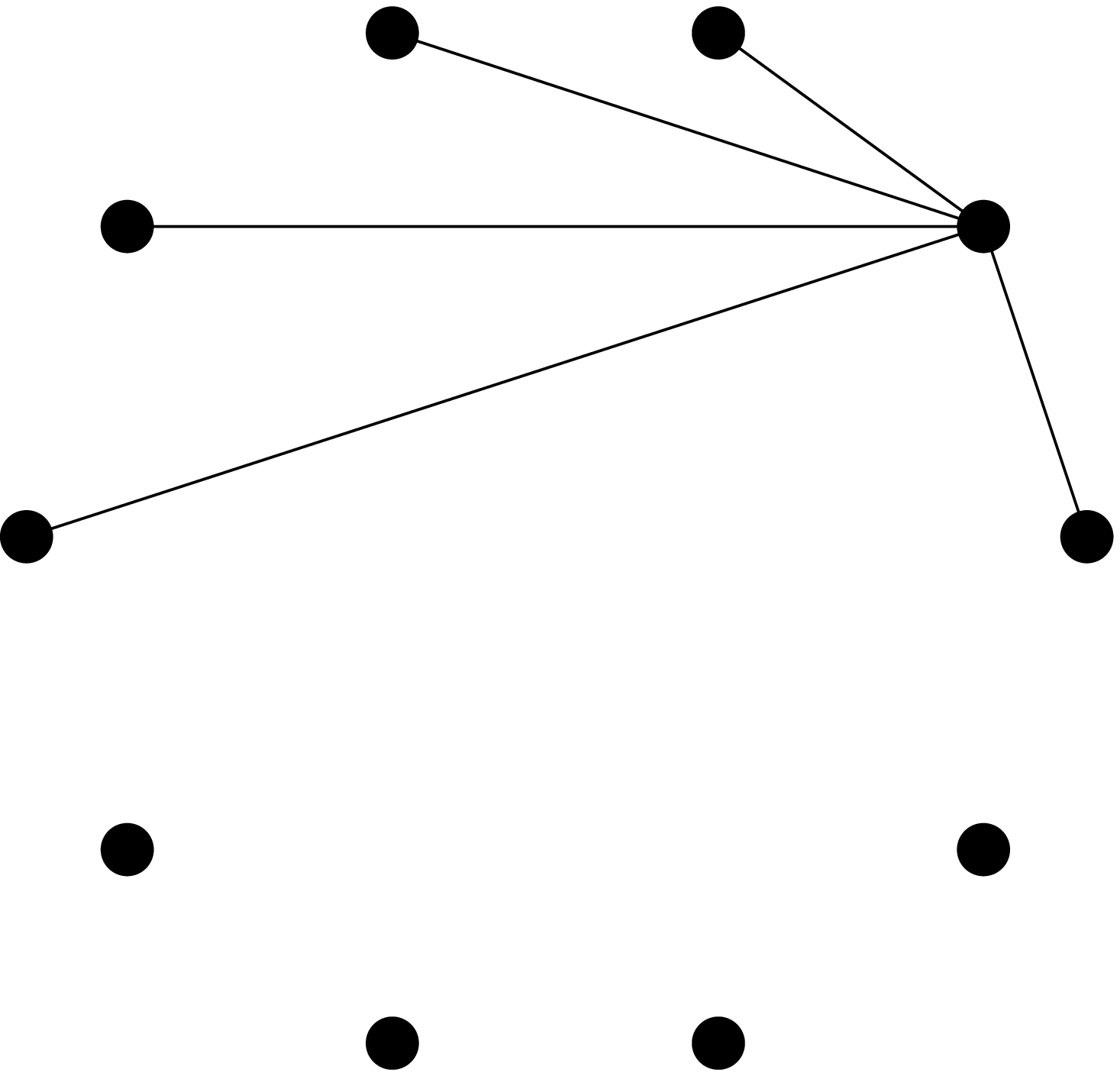}
\end{picture}\par
\end{minipage}
\begin{minipage}[t]{2.2cm}
\begin{picture}(1.4,1.8)
\leavevmode
\epsfxsize=1.4cm
\epsffile{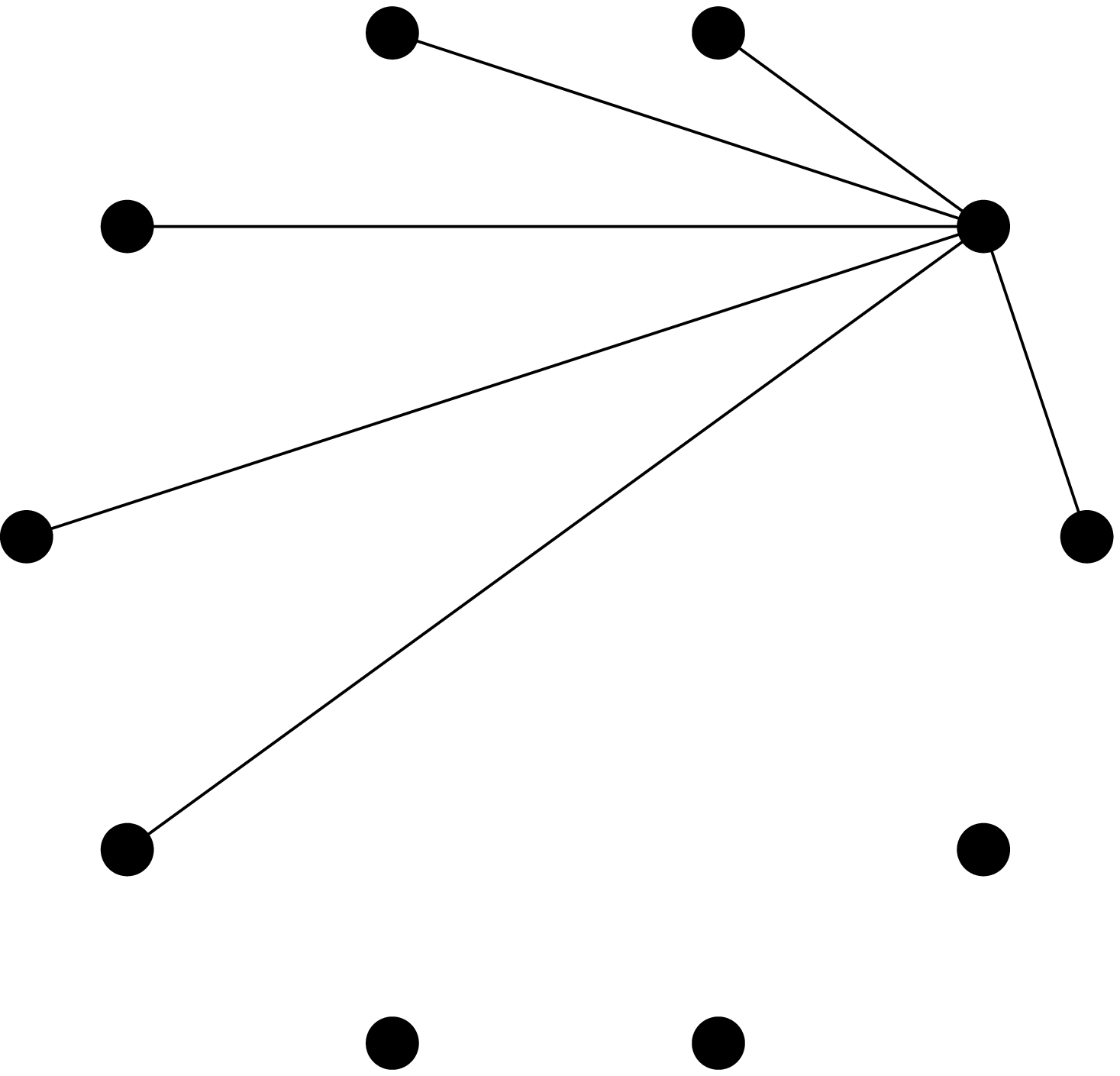}
\end{picture}\par
\end{minipage}
\begin{minipage}[t]{2.2cm}
\begin{picture}(1.4,1.8)
\leavevmode
\epsfxsize=1.4cm
\epsffile{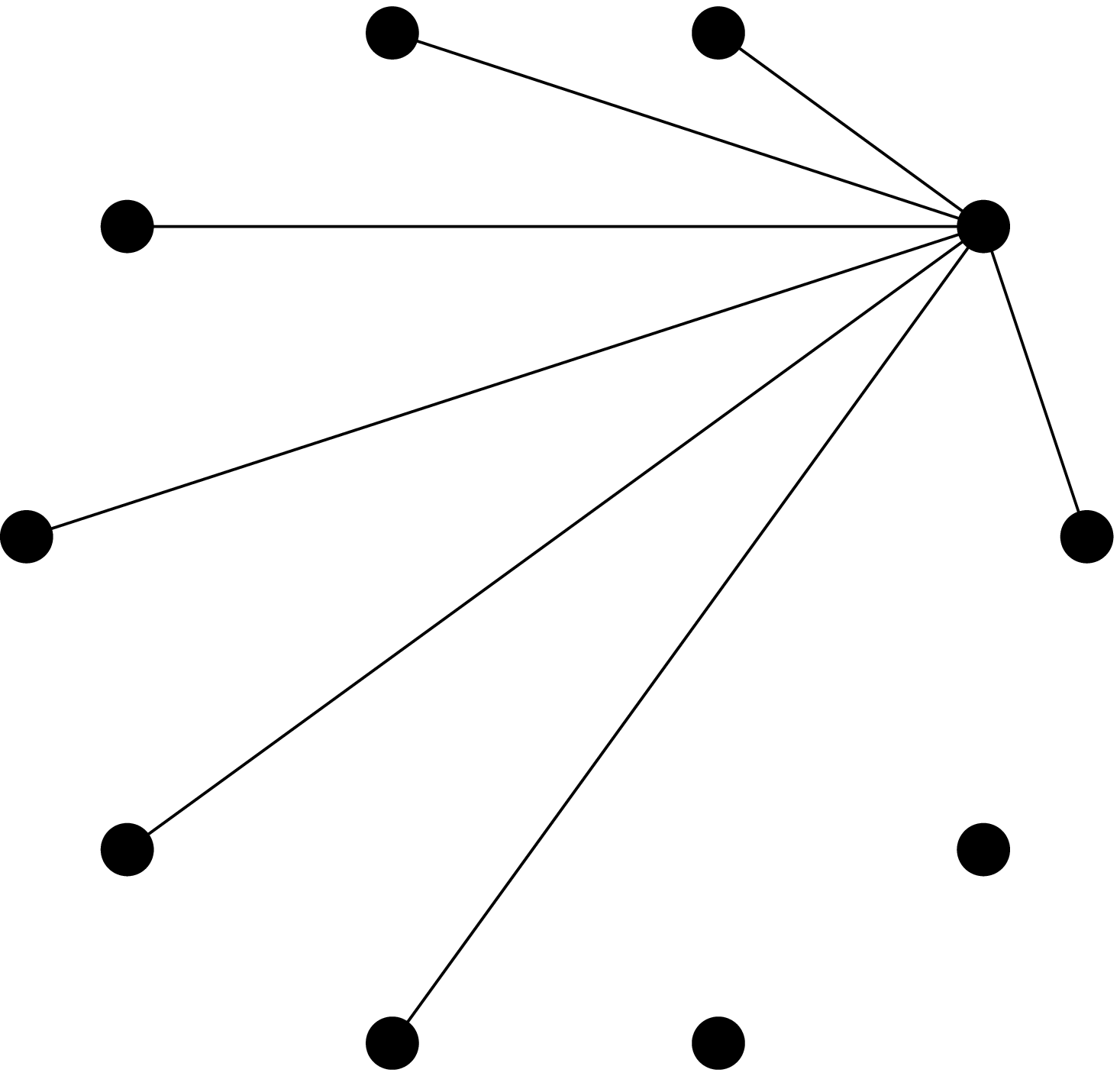}
\end{picture}\par
\end{minipage}
\begin{minipage}[t]{2.2cm}
\begin{picture}(1.4,1.8)
\leavevmode
\epsfxsize=1.4cm
\epsffile{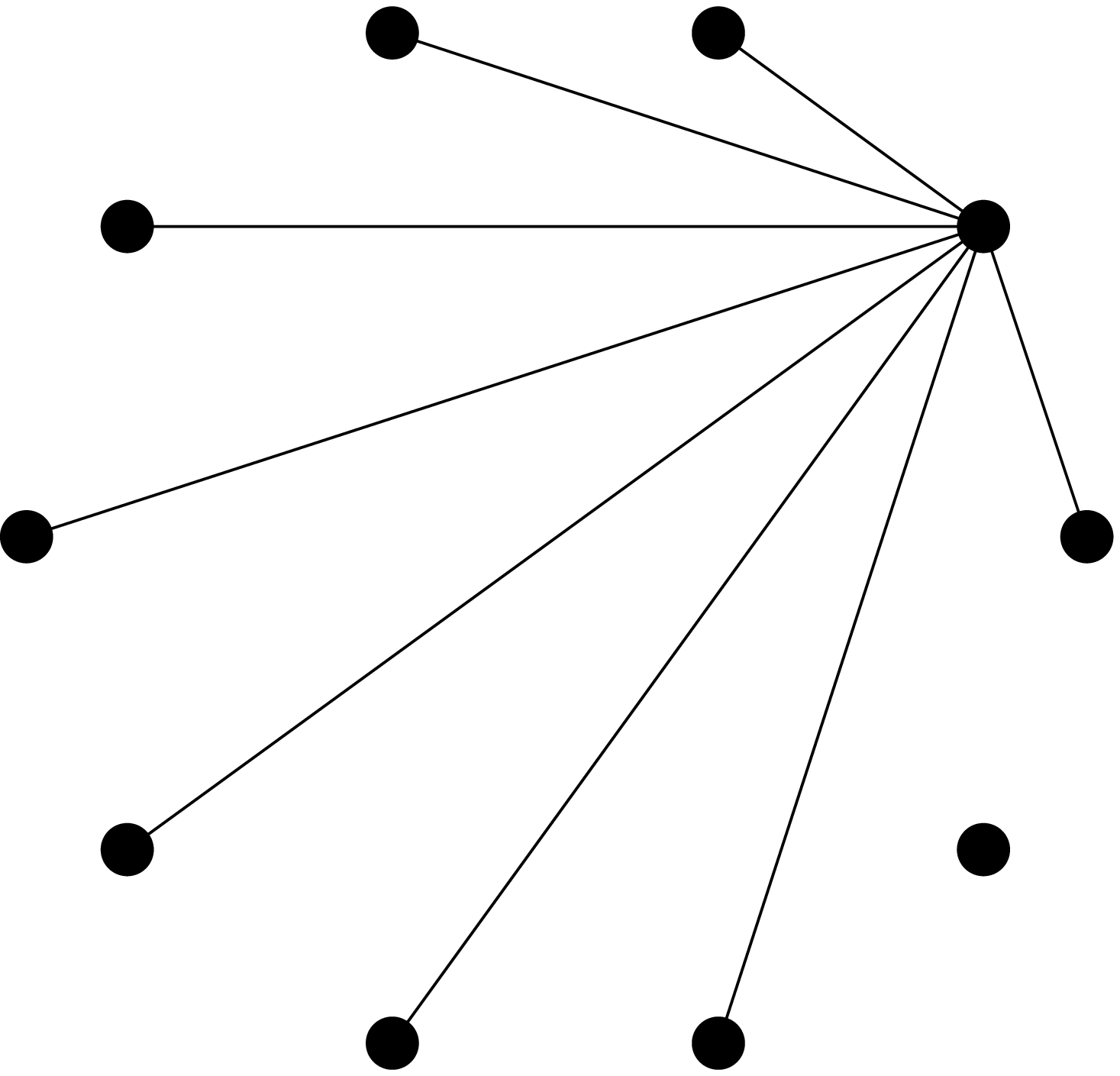}
\end{picture}\par
\end{minipage}

\medskip

Then $R(K_3,H) = 36$. These are possibly not \textbf{all} graphs with $R(K_3,H) = 36$.

\bigskip
\hrule
\bigskip

If $H^c$ is contained in one of the graphs:

\setlength{\unitlength}{1cm}
\begin{minipage}[t]{2.2cm}
\begin{picture}(1.4,1.8)
\leavevmode
\epsfxsize=1.4cm
\epsffile{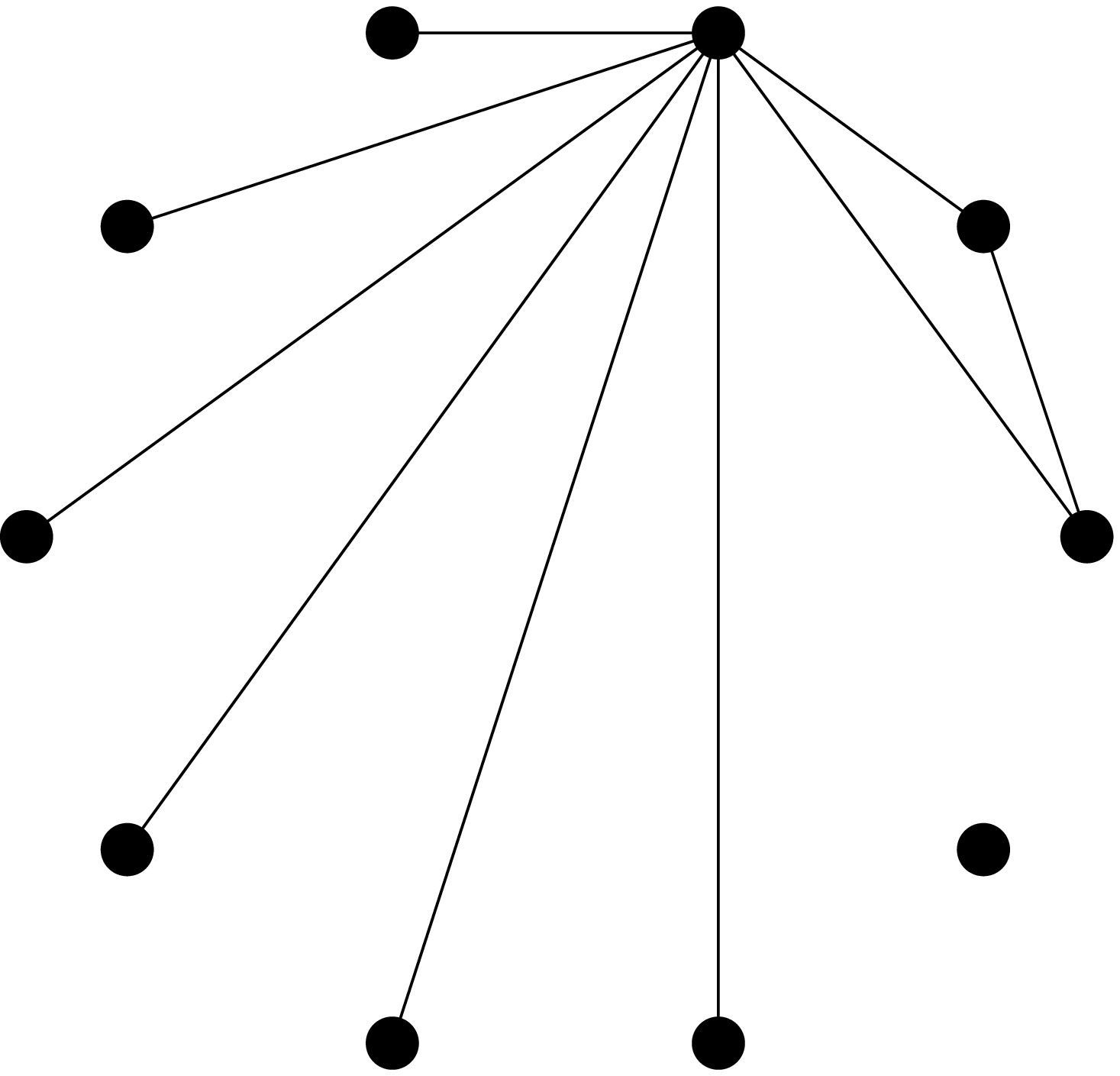}
\end{picture}\par
\end{minipage}
\begin{minipage}[t]{2.2cm}
\begin{picture}(1.4,1.8)
\leavevmode
\epsfxsize=1.4cm
\epsffile{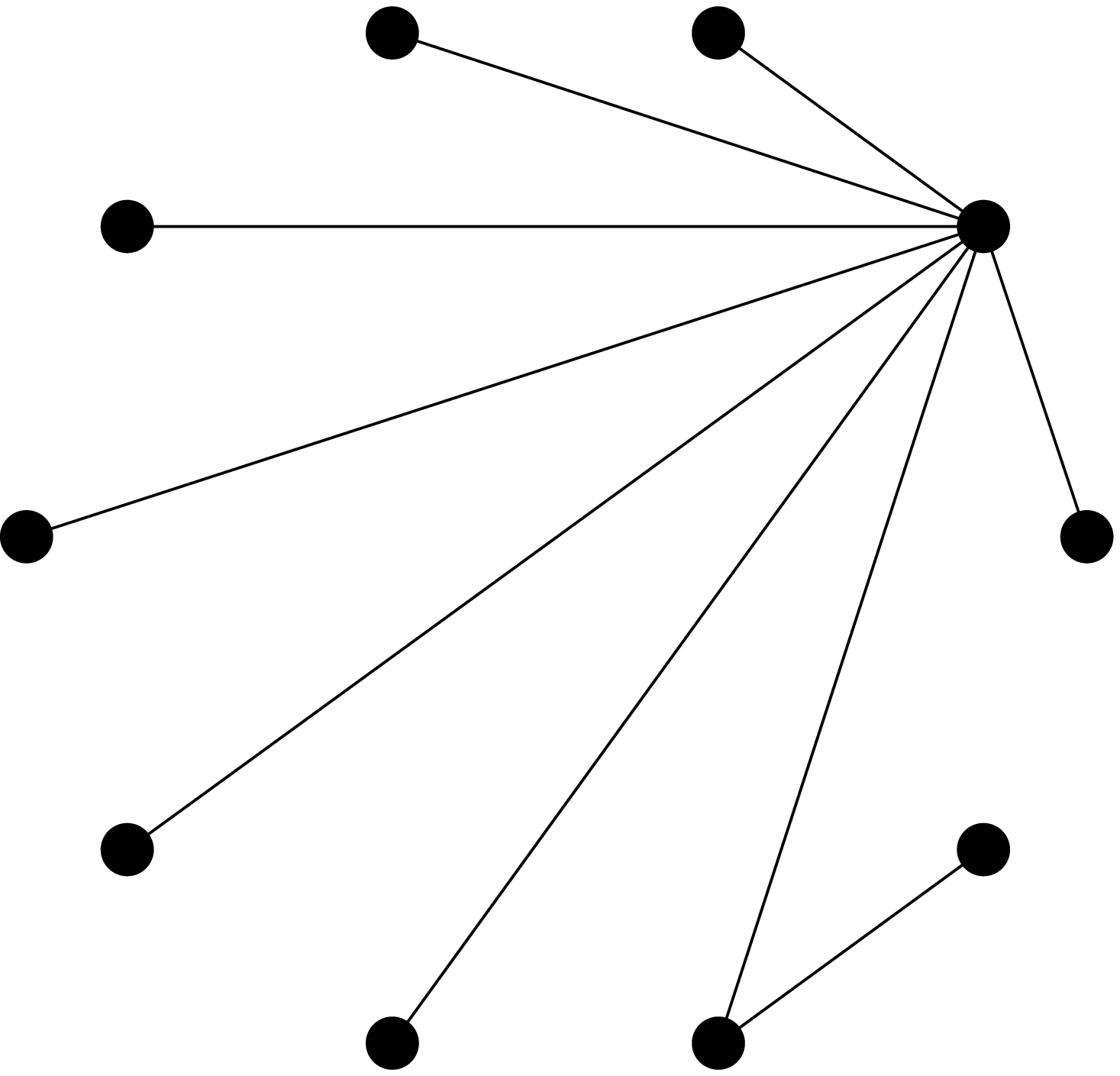}
\end{picture}\par
\end{minipage}

\bigskip
and contains one of the graphs:

\setlength{\unitlength}{1cm}
\begin{minipage}[t]{2.2cm}
\begin{picture}(1.4,1.8)
\leavevmode
\epsfxsize=1.4cm
\epsffile{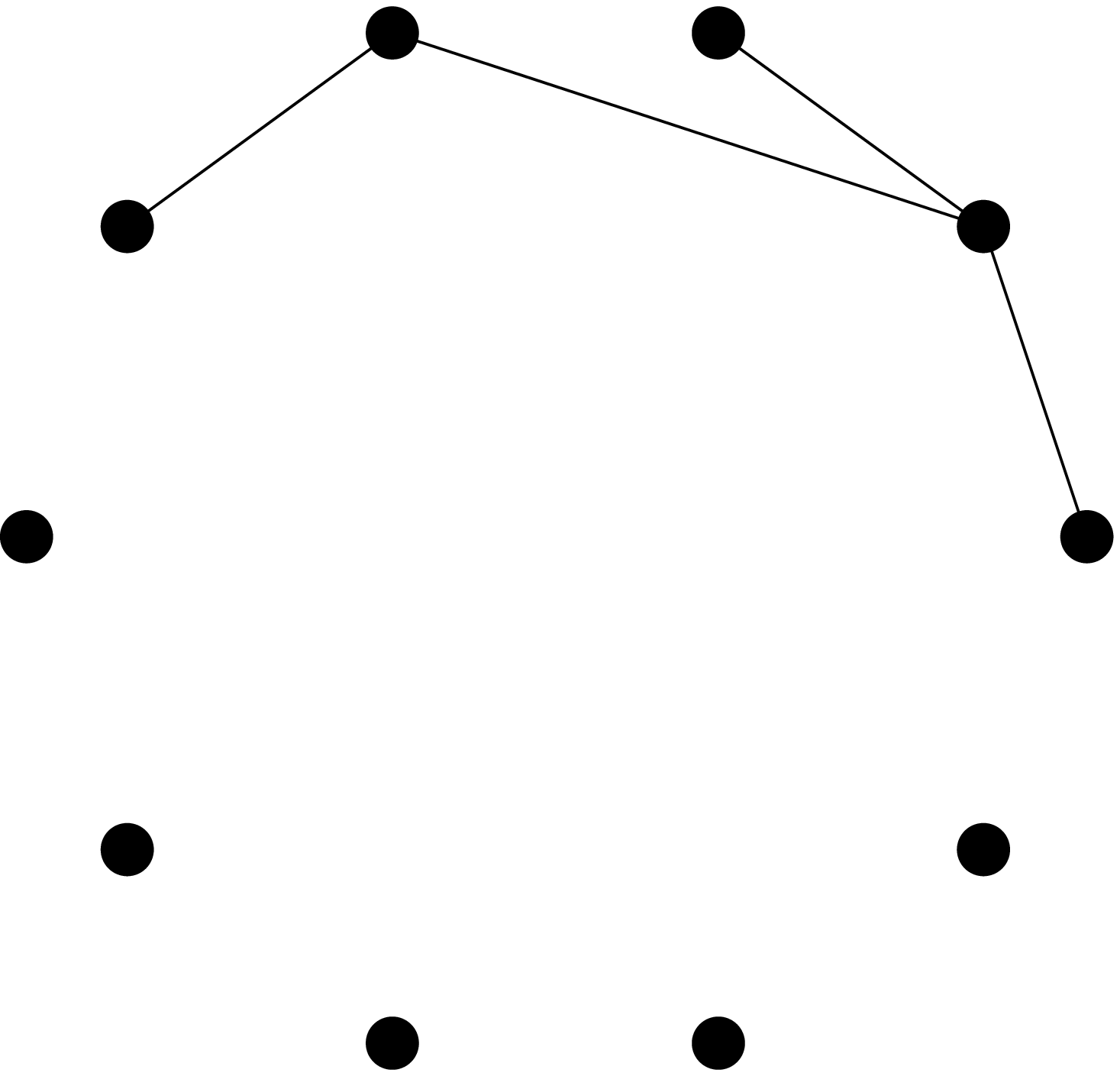}
\end{picture}\par
\end{minipage}
\begin{minipage}[t]{2.2cm}
\begin{picture}(1.4,1.8)
\leavevmode
\epsfxsize=1.4cm
\epsffile{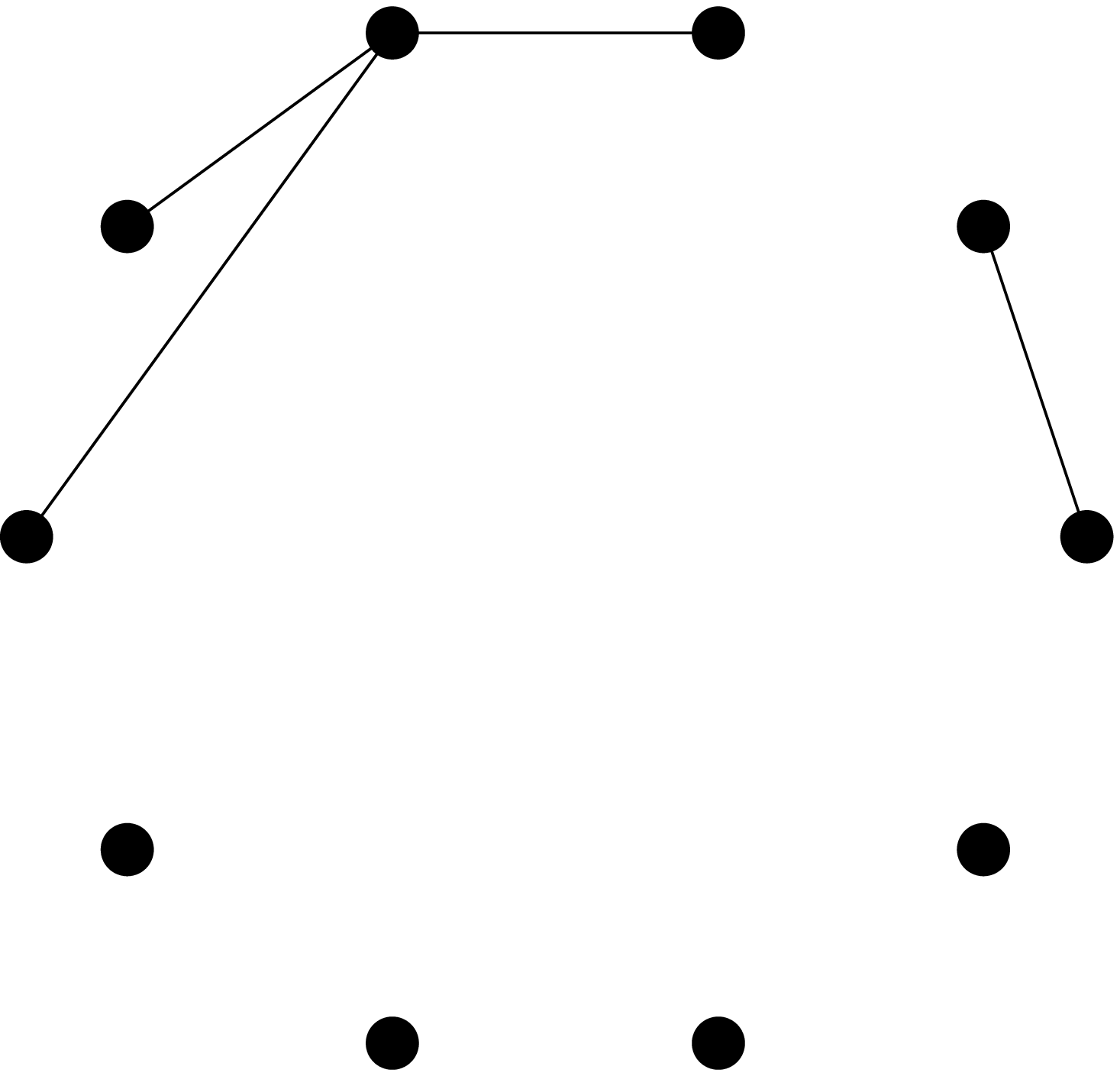}
\end{picture}\par
\end{minipage}

\medskip

Then $R(K_3,H) = 31$. These are possibly not \textbf{all} graphs with $R(K_3,H) = 31$.

\bigskip
\hrule
\smallskip
\hrule
\bigskip

\newpage

$R(K_3,H) = 30$ if and only if $H^c$ is contained in:

\setlength{\unitlength}{1cm}
\begin{minipage}[t]{2.0cm}
\begin{picture}(1.4,1.8)
\leavevmode
\epsfxsize=1.4cm
\epsffile{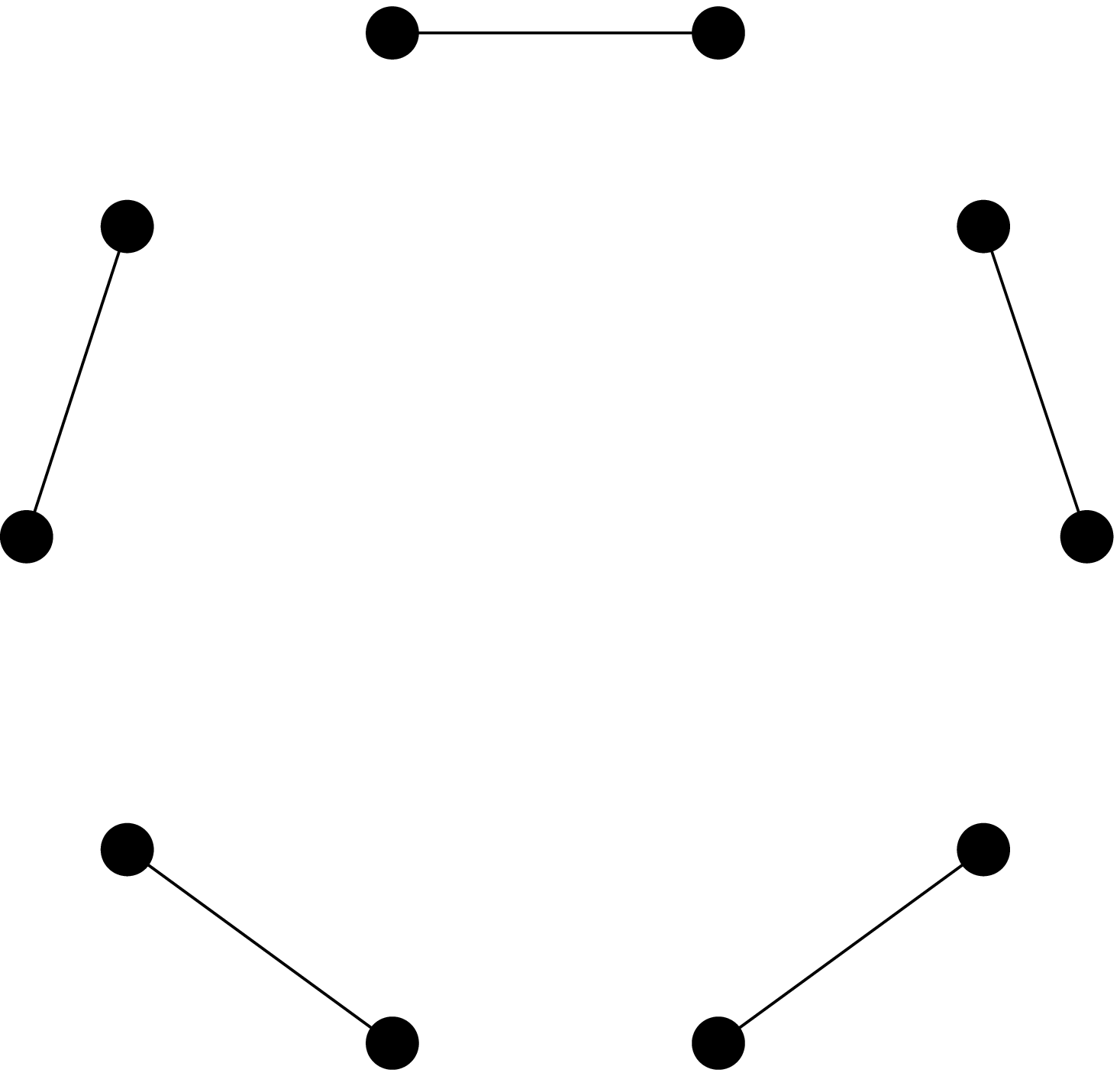}
\end{picture}\par
\end{minipage}

\bigskip
and contains:

\setlength{\unitlength}{1cm}
\begin{minipage}[t]{2.2cm}
\begin{picture}(1.4,1.8)
\leavevmode
\epsfxsize=1.4cm
\epsffile{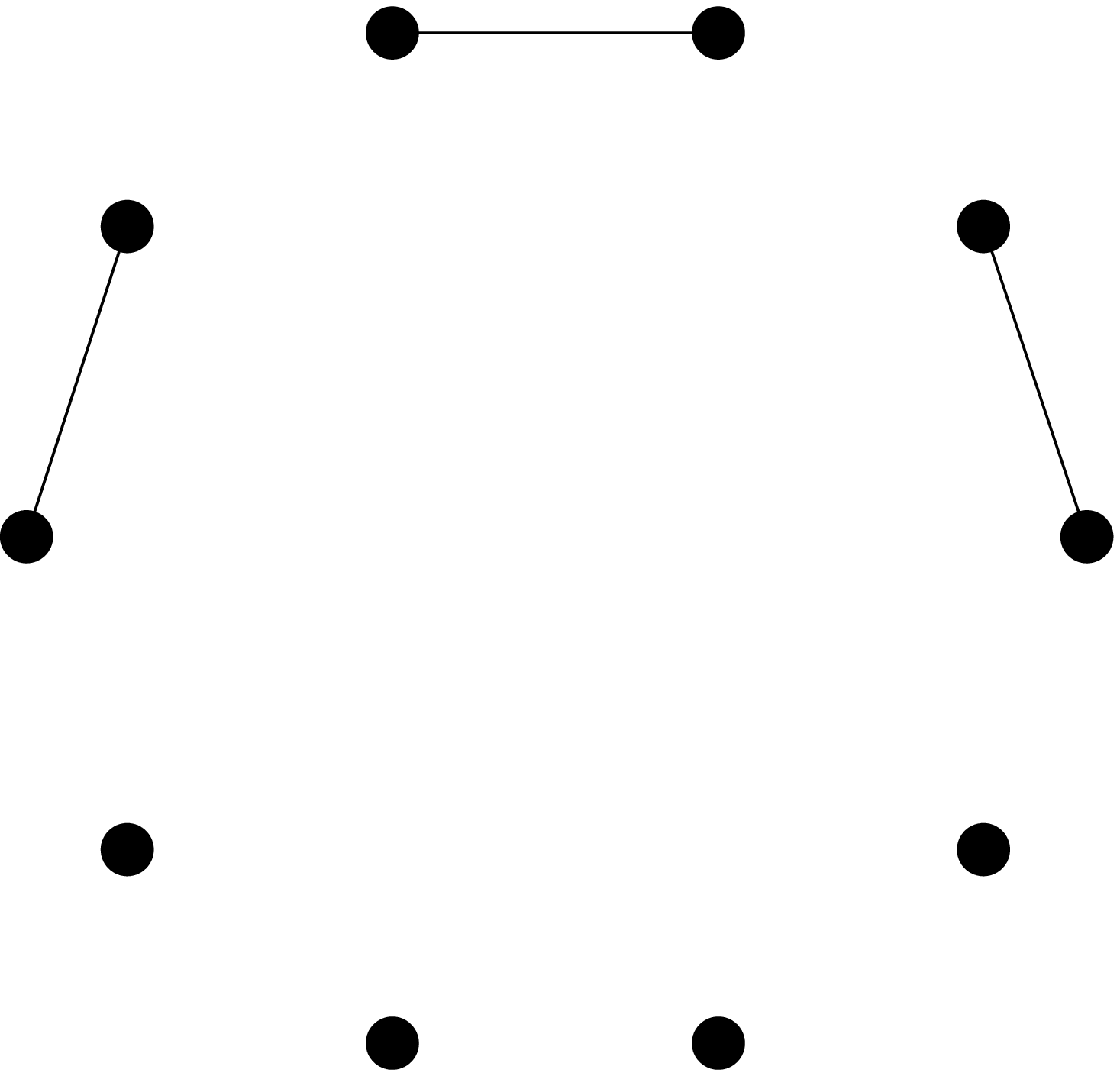}
\end{picture}\par
\end{minipage}

\bigskip
\hrule
\bigskip

$R(K_3,H) = 29$ if and only if $H^c$ is contained in one of the graphs:

\setlength{\unitlength}{1cm}
\begin{minipage}[t]{2.2cm}
\begin{picture}(1.4,1.8)
\leavevmode
\epsfxsize=1.4cm
\epsffile{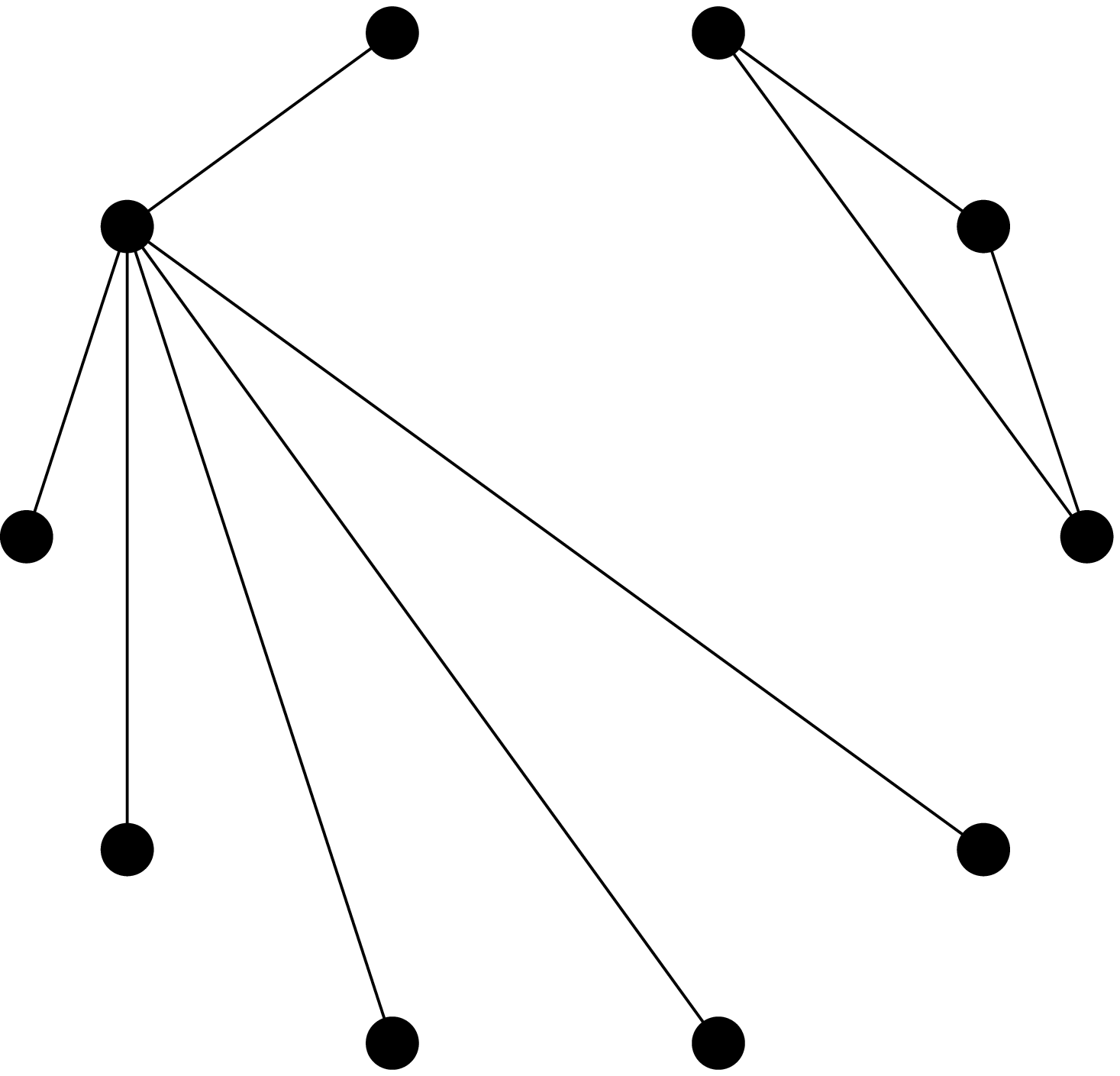}
\end{picture}\par
\end{minipage}
\begin{minipage}[t]{2.2cm}
\begin{picture}(1.4,1.8)
\leavevmode
\epsfxsize=1.4cm
\epsffile{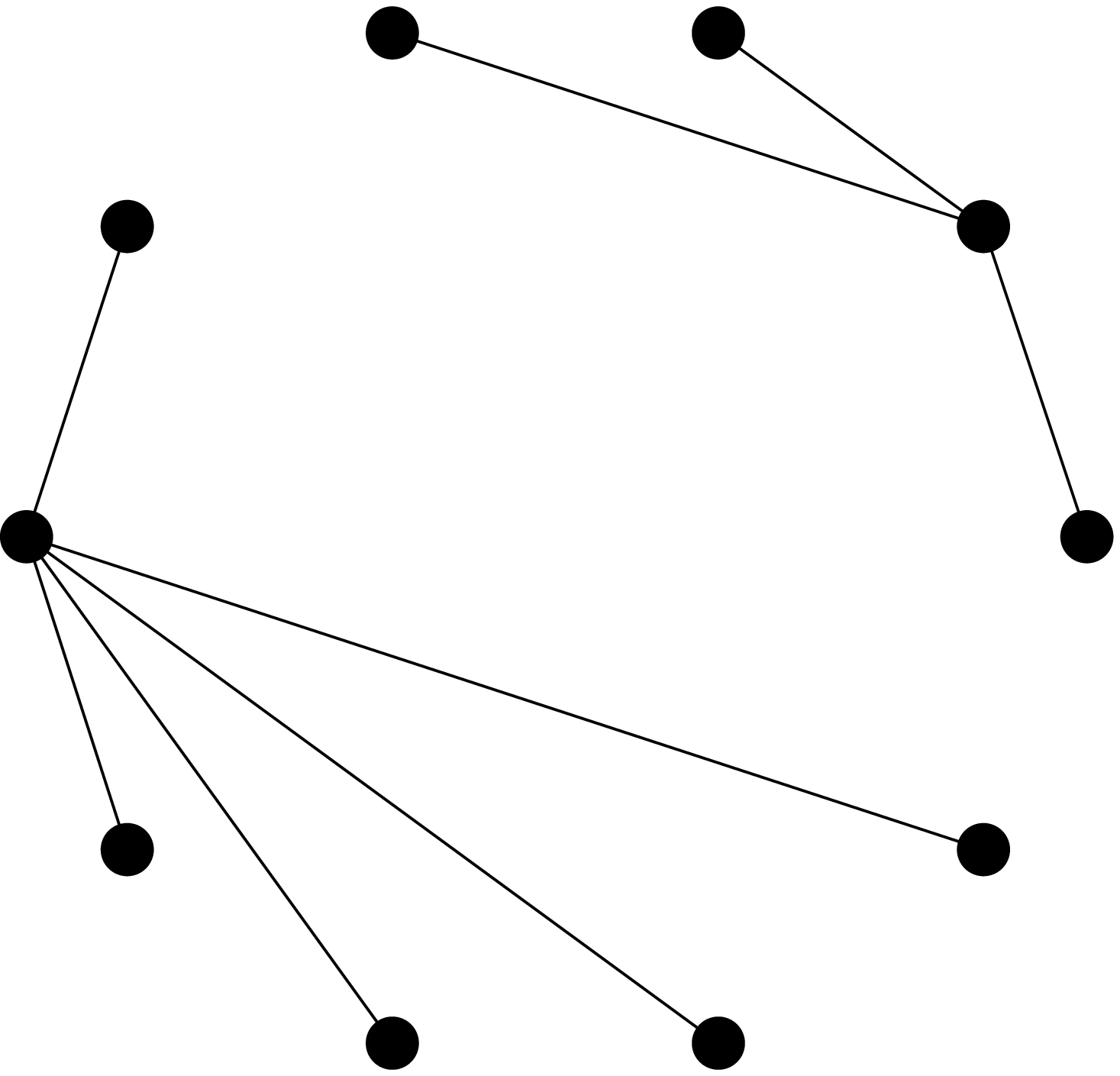}
\end{picture}\par
\end{minipage}
\begin{minipage}[t]{2.2cm}
\begin{picture}(1.4,1.8)
\leavevmode
\epsfxsize=1.4cm
\epsffile{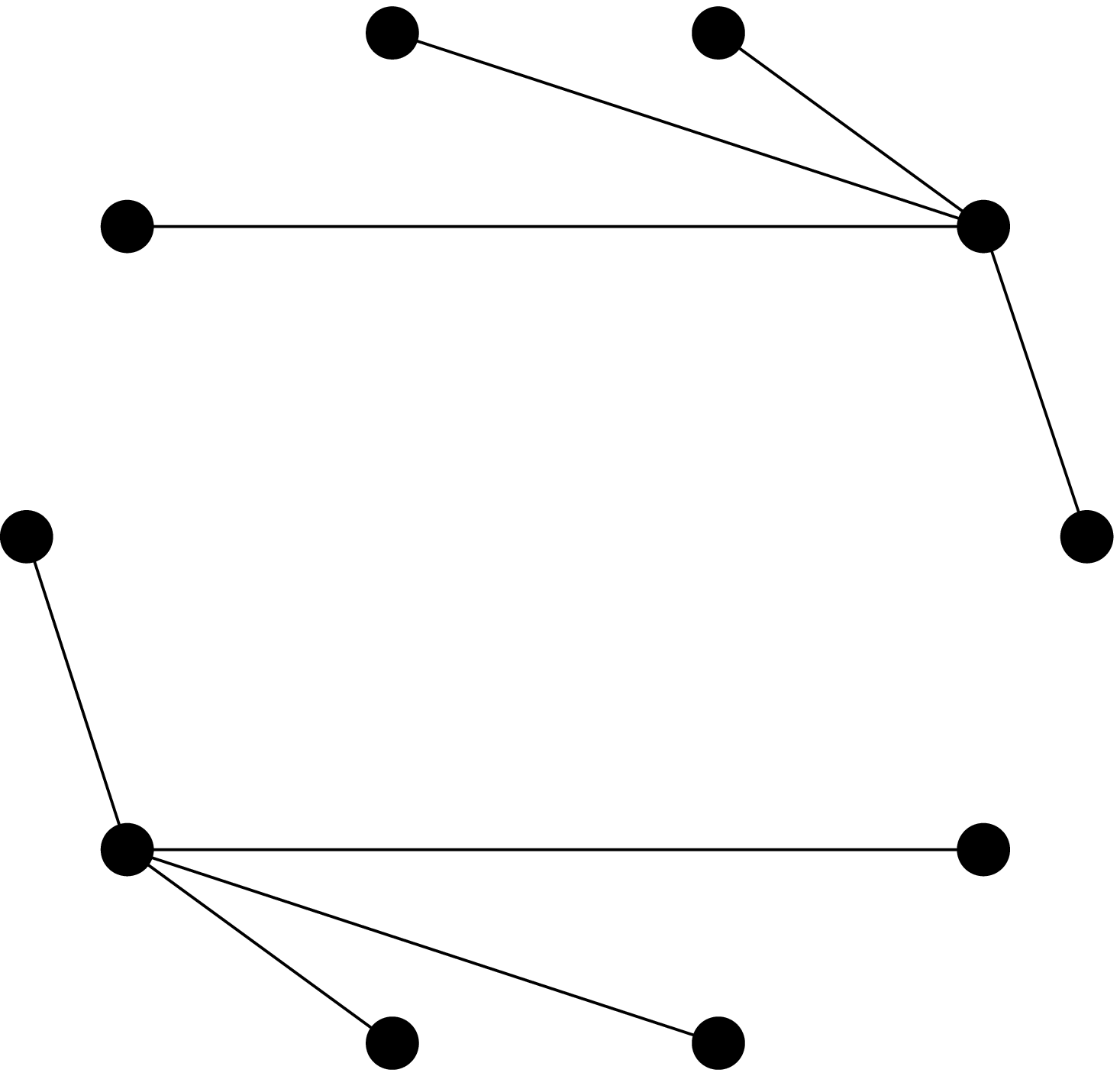}
\end{picture}\par
\end{minipage}
\begin{minipage}[t]{2.2cm}
\begin{picture}(1.4,1.8)
\leavevmode
\epsfxsize=1.4cm
\epsffile{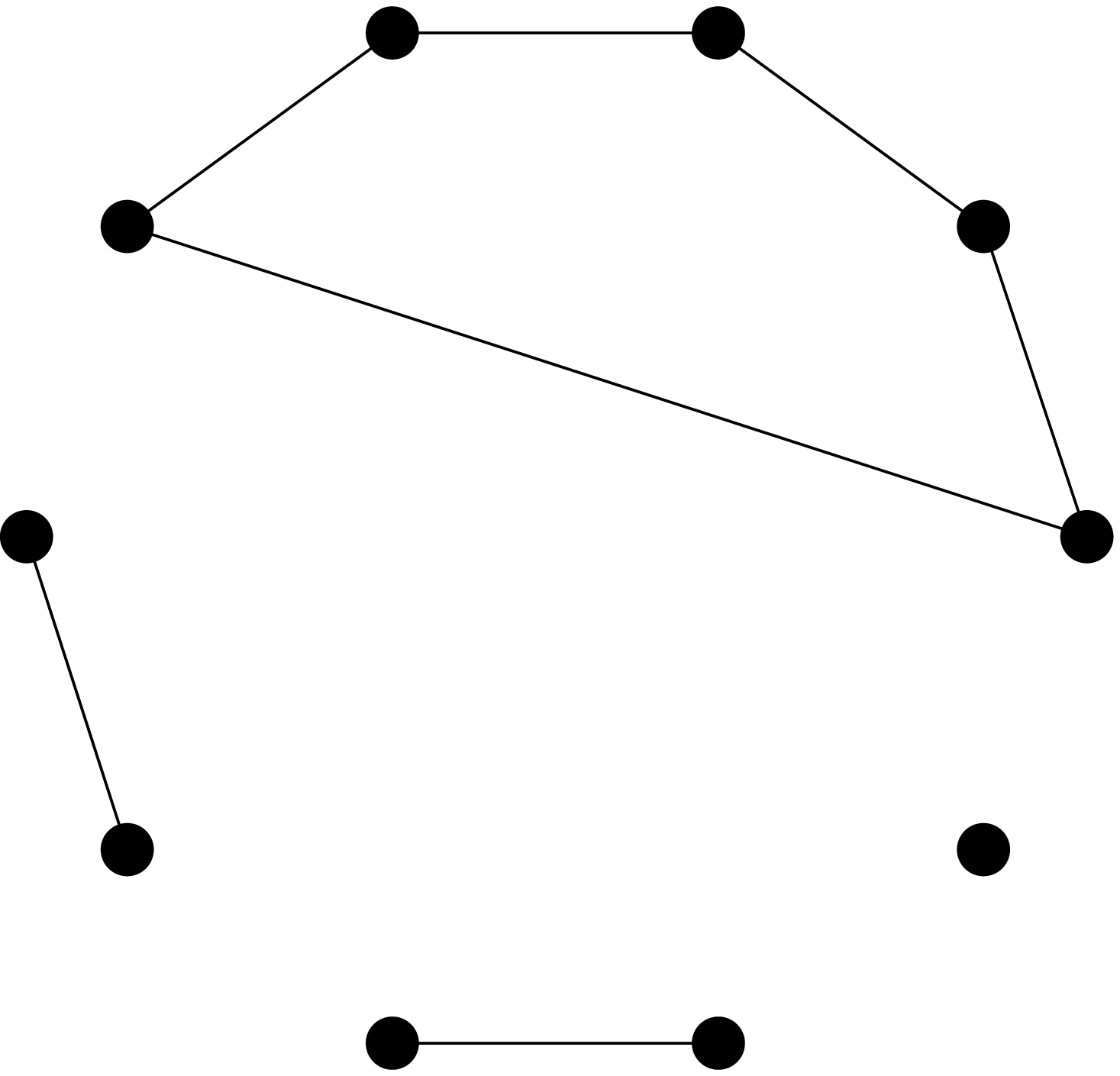}
\end{picture}\par
\end{minipage}
\begin{minipage}[t]{2.2cm}
\begin{picture}(1.4,1.8)
\leavevmode
\epsfxsize=1.4cm
\epsffile{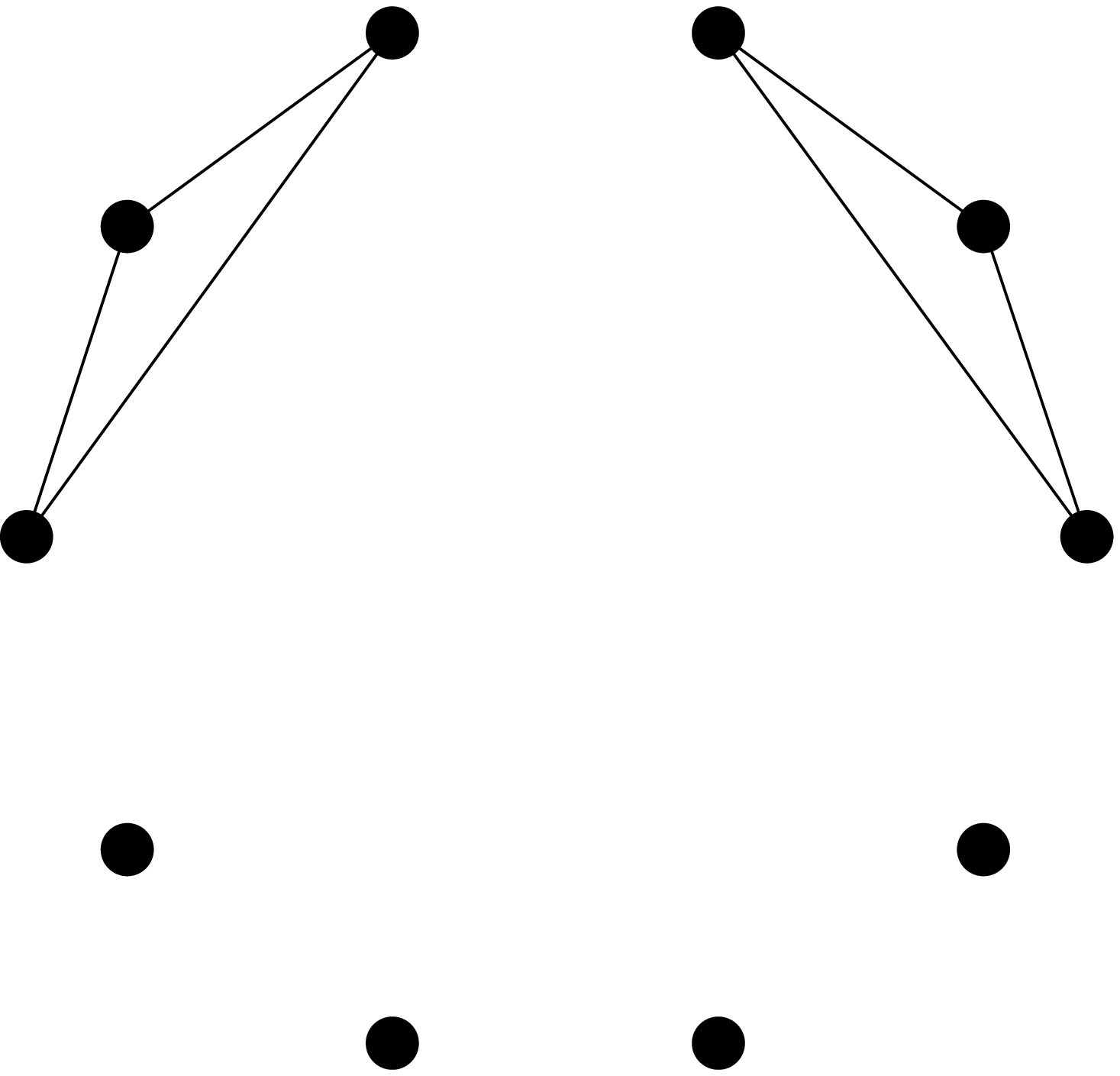}
\end{picture}\par
\end{minipage}
\begin{minipage}[t]{2.2cm}
\begin{picture}(1.4,1.8)
\leavevmode
\epsfxsize=1.4cm
\epsffile{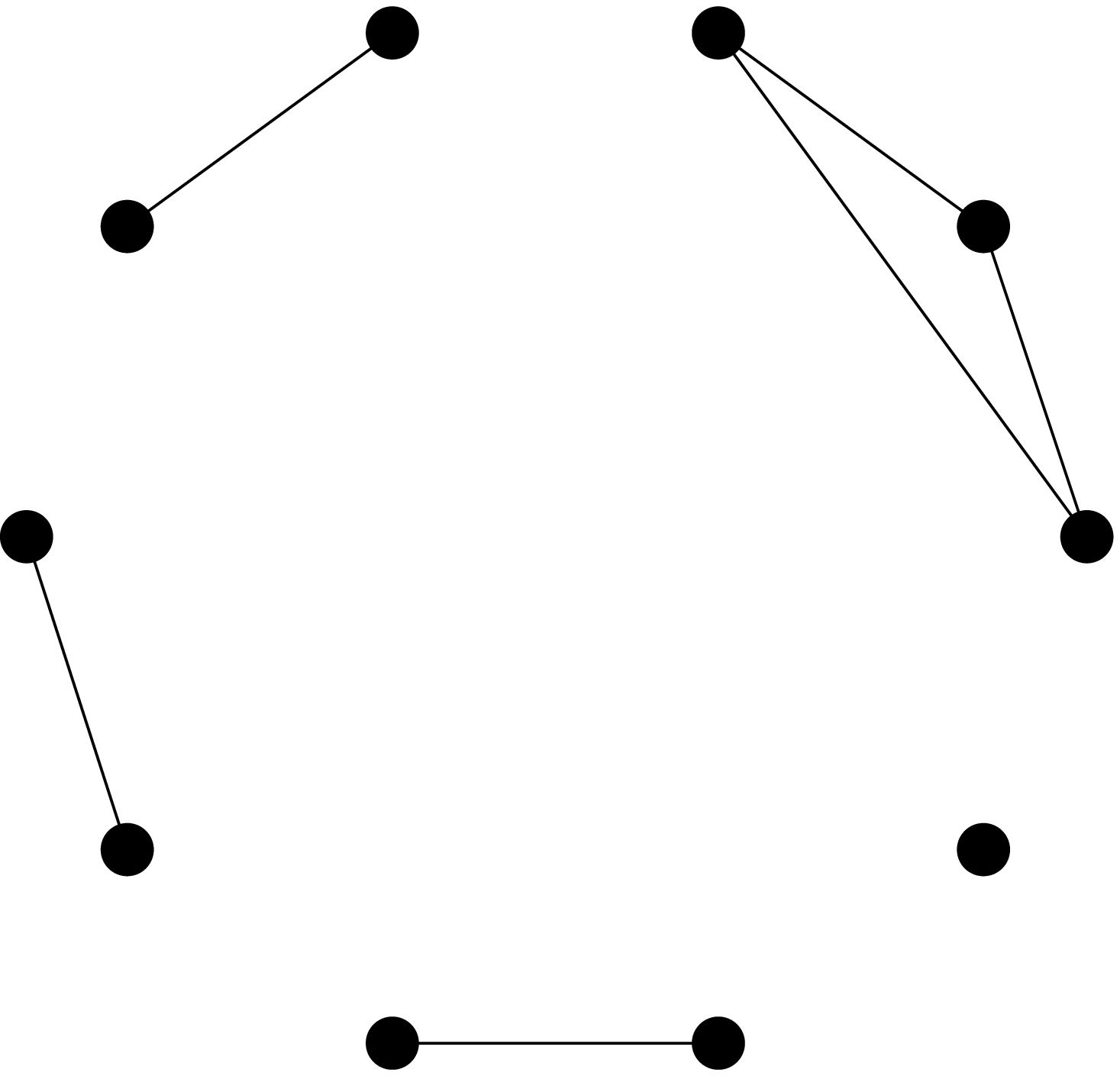}
\end{picture}\par
\end{minipage}
\begin{minipage}[t]{2.2cm}
\begin{picture}(1.4,1.8)
\leavevmode
\epsfxsize=1.4cm
\epsffile{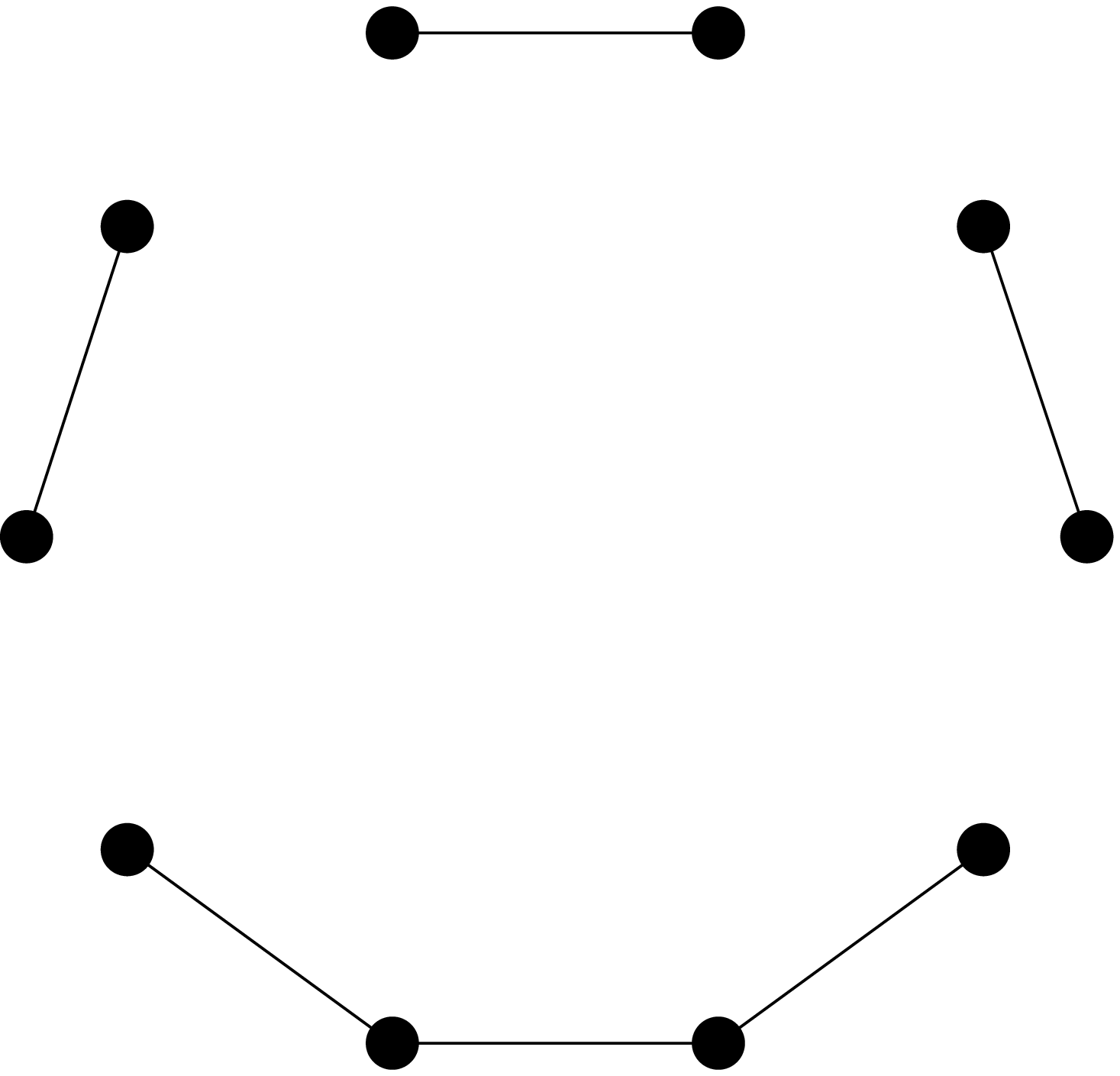}
\end{picture}\par
\end{minipage}
\begin{minipage}[t]{2.2cm}
\begin{picture}(1.4,1.8)
\leavevmode
\epsfxsize=1.4cm
\epsffile{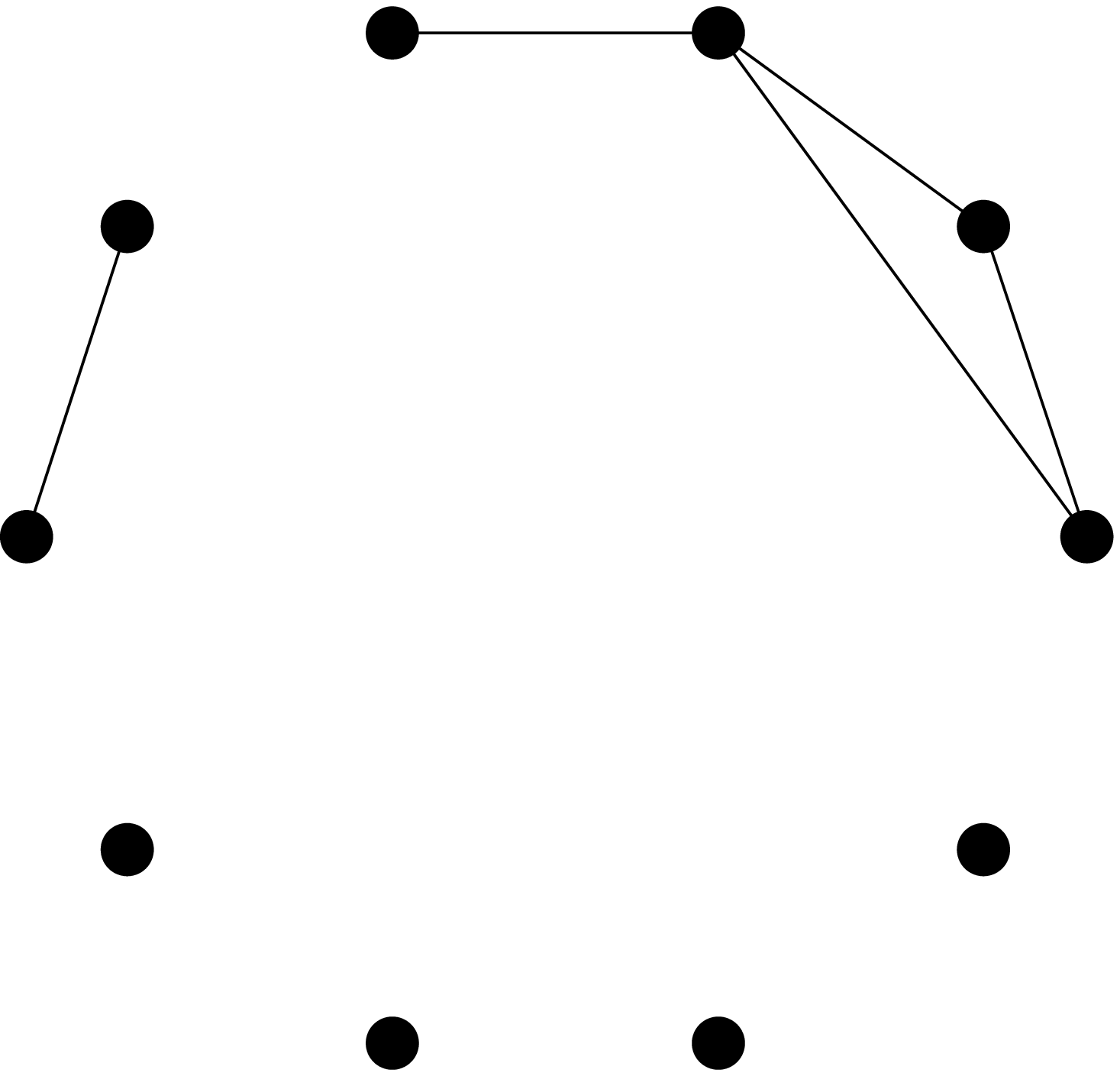}
\end{picture}\par
\end{minipage}
\begin{minipage}[t]{2.2cm}
\begin{picture}(1.4,1.8)
\leavevmode
\epsfxsize=1.4cm
\epsffile{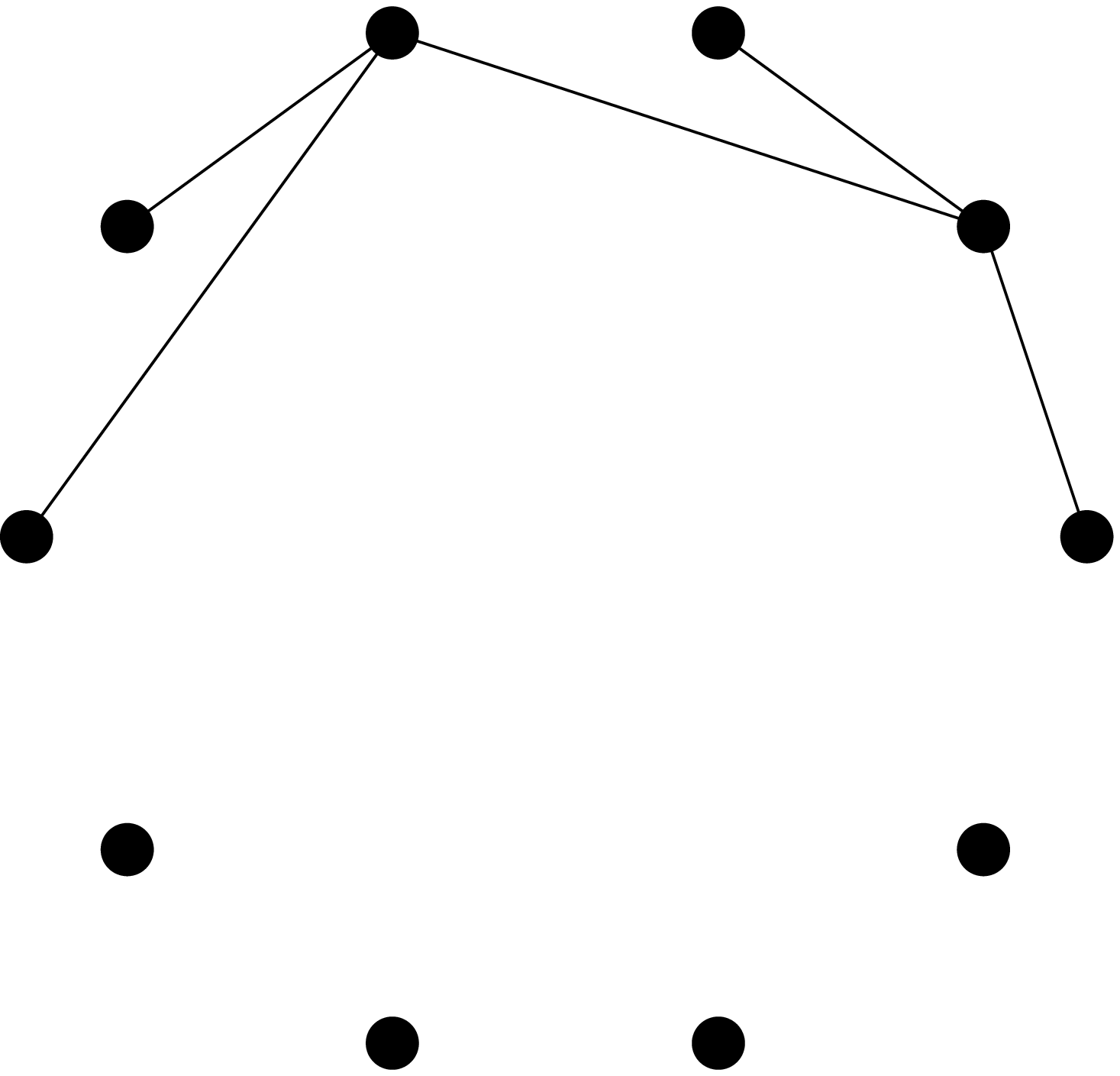}
\end{picture}\par
\end{minipage}

\bigskip
and contains one of the graphs:

\setlength{\unitlength}{1cm}
\begin{minipage}[t]{2.2cm}
\begin{picture}(1.4,1.8)
\leavevmode
\epsfxsize=1.4cm
\epsffile{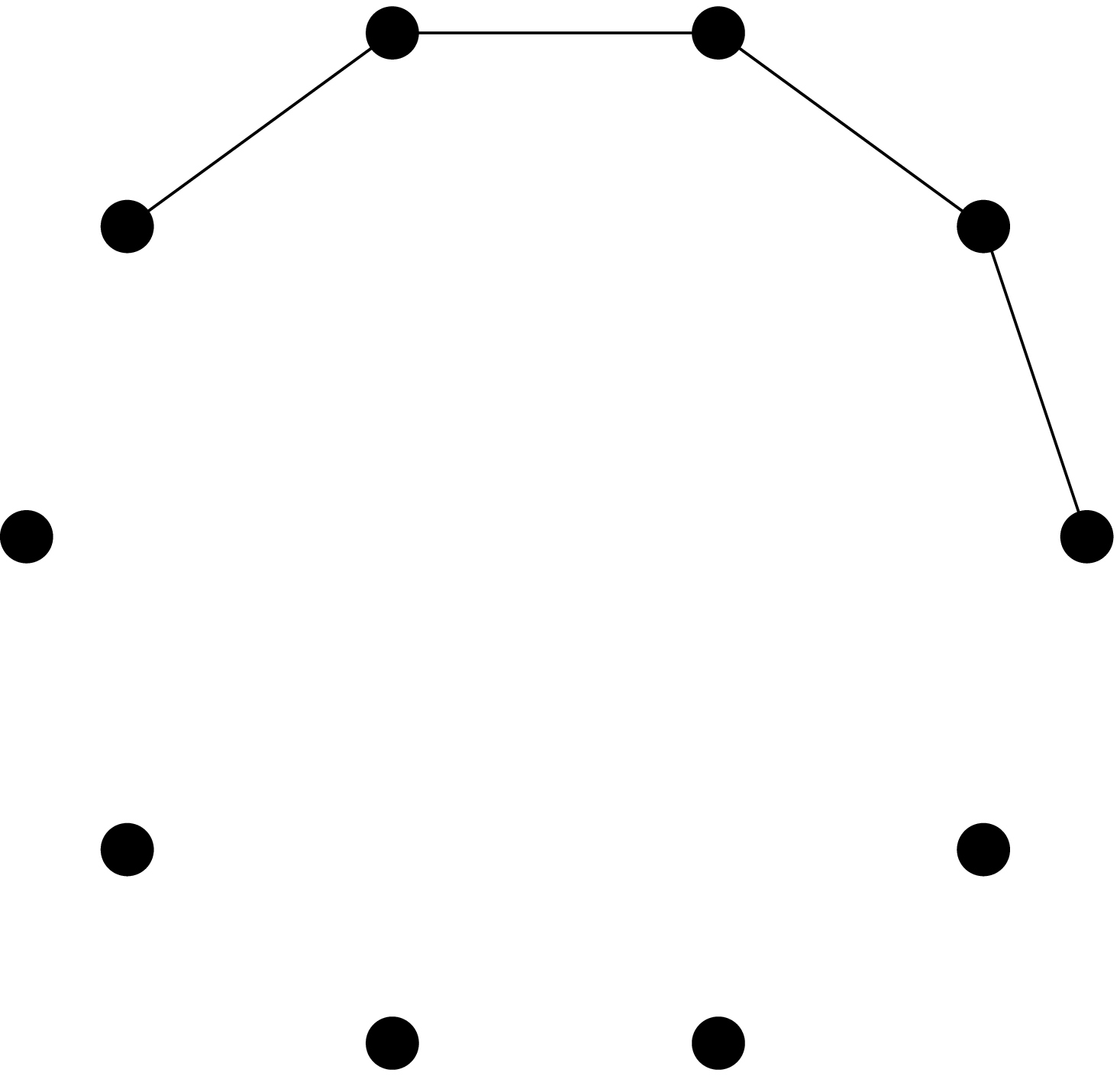}
\end{picture}\par
\end{minipage}
\begin{minipage}[t]{2.2cm}
\begin{picture}(1.4,1.8)
\leavevmode
\epsfxsize=1.4cm
\epsffile{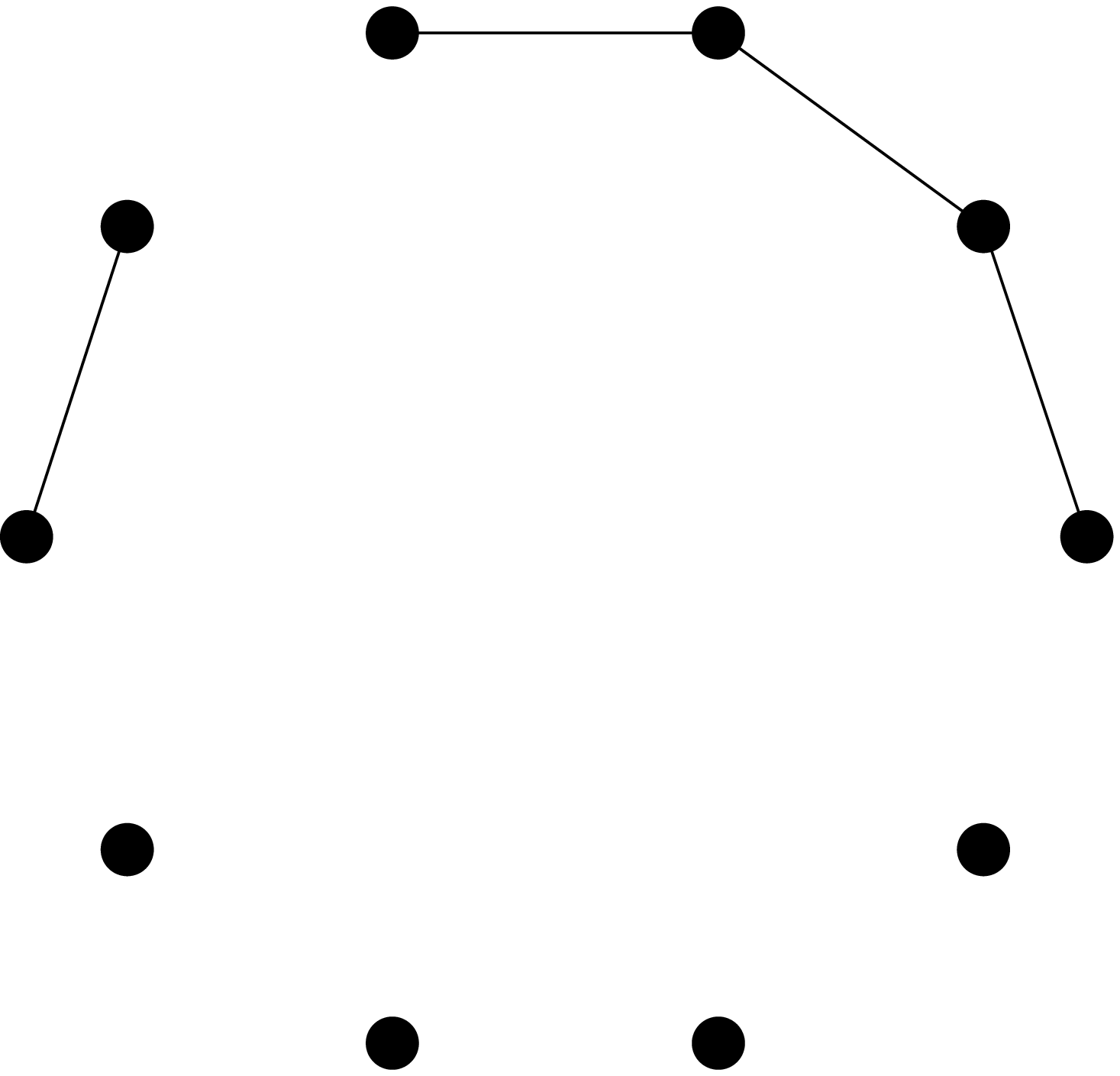}
\end{picture}\par
\end{minipage}
\begin{minipage}[t]{2.2cm}
\begin{picture}(1.4,1.8)
\leavevmode
\epsfxsize=1.4cm
\epsffile{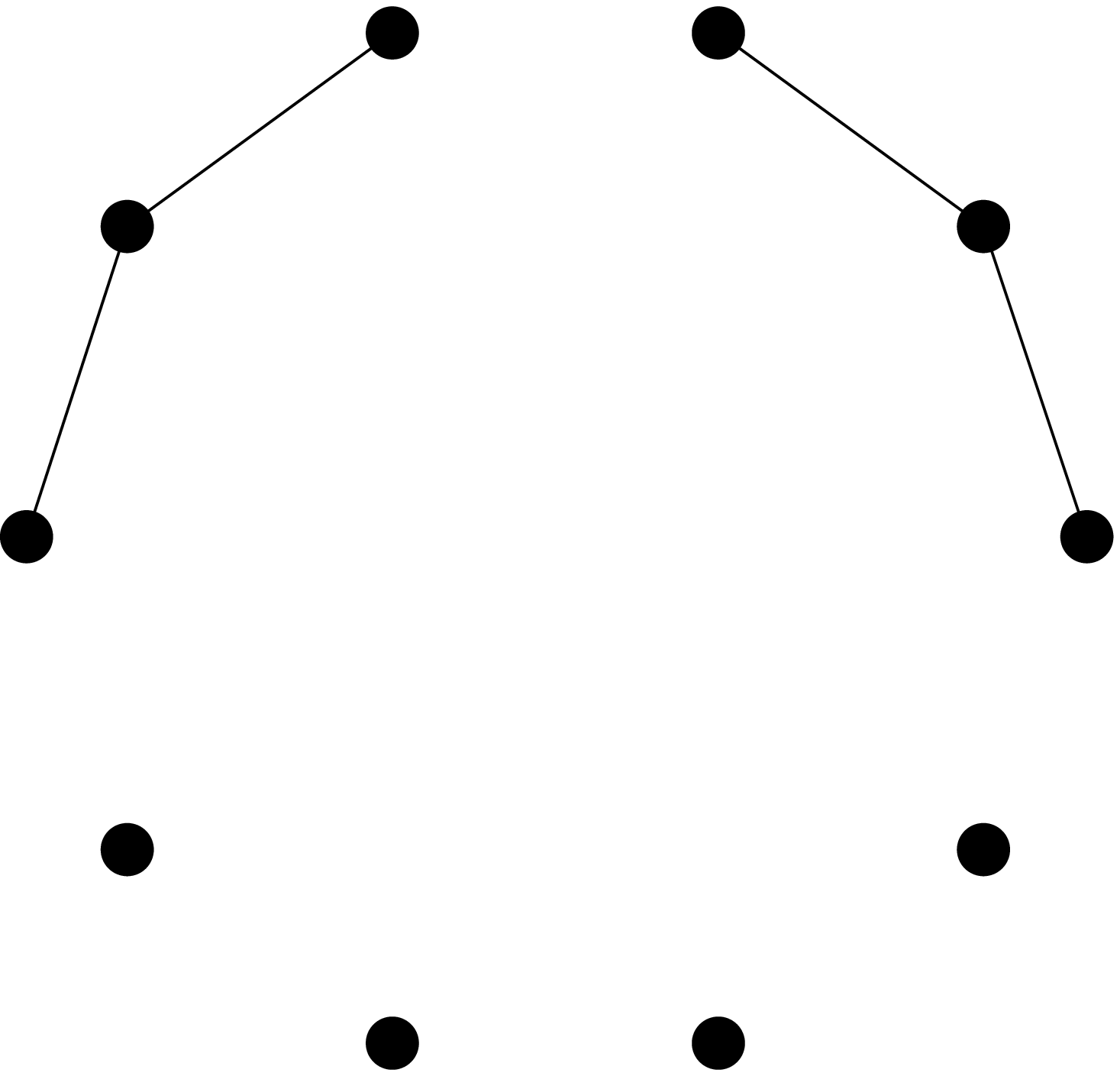}
\end{picture}\par
\end{minipage}
\begin{minipage}[t]{2.2cm}
\begin{picture}(1.4,1.8)
\leavevmode
\epsfxsize=1.4cm
\epsffile{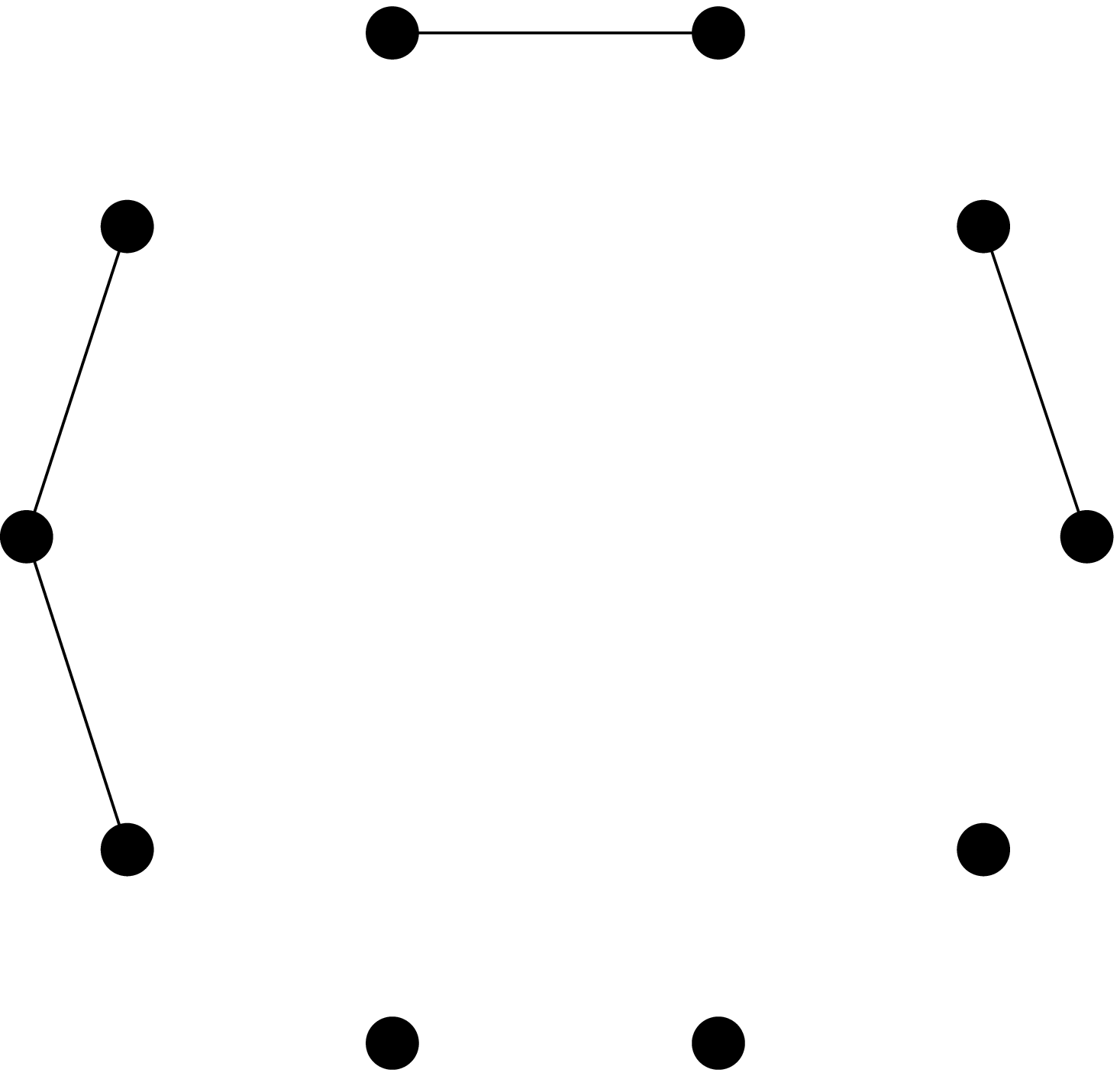}
\end{picture}\par
\end{minipage}

\bigskip
\hrule
\bigskip

$R(K_3,H) = 28$ if and only if $H^c$ is contained in one of the graphs:

\setlength{\unitlength}{1cm}
\begin{minipage}[t]{2.2cm}
\begin{picture}(1.4,1.8)
\leavevmode
\epsfxsize=1.4cm
\epsffile{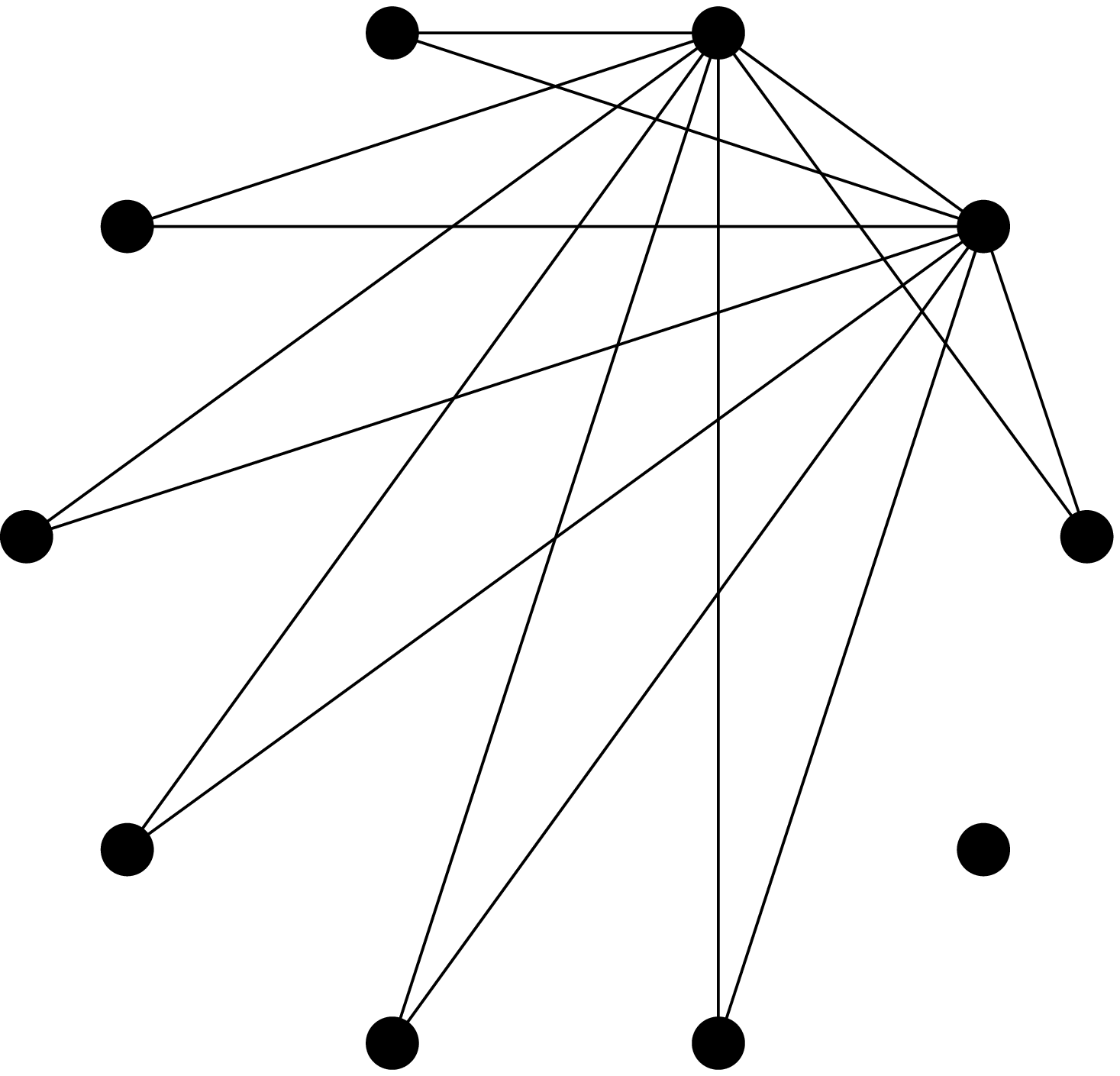}
\end{picture}\par
\end{minipage}
\begin{minipage}[t]{2.2cm}
\begin{picture}(1.4,1.8)
\leavevmode
\epsfxsize=1.4cm
\epsffile{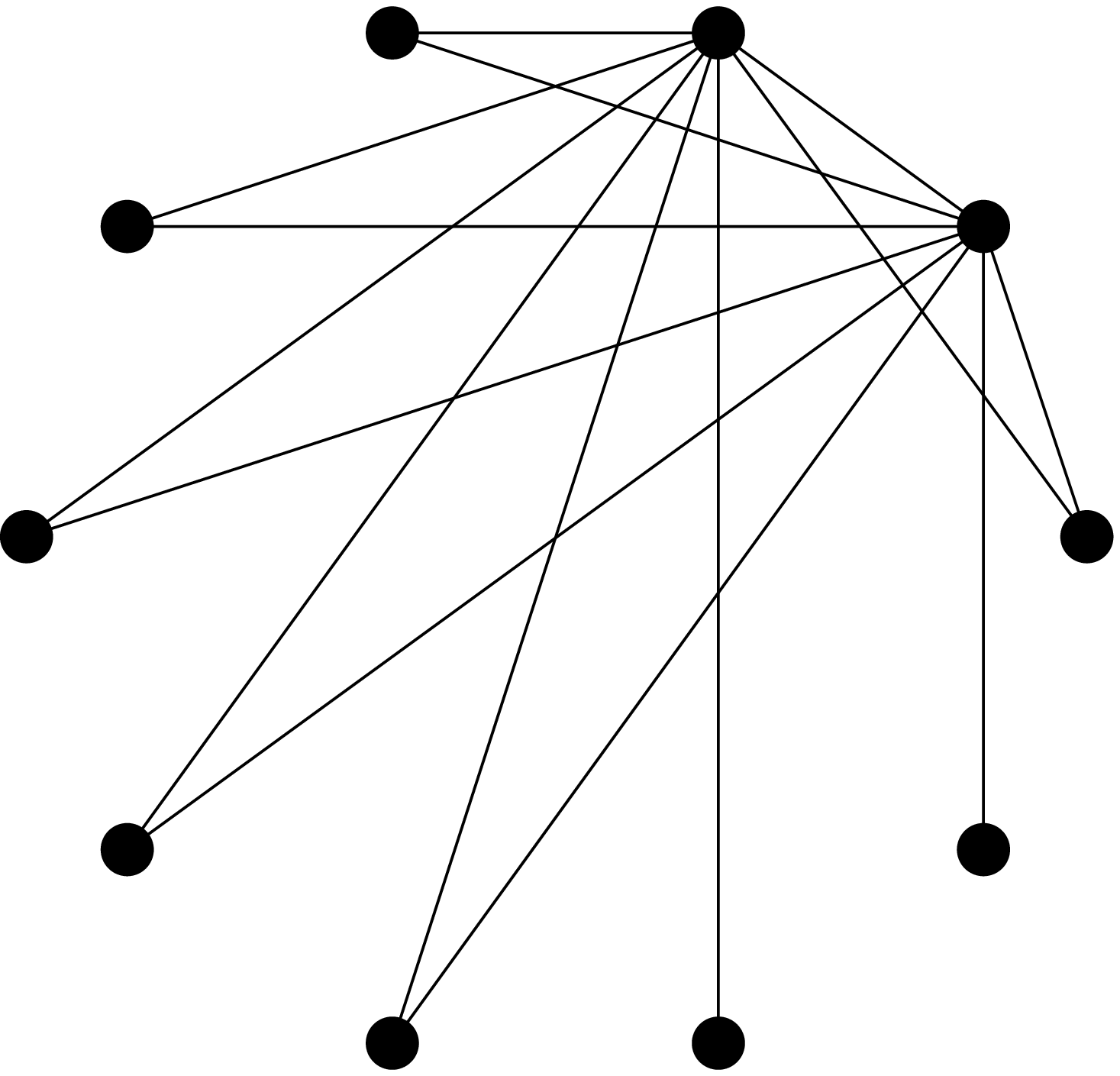}
\end{picture}\par
\end{minipage}
\begin{minipage}[t]{2.2cm}
\begin{picture}(1.4,1.8)
\leavevmode
\epsfxsize=1.4cm
\epsffile{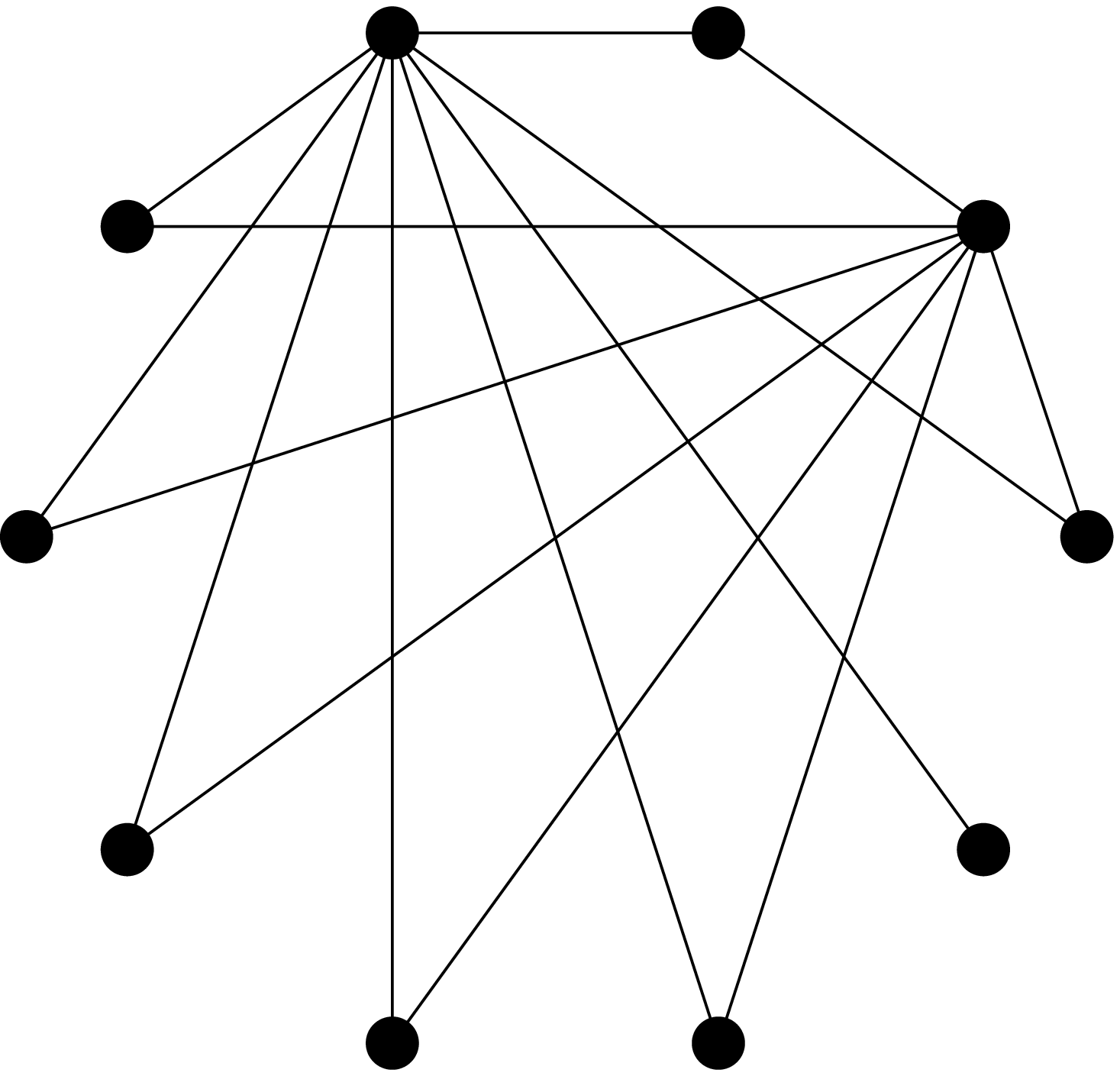}
\end{picture}\par
\end{minipage}
\begin{minipage}[t]{2.2cm}
\begin{picture}(1.4,1.8)
\leavevmode
\epsfxsize=1.4cm
\epsffile{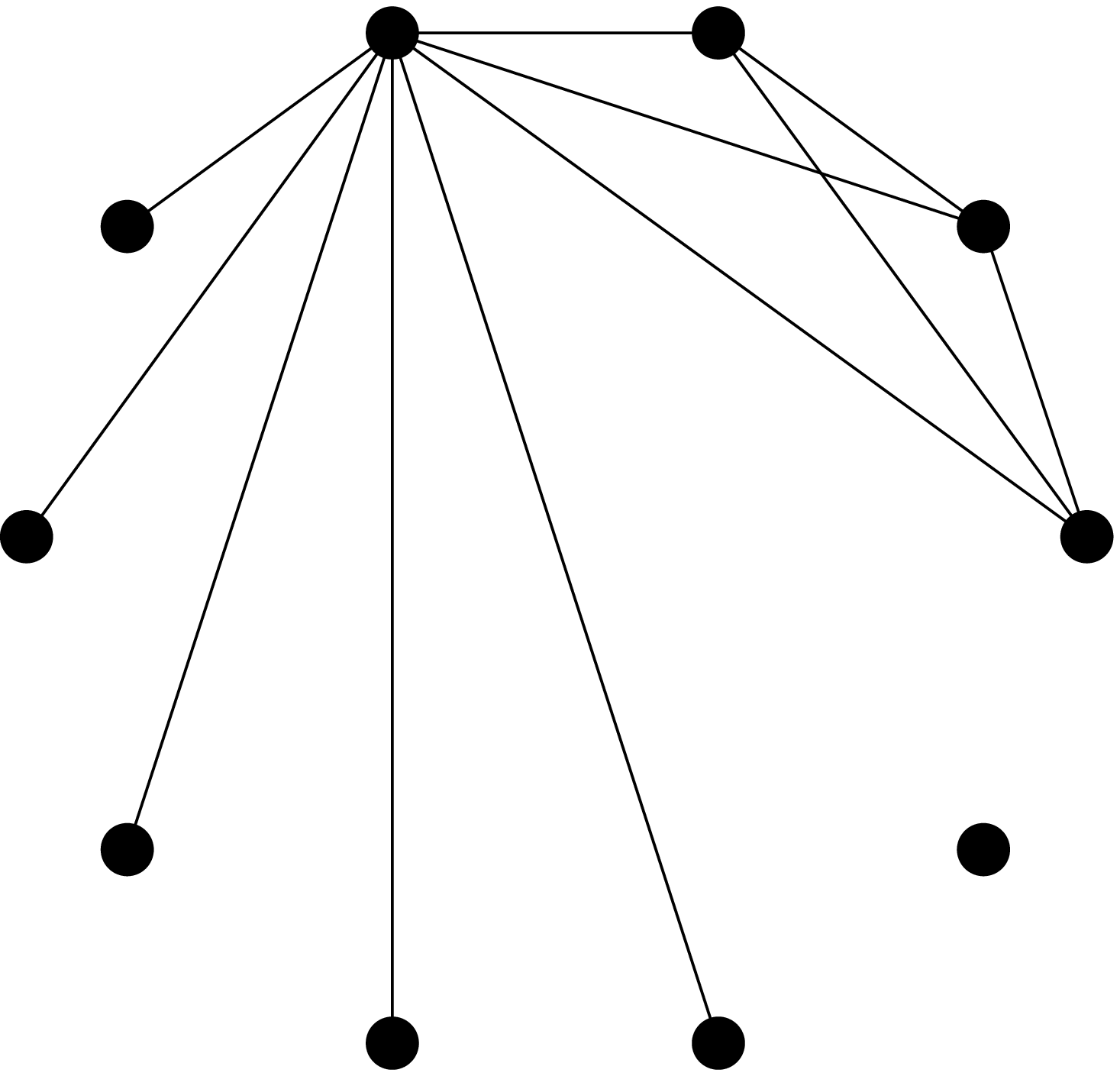}
\end{picture}\par
\end{minipage}
\begin{minipage}[t]{2.2cm}
\begin{picture}(1.4,1.8)
\leavevmode
\epsfxsize=1.4cm
\epsffile{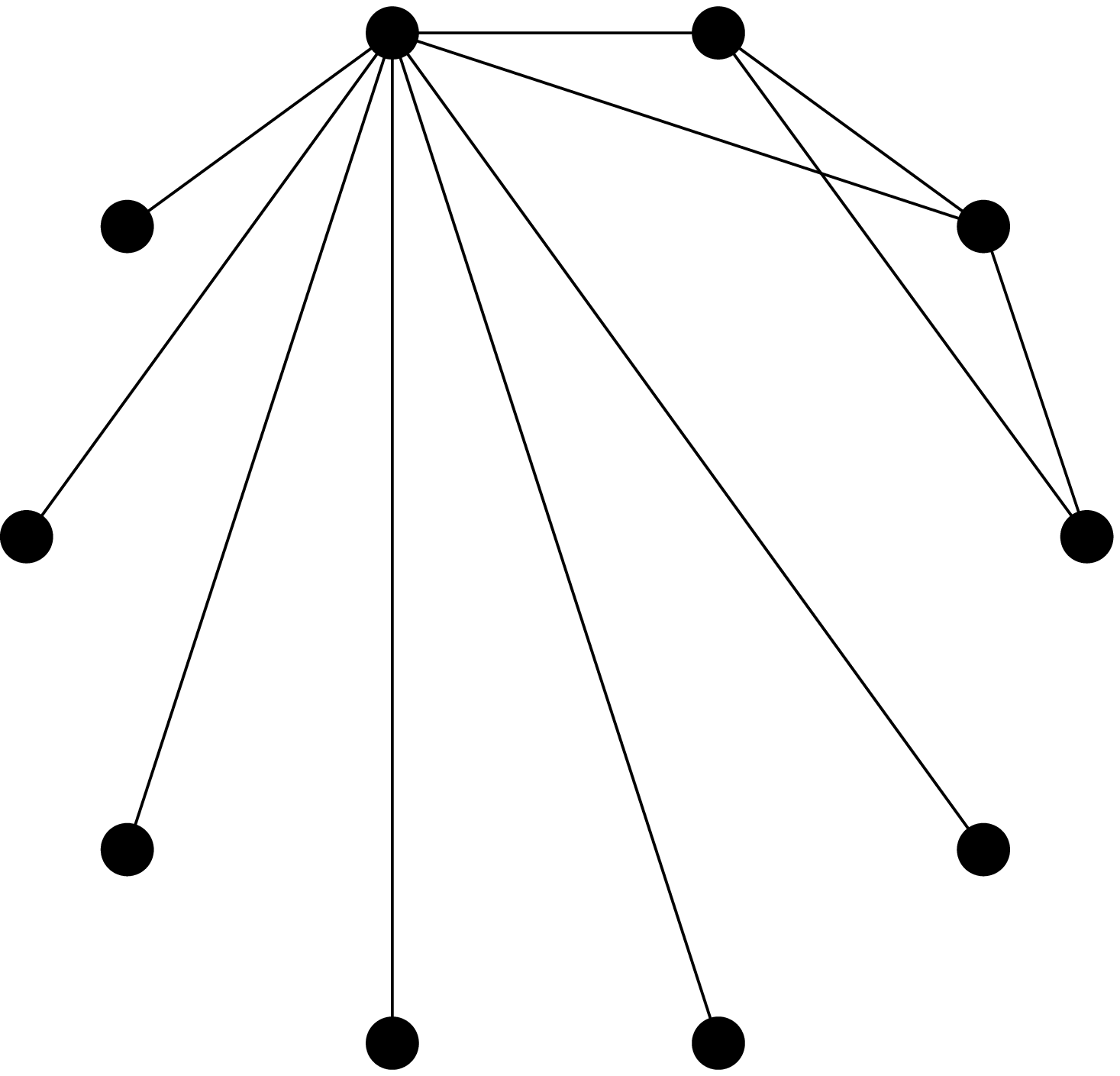}
\end{picture}\par
\end{minipage}
\begin{minipage}[t]{2.2cm}
\begin{picture}(1.4,1.8)
\leavevmode
\epsfxsize=1.4cm
\epsffile{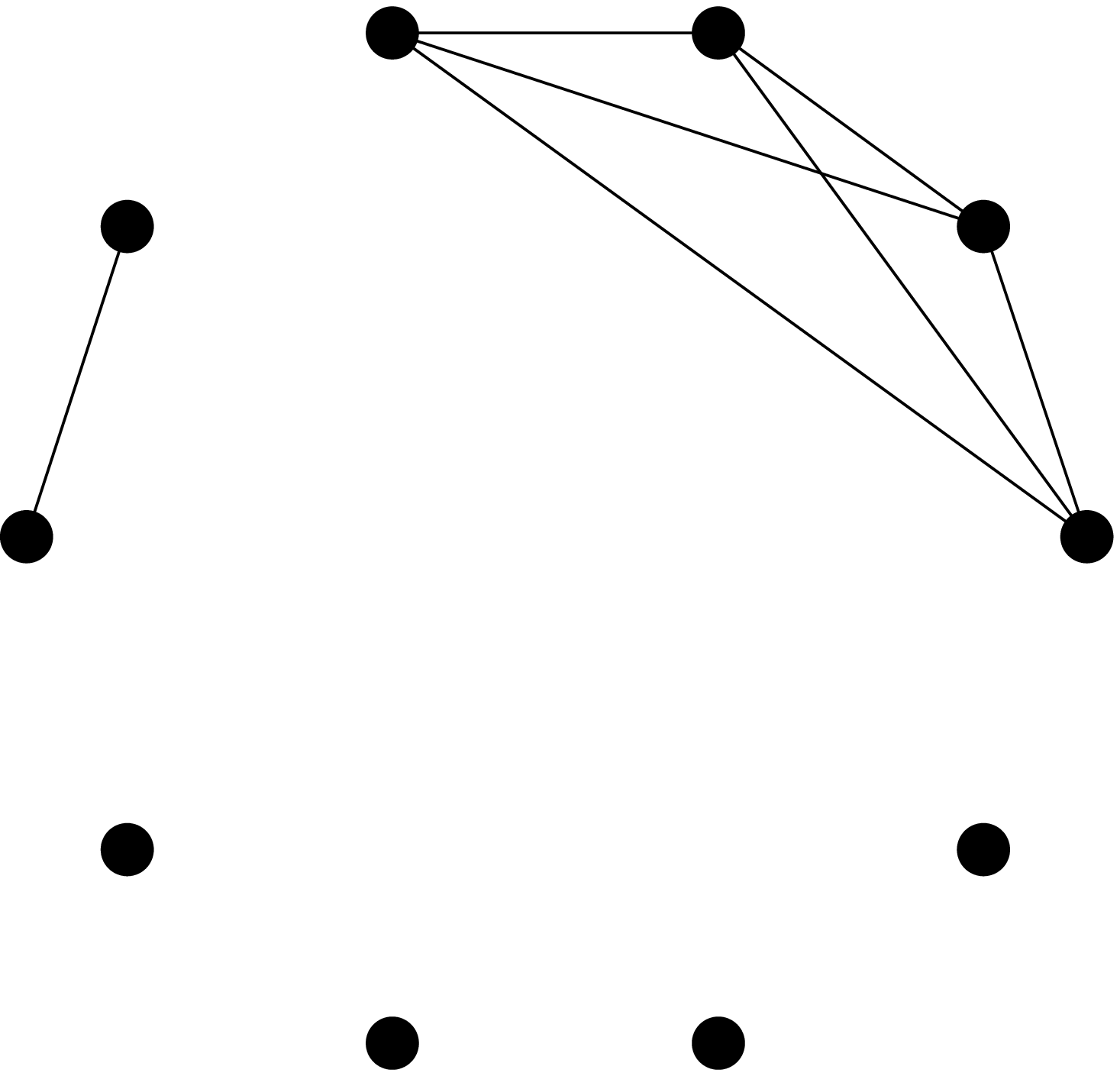}
\end{picture}\par
\end{minipage}
\begin{minipage}[t]{2.2cm}
\begin{picture}(1.4,1.8)
\leavevmode
\epsfxsize=1.4cm
\epsffile{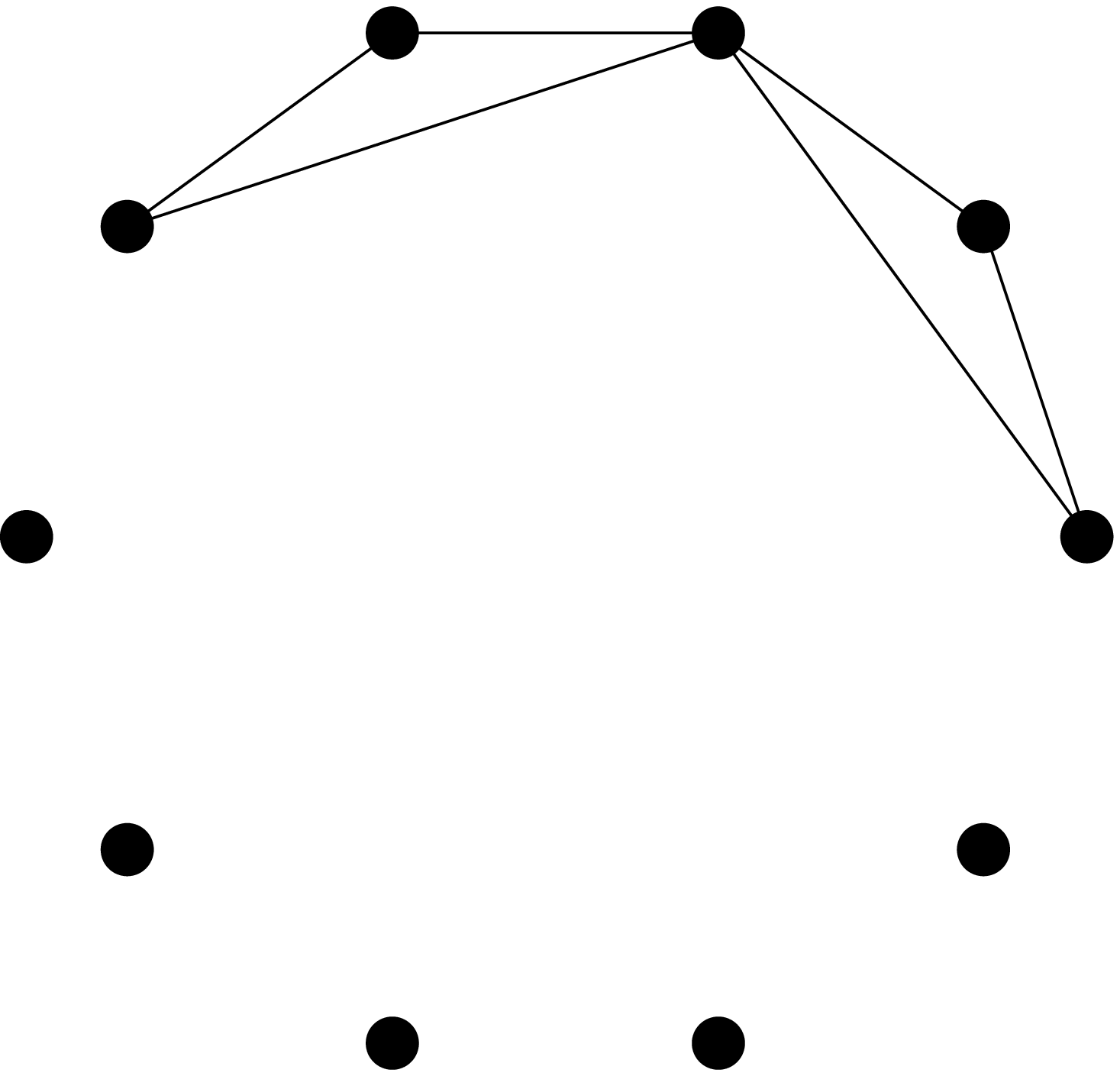}
\end{picture}\par
\end{minipage}

\bigskip
and contains one of the graphs:

\setlength{\unitlength}{1cm}
\begin{minipage}[t]{2.2cm}
\begin{picture}(1.4,1.8)
\leavevmode
\epsfxsize=1.4cm
\epsffile{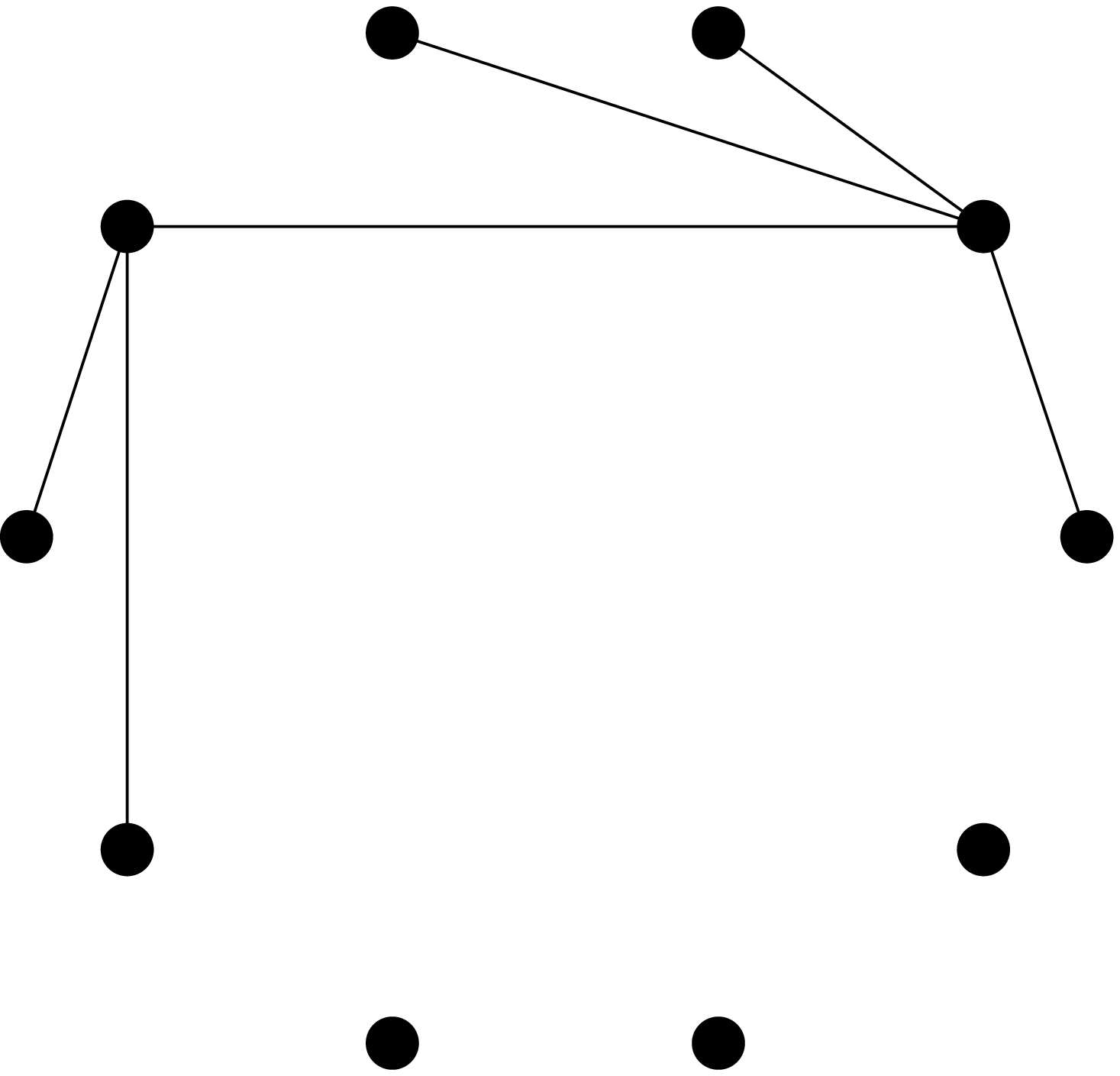}
\end{picture}\par
\end{minipage}
\begin{minipage}[t]{2.2cm}
\begin{picture}(1.4,1.8)
\leavevmode
\epsfxsize=1.4cm
\epsffile{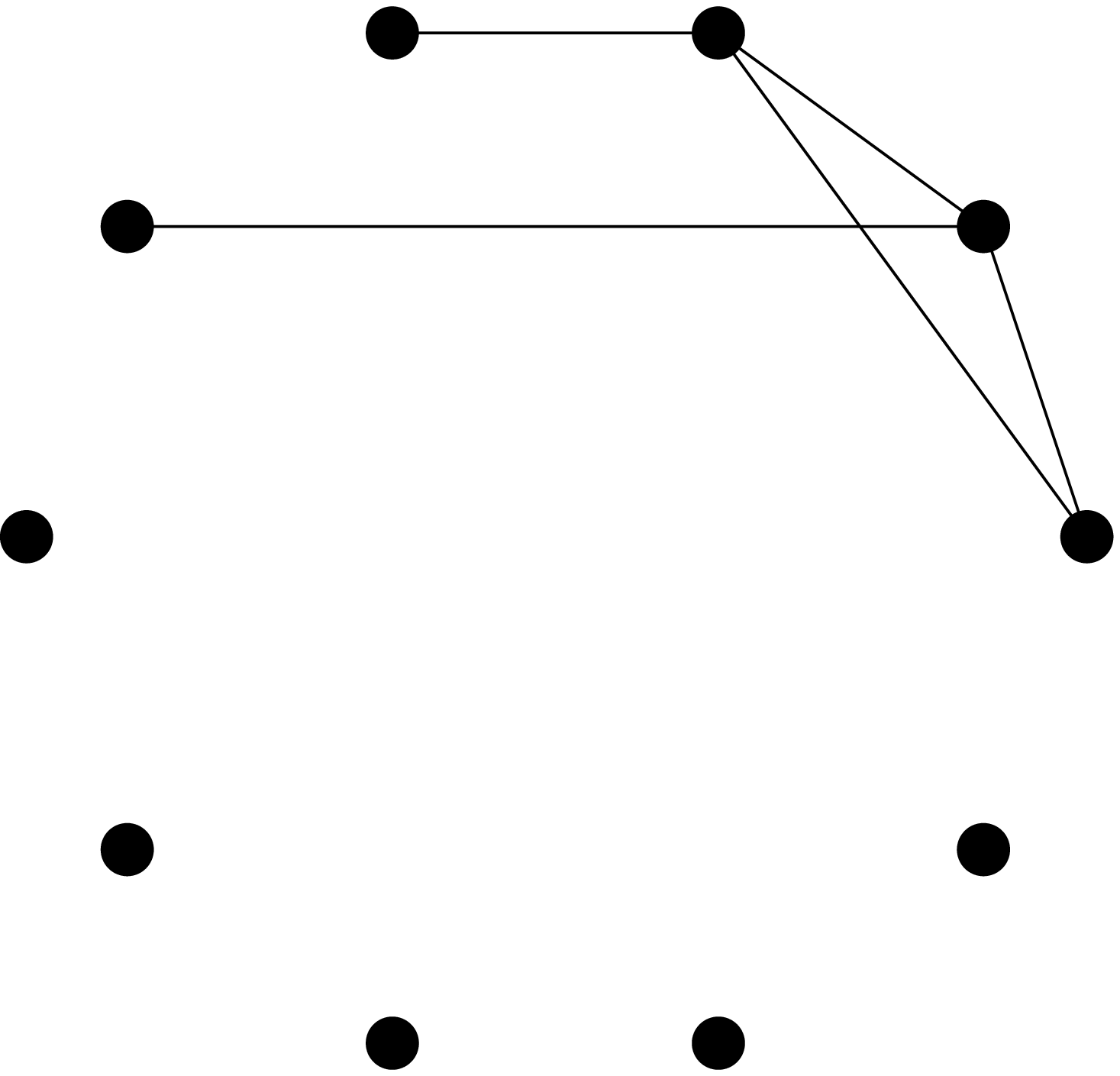}
\end{picture}\par
\end{minipage}
\begin{minipage}[t]{2.2cm}
\begin{picture}(1.4,1.8)
\leavevmode
\epsfxsize=1.4cm
\epsffile{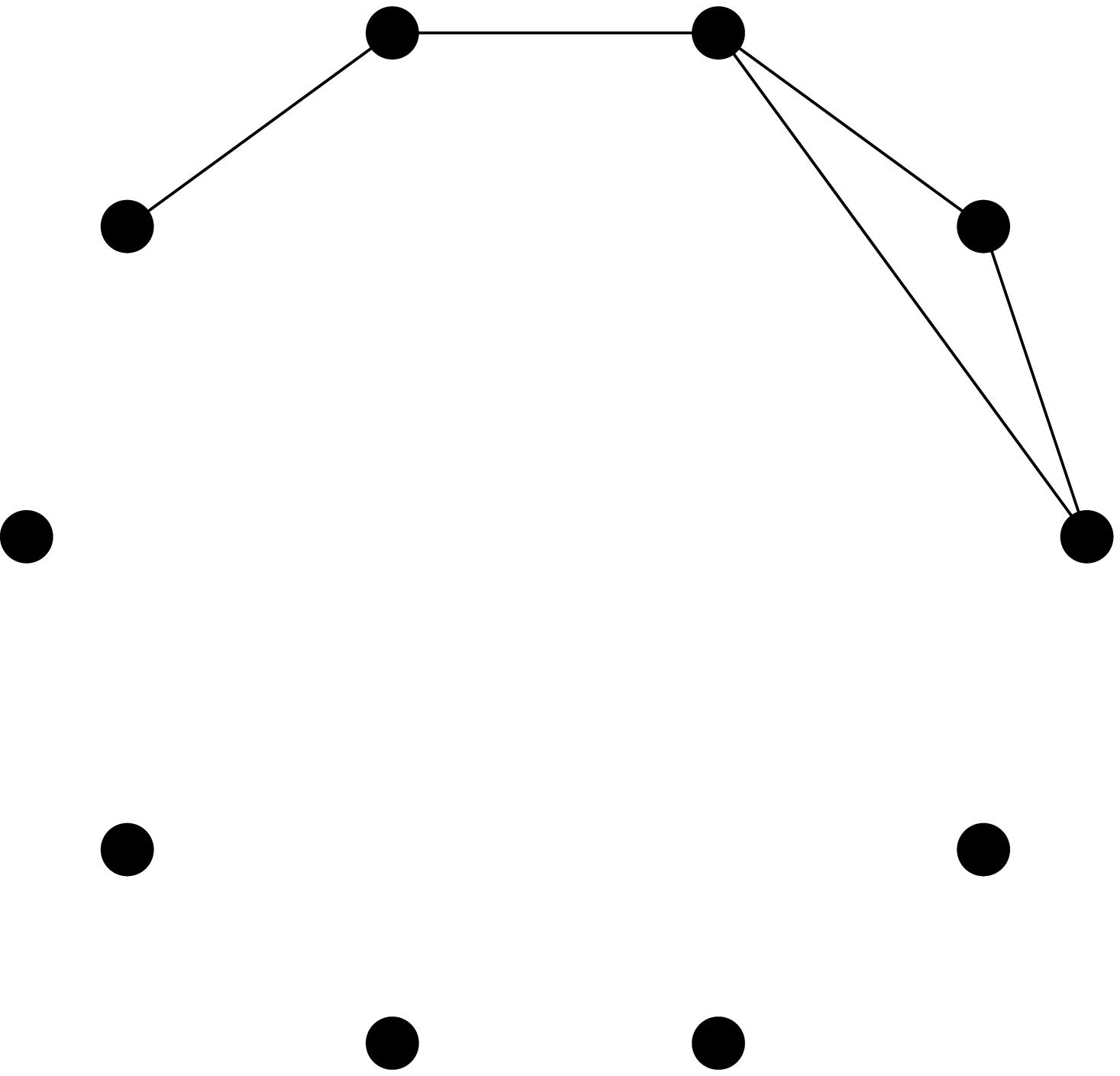}
\end{picture}\par
\end{minipage}
\begin{minipage}[t]{2.2cm}
\begin{picture}(1.4,1.8)
\leavevmode
\epsfxsize=1.4cm
\epsffile{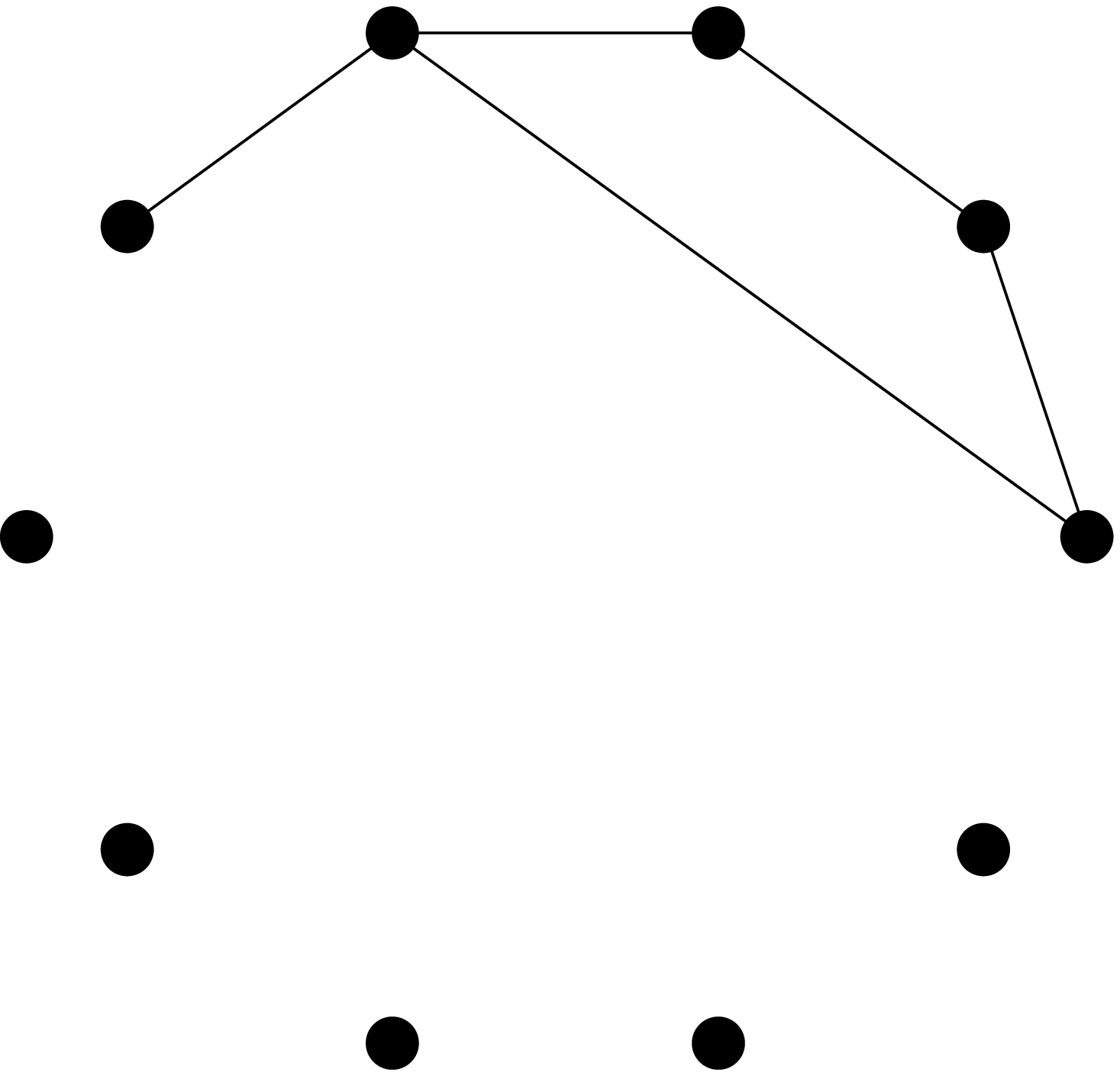}
\end{picture}\par
\end{minipage}
\begin{minipage}[t]{2.2cm}
\begin{picture}(1.4,1.8)
\leavevmode
\epsfxsize=1.4cm
\epsffile{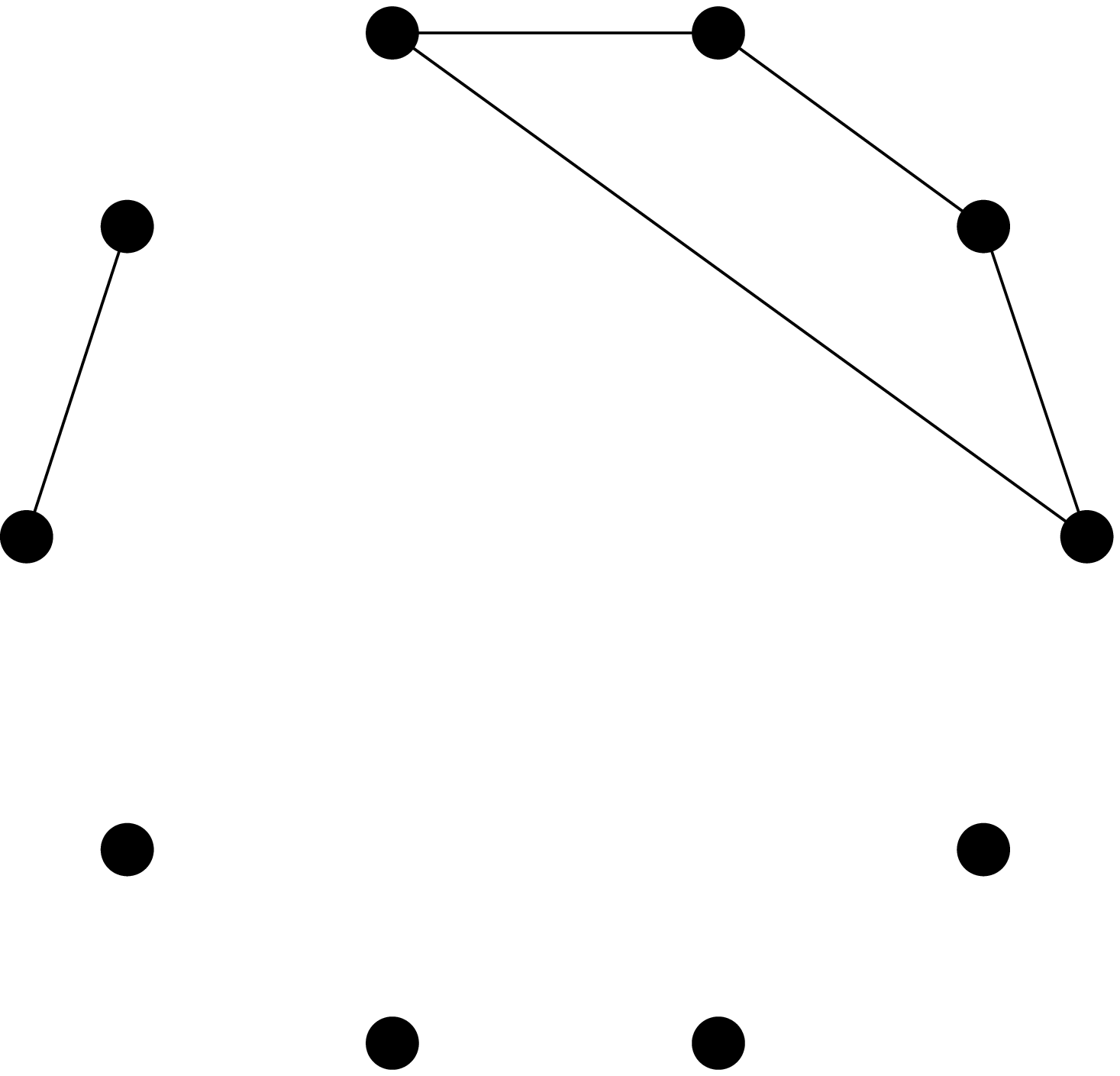}
\end{picture}\par
\end{minipage}
\begin{minipage}[t]{2.2cm}
\begin{picture}(1.4,1.8)
\leavevmode
\epsfxsize=1.4cm
\epsffile{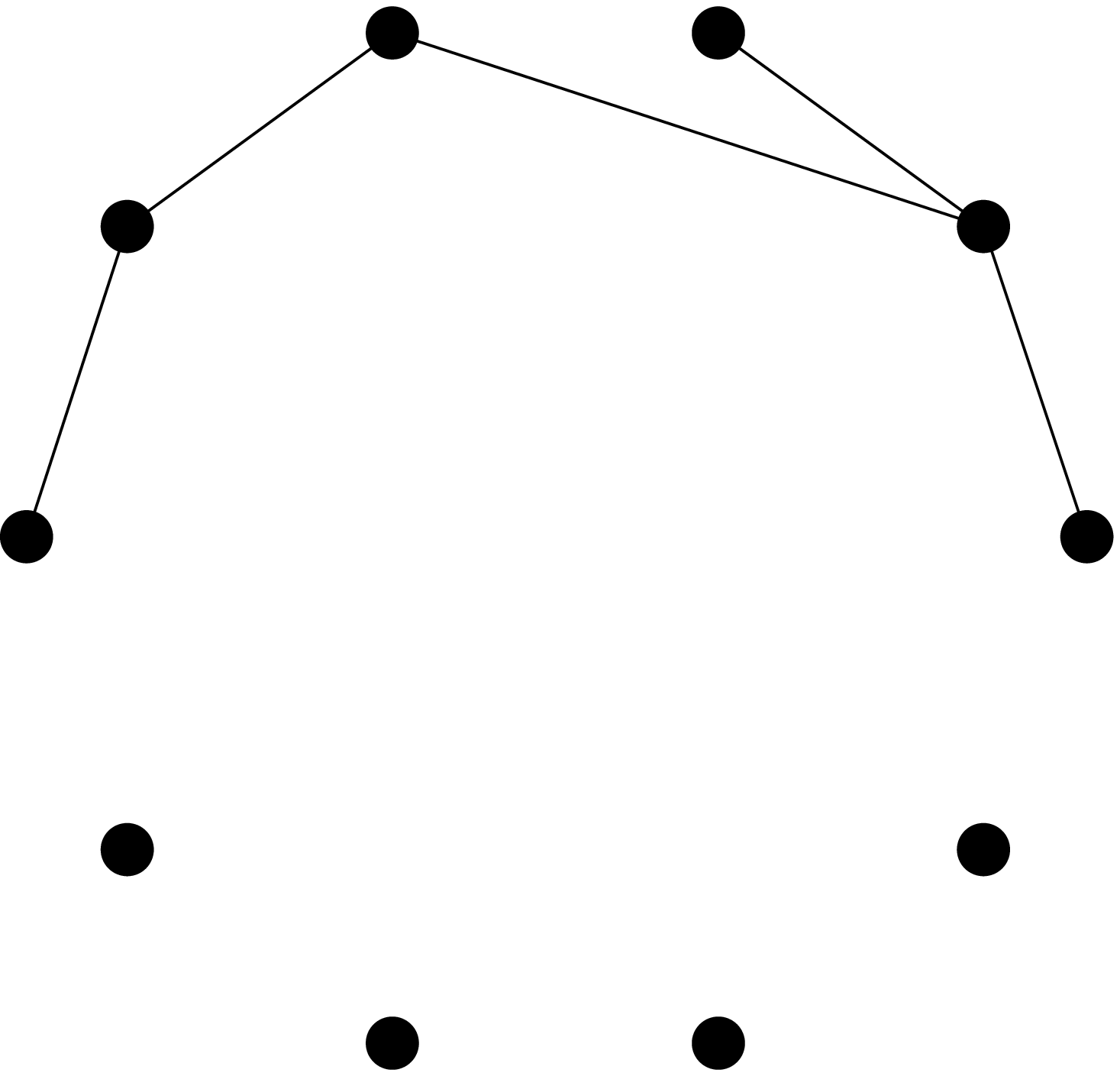}
\end{picture}\par
\end{minipage}

\bigskip
\hrule
\bigskip

$R(K_3,H) = 27$ if and only if $H^c$ is contained in one of the graphs:

\setlength{\unitlength}{1cm}
\begin{minipage}[t]{2.2cm}
\begin{picture}(1.4,1.8)
\leavevmode
\epsfxsize=1.4cm
\epsffile{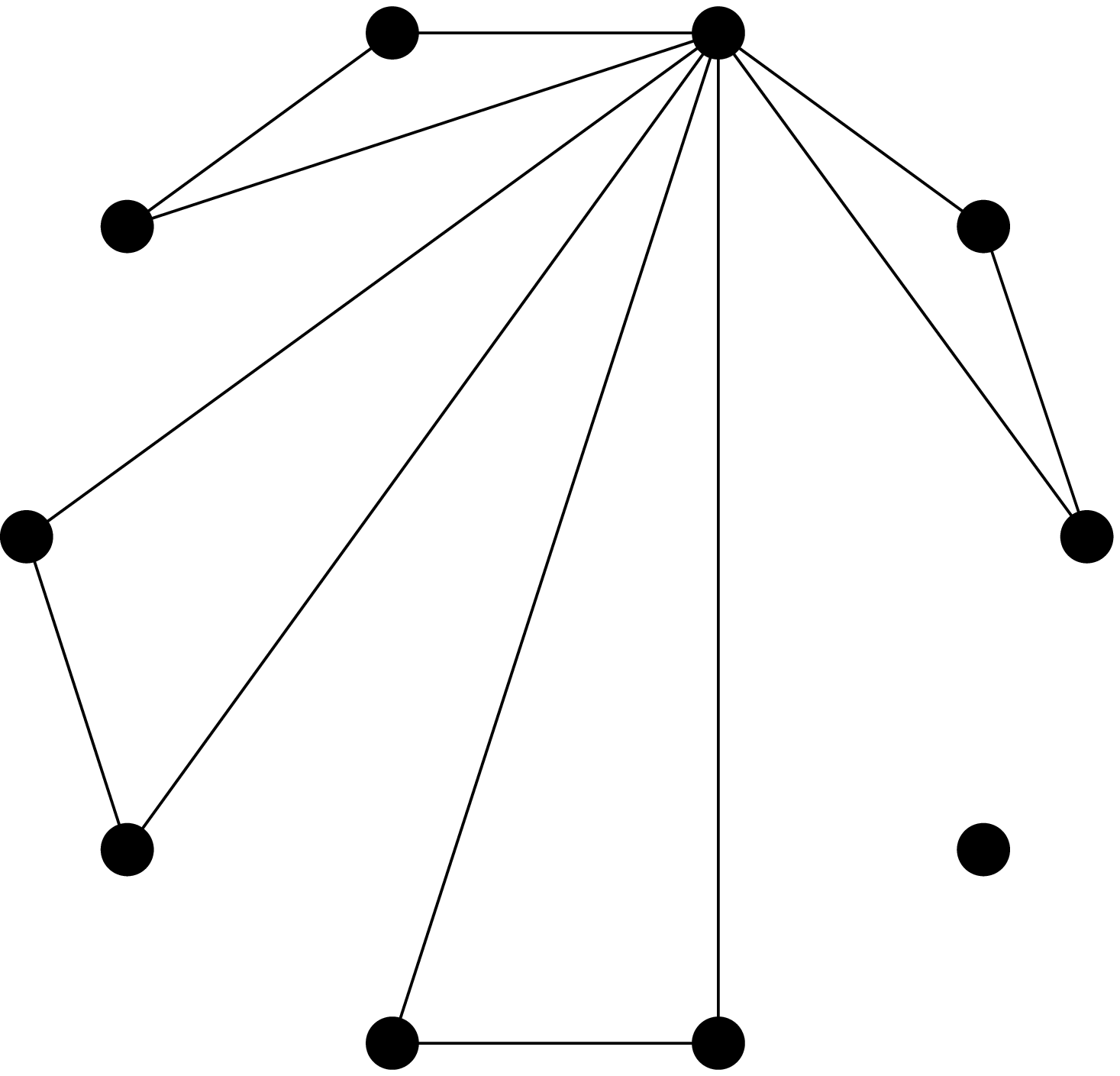}
\end{picture}\par
\end{minipage}
\begin{minipage}[t]{2.2cm}
\begin{picture}(1.4,1.8)
\leavevmode
\epsfxsize=1.4cm
\epsffile{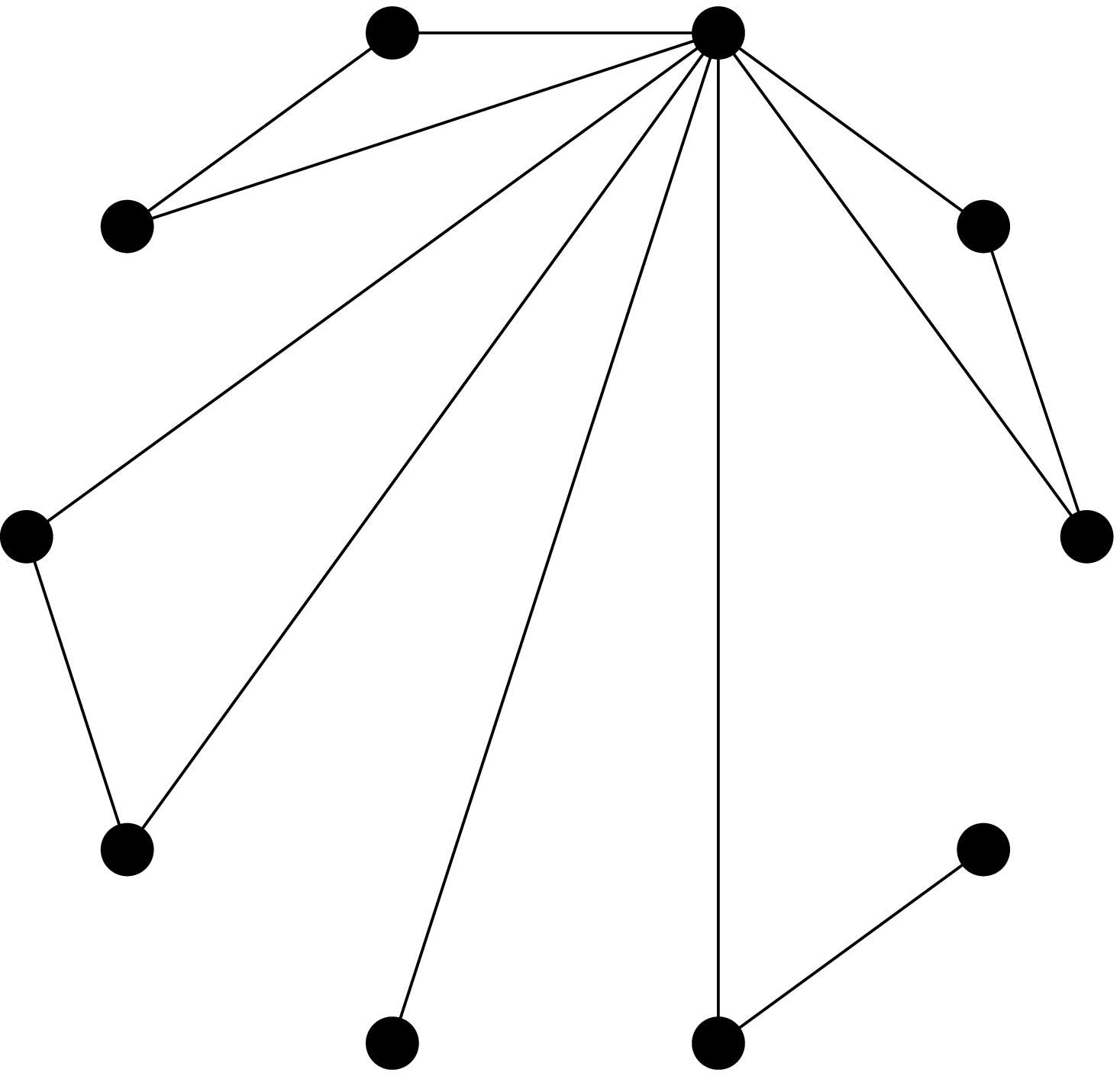}
\end{picture}\par
\end{minipage}
\begin{minipage}[t]{2.2cm}
\begin{picture}(1.4,1.8)
\leavevmode
\epsfxsize=1.4cm
\epsffile{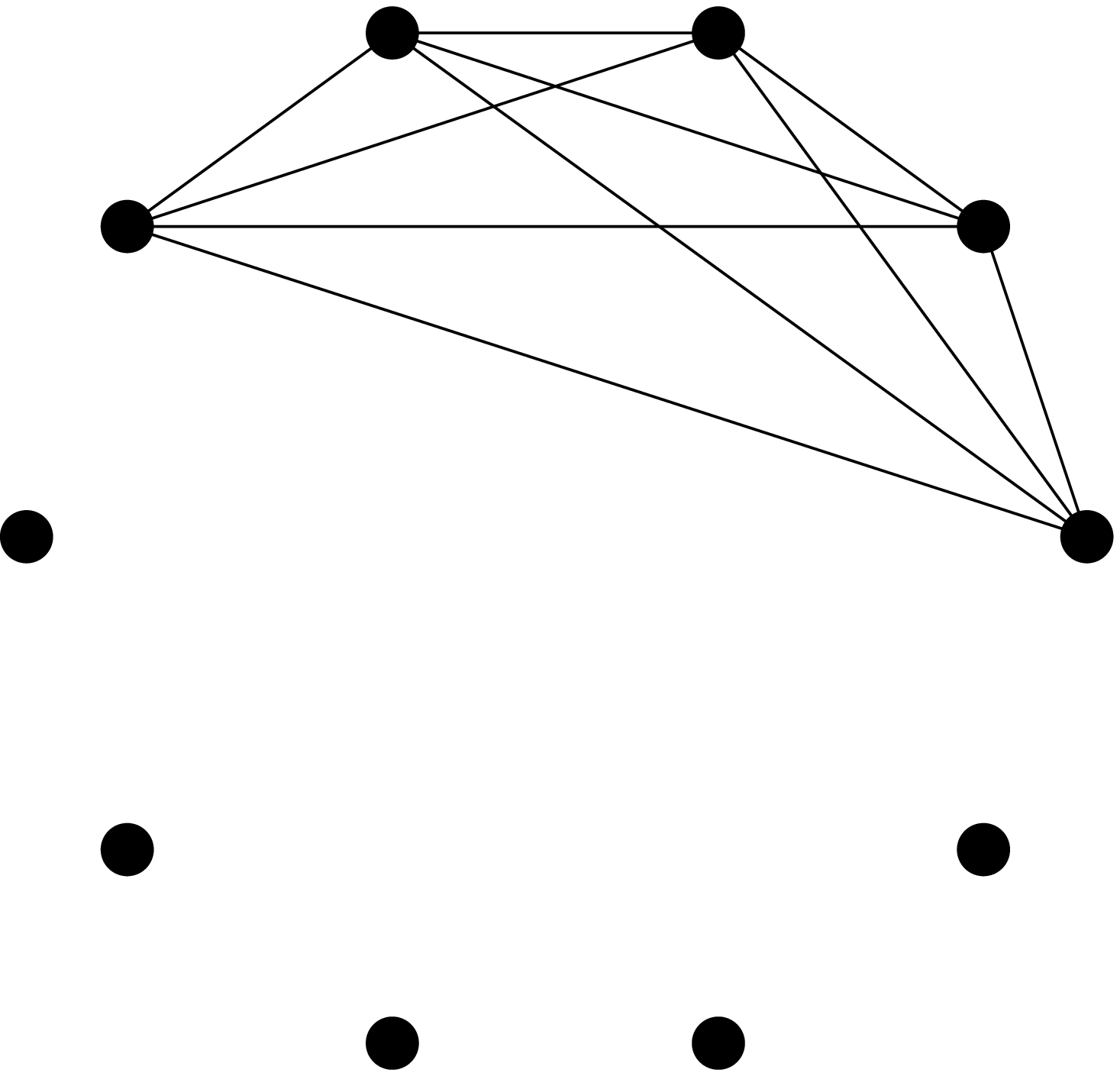}
\end{picture}\par
\end{minipage}
\begin{minipage}[t]{2.2cm}
\begin{picture}(1.4,1.8)
\leavevmode
\epsfxsize=1.4cm
\epsffile{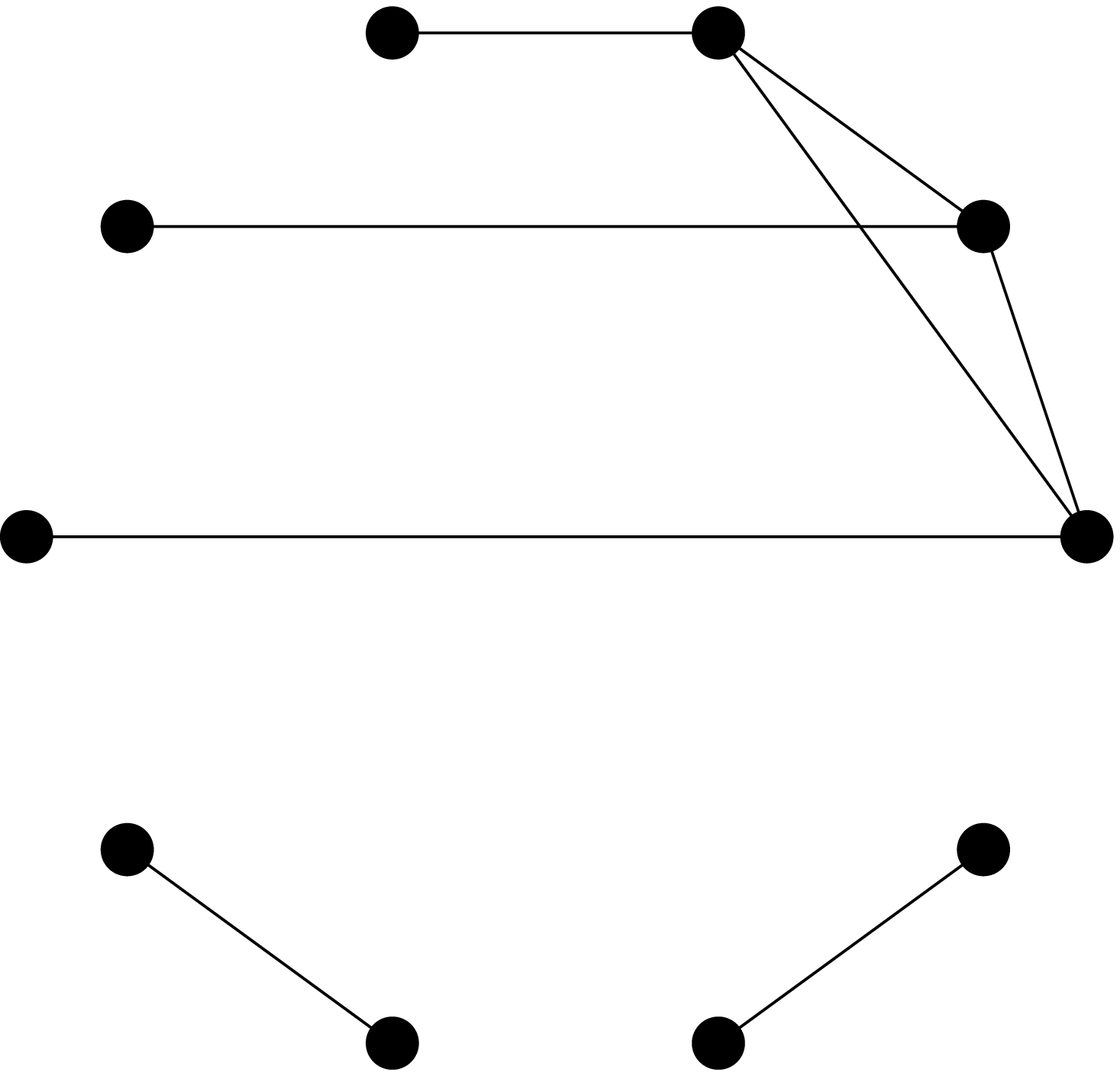}
\end{picture}\par
\end{minipage}

\bigskip
and contains one of the graphs:

\setlength{\unitlength}{1cm}
\begin{minipage}[t]{2.2cm}
\begin{picture}(1.4,1.8)
\leavevmode
\epsfxsize=1.4cm
\epsffile{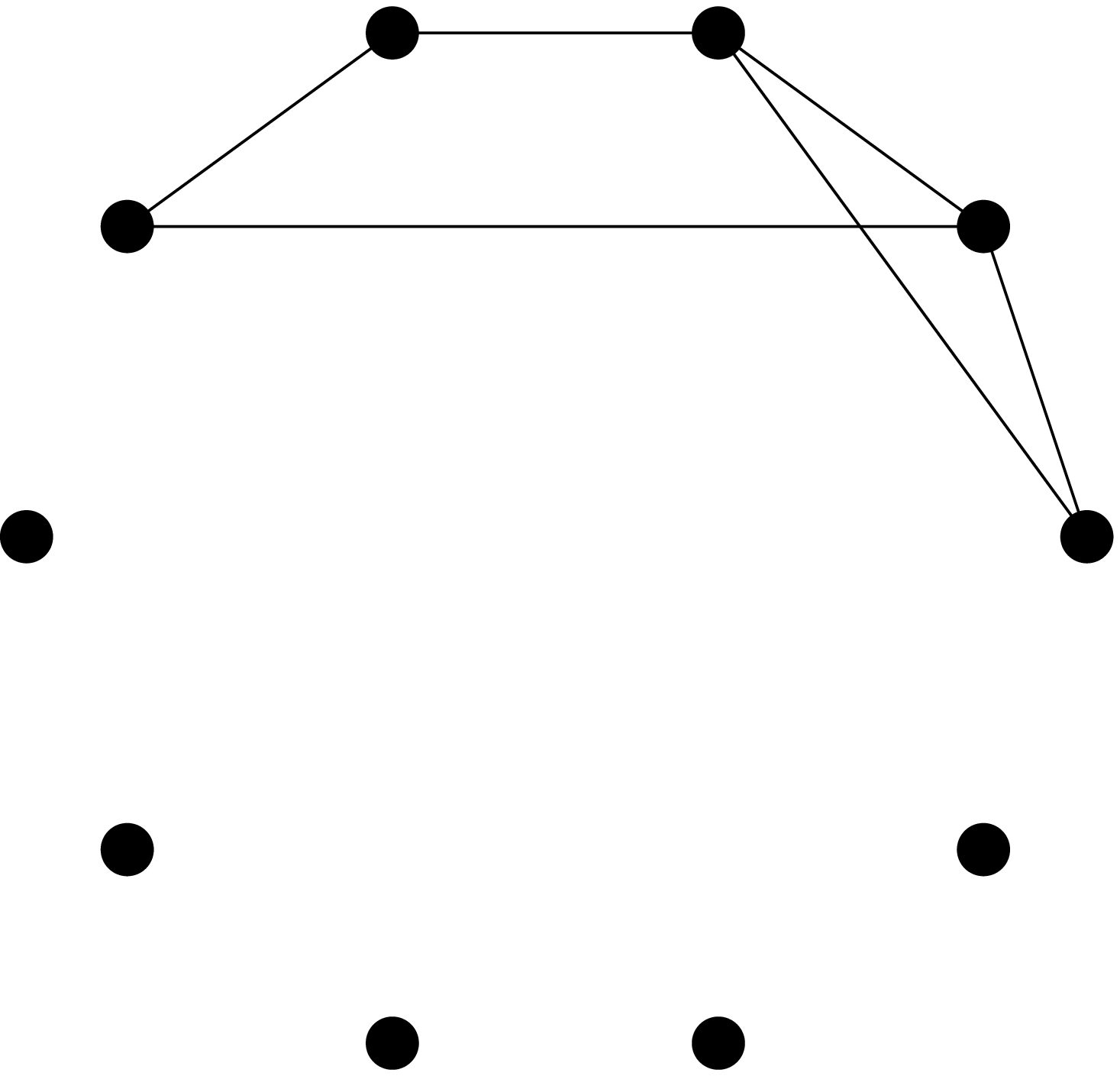}
\end{picture}\par
\end{minipage}
\begin{minipage}[t]{2.2cm}
\begin{picture}(1.4,1.8)
\leavevmode
\epsfxsize=1.4cm
\epsffile{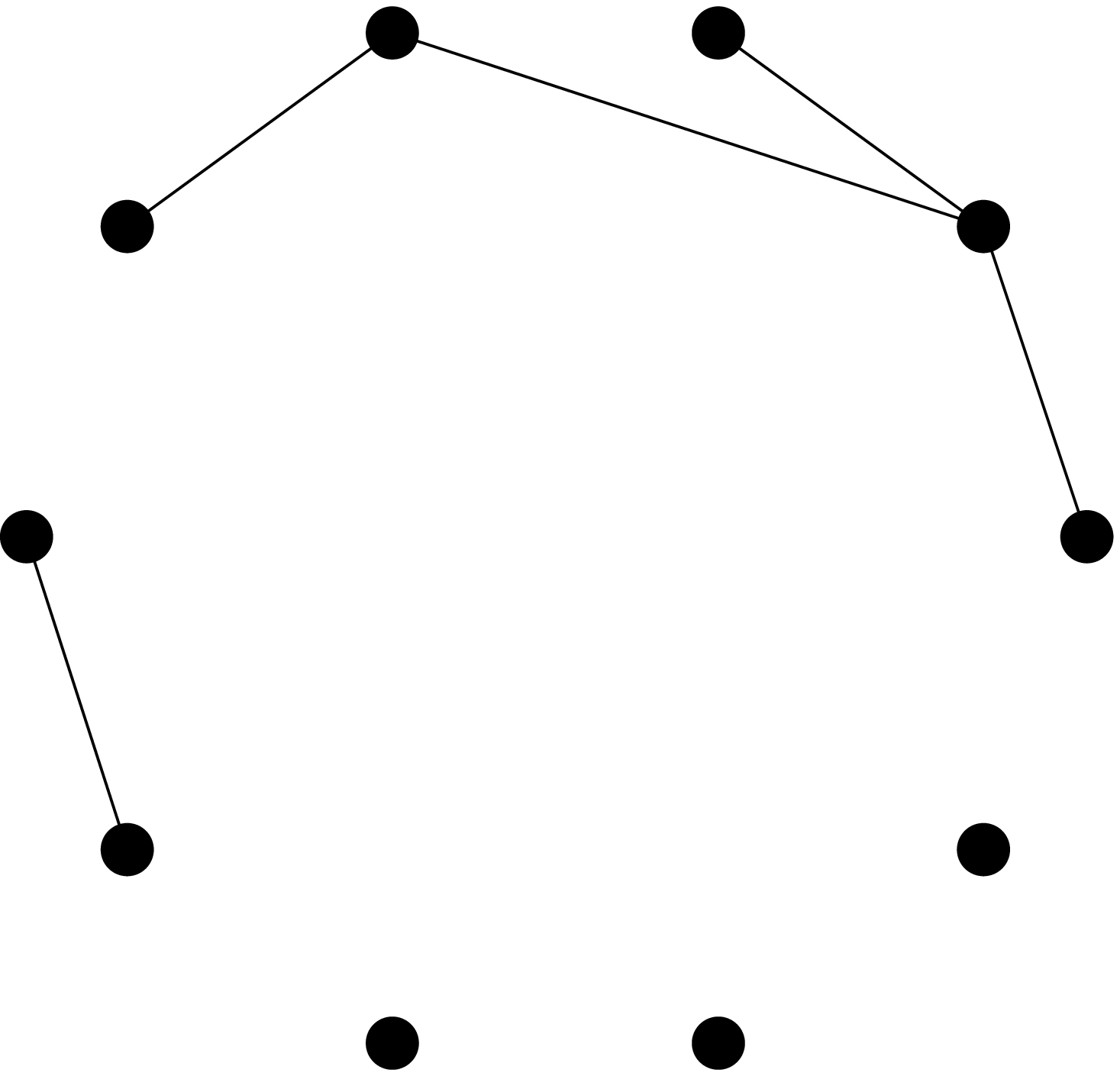}
\end{picture}\par
\end{minipage}
\begin{minipage}[t]{2.2cm}
\begin{picture}(1.4,1.8)
\leavevmode
\epsfxsize=1.4cm
\epsffile{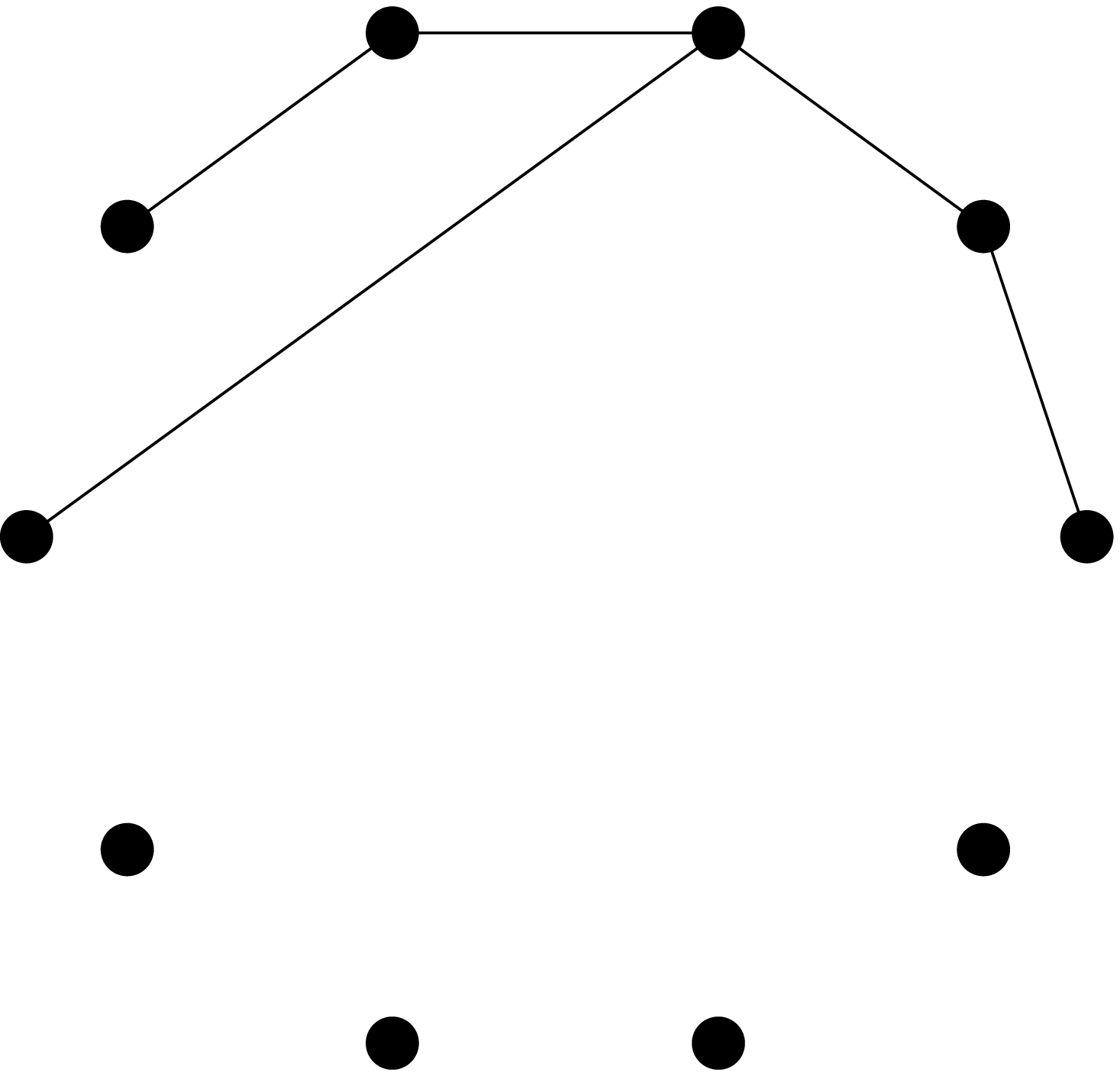}
\end{picture}\par
\end{minipage}
\begin{minipage}[t]{2.2cm}
\begin{picture}(1.4,1.8)
\leavevmode
\epsfxsize=1.4cm
\epsffile{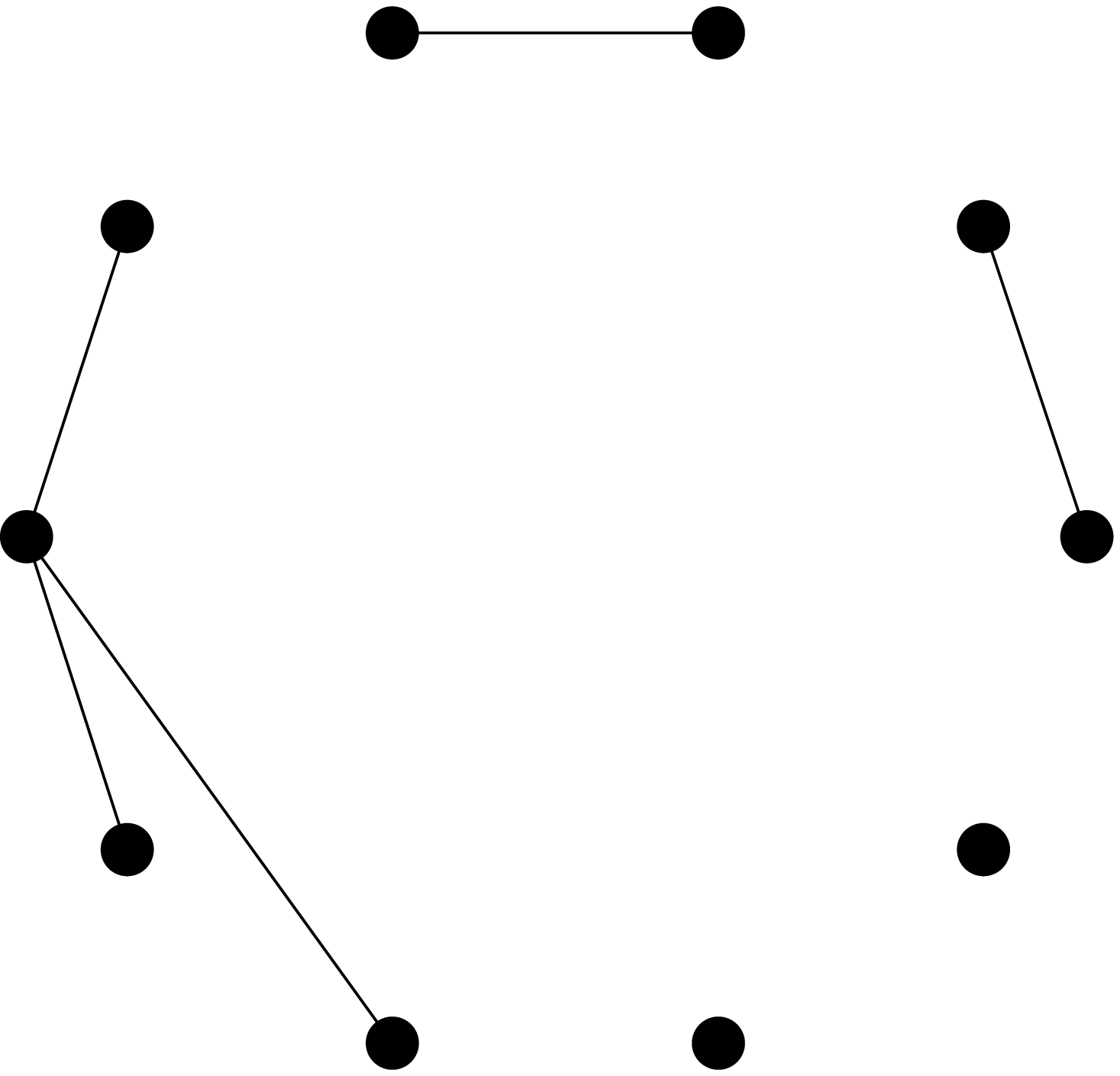}
\end{picture}\par
\end{minipage}

\bigskip
\hrule
\bigskip

The graphs with Ramsey number $R(K_3,H) < 27$ can be obtained from~\cite{ramseynumber-site}.

\bibliographystyle{plain}
\bibliography{ramsey-references.bib}

\end{document}